\documentclass[twoside,leqno,10pt, A4]{amsart}
\usepackage{amsfonts}
\usepackage{amsmath}
\usepackage{amscd}
\usepackage{amssymb}
\usepackage{amsthm}
\usepackage{amsrefs}
\usepackage{latexsym}
\usepackage{mathrsfs}
\usepackage{bbm}
\usepackage{enumerate}
\usepackage{graphicx}
\usepackage{amscd}
\usepackage{color}
\begin{document}

\newtheorem{theorem}[subsection]{Theorem}
\newtheorem{proposition}[subsection]{Proposition}
\newtheorem{lemma}[subsection]{Lemma}
\newtheorem{corollary}[subsection]{Corollary}
\newtheorem{conjecture}[subsection]{Conjecture}
\newtheorem{prop}[subsection]{Proposition}
\numberwithin{equation}{section}
\newcommand{\mr}{\ensuremath{\mathbb R}}
\newcommand{\mc}{\ensuremath{\mathbb C}}
\newcommand{\dif}{\mathrm{d}}
\newcommand{\intz}{\mathbb{Z}}
\newcommand{\ratq}{\mathbb{Q}}
\newcommand{\natn}{\mathbb{N}}
\newcommand{\comc}{\mathbb{C}}
\newcommand{\rear}{\mathbb{R}}
\newcommand{\prip}{\mathbb{P}}
\newcommand{\uph}{\mathbb{H}}
\newcommand{\fief}{\mathbb{F}}
\newcommand{\majorarc}{\mathfrak{M}}
\newcommand{\minorarc}{\mathfrak{m}}
\newcommand{\sings}{\mathfrak{S}}
\newcommand{\fA}{\ensuremath{\mathfrak A}}
\newcommand{\mn}{\ensuremath{\mathbb N}}
\newcommand{\mq}{\ensuremath{\mathbb Q}}
\newcommand{\half}{\tfrac{1}{2}}
\newcommand{\f}{f\times \chi}
\newcommand{\summ}{\mathop{{\sum}^{\star}}}
\newcommand{\chiq}{\chi \bmod q}
\newcommand{\chidb}{\chi \bmod db}
\newcommand{\chid}{\chi \bmod d}
\newcommand{\sym}{\text{sym}^2}
\newcommand{\hhalf}{\tfrac{1}{2}}
\newcommand{\sumstar}{\sideset{}{^*}\sum}
\newcommand{\sumprime}{\sideset{}{'}\sum}
\newcommand{\sumprimeprime}{\sideset{}{''}\sum}
\newcommand{\sumflat}{\sideset{}{^\flat}\sum}
\newcommand{\shortmod}{\ensuremath{\negthickspace \negthickspace \negthickspace \pmod}}
\newcommand{\V}{V\left(\frac{nm}{q^2}\right)}
\newcommand{\sumi}{\mathop{{\sum}^{\dagger}}}
\newcommand{\mz}{\ensuremath{\mathbb Z}}
\newcommand{\leg}[2]{\left(\frac{#1}{#2}\right)}
\newcommand{\muK}{\mu_{\omega}}
\newcommand{\thalf}{\tfrac12}
\newcommand{\lp}{\left(}
\newcommand{\rp}{\right)}
\newcommand{\Lam}{\Lambda_{[i]}}
\newcommand{\lam}{\lambda}
\def\L{\fracwithdelims}
\def\om{\omega}
\def\pbar{\overline{\psi}}
\def\phis{\varphi^*}
\def\lam{\lambda}
\def\lbar{\overline{\lambda}}
\newcommand\Sum{\Cal S}
\def\Lam{\Lambda}
\newcommand{\sumtt}{\underset{(d,2)=1}{{\sum}^*}}
\newcommand{\sumt}{\underset{(d,2)=1}{\sum \nolimits^{*}} \widetilde w\left( \frac dX \right) }

\newcommand{\hf}{\tfrac{1}{2}}
\newcommand{\af}{\mathfrak{a}}
\newcommand{\Wf}{\mathcal{W}}

\newcommand{\MT}{\operatorname{MT}}
\newcommand{\ET}{\operatorname{ET}}
\newcommand{\PP}{\operatorname{Prob}}
\newcommand{\sm}{\operatorname{small}}
\newcommand{\lr}{\operatorname{large}}
\newcommand{\LCM}{\operatorname{LCM}}
\newcommand{\tp}{\operatorname{top}}
\newcommand{\spn}{\operatorname{span}}
\newcommand{\Vol}{\operatorname{Vol}}
\newcommand{\new}{\operatorname{new}}
\newcommand{\old}{\operatorname{old}}
\newcommand{\GCD}{\operatorname{GCD}}

\newcommand{\bC}{\mathbb{C}}
\newcommand{\bH}{\mathbb{H}}
\newcommand{\bN}{\mathbb{N}}
\newcommand{\bQ}{\mathbb{Q}}
\newcommand{\bR}{\mathbb{R}}
\newcommand{\bS}{\mathbb{S}}
\newcommand{\bT}{\mathbb{T}}
\newcommand{\bU}{\mathbb{U}}
\newcommand{\bZ}{\mathbb{Z}}
\newcommand{\zed}{\mathbb{Z}}

\newcommand{\Fix}{\mathrm{Fix}}
\newcommand{\sgn}{\mathrm{sgn}}
\newcommand{\Res}{\mathrm{Res}}
\newcommand{\rNW}{\mathrm{NW}}
\newcommand{\rKL}{\mathrm{Kl}}
\newcommand{\PSL}{\mathrm{PSL}}
\newcommand{\GL}{\mathrm{GL}}

\newcommand{\E}{\mathbf{E}}
\newcommand{\Var}{\mathbf{Var}}
\newcommand{\one}{\mathbf{1}}

\newcommand{\vlambda}{{\mathbf{\lambda}}}
\newcommand{\vrho}{{\mathbf{\rho}}}
\newcommand{\vdelta}{{\mathbf{\delta}}}
\newcommand{\vmu}{{\mathbf{\mu}}}
\newcommand{\vx}{{\mathbf{x}}}
\newcommand{\vr}{{\mathbf{r}}}
\newcommand{\vz}{{\mathbf{z}}}
\newcommand{\ve}{{\mathbf{e}}}
\newcommand{\vk}{{\mathbf{k}}}

\newcommand{\cC}{{\mathcal{C}}}
\newcommand{\cD}{{\mathcal{D}}}
\newcommand{\cB}{{\mathcal{B}}}
\newcommand{\cF}{{\mathcal{F}}}
\newcommand{\cN}{{\mathcal{N}}}
\newcommand{\cP}{{\mathcal{P}}}
\newcommand{\cI}{{\mathcal{I}}}
\newcommand{\cS}{{\mathcal{S}}}

\newcommand{\fa}{{\mathfrak{a}}}

\newcommand{\sB}{{\mathscr{B}}}
\newcommand{\sC}{{\mathscr{C}}}
\newcommand{\sD}{{\mathscr{D}}}
\newcommand{\sF}{{\mathscr{F}}}
\newcommand{\sG}{{\mathscr{G}}}
\newcommand{\sH}{{\mathscr{H}}}
\newcommand{\sI}{{\mathscr{I}}}
\newcommand{\sK}{{\mathscr{K}}}
\newcommand{\sL}{{\mathscr{L}}}
\newcommand{\sM}{{\mathscr{M}}}
\newcommand{\sN}{{\mathscr{N}}}
\newcommand{\sO}{{\mathscr{O}}}
\newcommand{\sP}{{\mathscr{P}}}
\newcommand{\sS}{{\mathscr{S}}}

\newcommand{\ug}{\underline{g}}
\newcommand{\ul}{\underline{\ell}}
\newcommand{\un}{\underline{n}}
\newcommand{\uw}{\underline{w}}
\newcommand{\ux}{\underline{x}}

\newcommand{\Ts}{\tilde{s}}
\newcommand{\Tt}{\tilde{t}}

\theoremstyle{plain}
\newtheorem{conj}{Conjecture}
\newtheorem{remark}[subsection]{Remark}

\makeatletter
\def\widebreve{\mathpalette\wide@breve}
\def\wide@breve#1#2{\sbox\z@{$#1#2$}%
     \mathop{\vbox{\m@th\ialign{##\crcr
\kern0.08em\brevefill#1{0.8\wd\z@}\crcr\noalign{\nointerlineskip}%
                    $\hss#1#2\hss$\crcr}}}\limits}
\def\brevefill#1#2{$\m@th\sbox\tw@{$#1($}%
  \hss\resizebox{#2}{\wd\tw@}{\rotatebox[origin=c]{90}{\upshape(}}\hss$}
\makeatletter

\title[Twisted fourth moment of Dirichlet $L$-functions to a fixed modulus]{Twisted fourth moment of Dirichlet $L$-functions to a fixed modulus}

\author[P. Gao]{Peng Gao}
\address{School of Mathematical Sciences, Beihang University, Beijing 100191, China}
\email{penggao@buaa.edu.cn}

\author[L. Zhao]{Liangyi Zhao}
\address{School of Mathematics and Statistics, University of New South Wales, Sydney NSW 2052, Australia}
\email{l.zhao@unsw.edu.au}

\begin{abstract}
We evaluate the twisted four moment on the critical line of the family of Dirichlet $L$-functions to a fixed prime power modulus, obtaining an asymptotic formula with a power saving error term. 
\end{abstract}

\maketitle

\noindent {\bf Mathematics Subject Classification (2010)}: 11M06  \newline

\noindent {\bf Keywords}: twisted fourth moment, Dirichlet $L$-functions, critical line

\section{Introduction}
\label{sec 1}

Moments of families of $L$-functions is an important subject of number theory.  A classical paradigm involves the study of moments of the family of Dirichlet $L$-functions at the central point. Let $\chi$  be a primitive Dirichlet character modulo $q$ and we assume throughout the paper that $q \not \equiv 2 \pmod 4$ to ensure the existence of such primitive characters.  It was conjectured in \cite{CFKRS} by J. B. Conrey, D. W. Farmer, J. P. Keating, M. O. Rubinstein and N. C. Snaith that for all positive integral values of $k$ and explicit constants $C_k$,
\begin{align}
\label{moments}
 \sumstar_{\substack{ \chi \shortmod q }}|L(\tfrac{1}{2},\chi)|^{2k} \sim C_k \phis(q)(\log q)^{k^2},
\end{align}
   where $L(s, \chi)$ denotes the $L$-function attached to $\chi$, $\phis(q)$ the number of primitive characters modulo $q$ and $\sum^*$ indicates the sum over primitive Dirichlet characters modulo $q$.  In fact, the relation given in \eqref{moments} is believed (see \cite{R&Sound}) to hold for all real $k  \geq 0$. \newline
   
     A result of A. Selberg \cite{Selberg46} concerning a more general twisted second moment implies the validity of \eqref{moments} for $k=1$. In \cite{HB81}, D. R. Heath-Brown evaluated asymptotically the left-hand side of \eqref{moments} with $k=2$ for all most all $q$, whose result was later extended to hold for all $q$ by K. Soundararajan in \cite{Sound2007}. In \cite{Young2011}, M. P. Young further obtained an asymptotic formula for the left-hand side expression in \eqref{moments} when $k=2$ with a power saving error term for primes $q$. Subsequent improvements on the error term in Young's result can be found in  \cites{BFKMM1,BFKMM,BM15,BFKMMS}. Here we point out that the result in \cite{Young2011} was extended by V. Blomer and D. Mili\'cevi\'c \cite{BM15} to hold for all sufficiently factorable $q$ including $99.9\%$ of all admissible moduli, and was later shown by X. Wu \cite{Wu2020} to be valid for all general moduli $q$. \newline

  Often, in order to study the behaviour of the left-hand side expression of \eqref{moments}, it is natural to consider a more general situation involving with shifts of the central values. To be precise, the shifted moments corresponding to the case $k=2$ is given by
\begin{align*}
 \sumstar_{\substack{ \chi \shortmod q }} L\left(\tfrac{1}{2} + \alpha, \chi \right)L\left(\tfrac{1}{2} + \beta, \chi \right)L\left(\tfrac{1}{2} + \gamma, \overline{\chi} \right)L\left(\tfrac{1}{2} + \delta, \overline{\chi} \right),
\end{align*}
  where $\alpha$, $\beta$, $\gamma$ and $\delta$ are complex numbers. \newline
 
One of the advantages of considering these shifted moments is that, with these the various shifts, contributions from higher order poles to the main term in the eventual asymptotic formula become contributions from simple poles, thus demystifying further the structure of those main terms.  Moreover, as shown in \cites{Munsch17, Szab}, the shifted moments have significant applications towards establishing bounds for moments of the corresponding character sums. \newline

  Other than investigating shifted moments, one may consider more generally the twisted shifted moments. For example, when $k=2$,  the twisted shifted fourth moment is given by
\begin{align}
\label{twistedshiftedfourthmoments}
 \sumstar_{\substack{ \chi \shortmod q }} L\left(\tfrac{1}{2} + \alpha, \chi \right)L\left(\tfrac{1}{2} + \beta, \chi \right)L\left(\tfrac{1}{2} + \gamma, \overline{\chi} \right)L\left(\tfrac{1}{2} + \delta, \overline{\chi} \right)\chi(a)\overline{\chi}(b),
\end{align}
with $a, b \in \intz$. \newline
 
 For primes $q$,  an asymptotic evaluation on \eqref{twistedshiftedfourthmoments} was first done by B. Hough \cite[Theorem 4]{Hough2016} (the proof of which can be found in the preprint of the paper {\tt arXiv:1304.1241v3}) for square-free $a$, $b$. The result was extended by R. Zacharias \cite{Z2019} for cube-free $a$, $b$, and by D. Liu \cite[Theorem 6.2]{Liu24} for general $a$, $b$. \newline 

The importance of studying the shifted moments partly lies in their applications to evaluations on mollified moments, which in turn can be applied for instance to establish sharp bounds on the correspond moments in certain ranges, using the upper bound principle developed by  M. Radziwi{\l\l} and K. Soundararajan in \cite{Radziwill&Sound}.  In particular, as pointed out in \cite{BPRZ}, one may apply the work in \cite{Hough2016} or  \cite{Z2019} to obtain sharp upper bounds on all moments at the central point of the family of Dirichlet $L$-functions under consideration below the fourth. This is later achieved in \cite{G&Zhao25-04}, utilizing most crucially the results in \cite{Hough2016} and \cite{Liu24}. \newline

Our aim in this paper is to evaluate asymptotically the expression in \eqref{twistedshiftedfourthmoments} for prime powers $q$. To state our result, we define 
for complex numbers $\alpha$, $\beta$ and any positive integer $n$, 
\begin{align}
\label{sigmadef}
 \sigma_\alpha (n)=\sum_{d\mid n}d^{\alpha}  \quad \mbox{and} \quad \sigma_{\alpha,\beta}(n)=\sum_{ad=n}a^{-\alpha} d^{-\beta}. 
\end{align}
Note that 
\begin{align}
\label{sigmarel}
\sigma_{\alpha,\beta}(n)=n^{-\alpha}\sigma_{\alpha-\beta}(n). 
\end{align}
   We also define for any positive integer $\ell$ and any complex numbers $\alpha, \beta, \gamma, \delta$, the function
\begin{align}
\label{taudef}
 \tau_{\alpha, \beta, \gamma, \delta}(\ell) = \prod_{\substack{ p| \ell \\ p^{\nu}\|\ell}}\left(1 + \frac{p^{\gamma-\delta}(p^{(\gamma-\delta)\nu}-1) \zeta_p(2 + \alpha + \beta + \gamma + \delta)}{(p^{\gamma-\delta}-1)\zeta_p(1 +\alpha + \gamma)\zeta_p(1 + \beta + \gamma)}\right),
\end{align}
  where $\zeta_p(s) = (1-p^{-s})^{-1}$. Here and throughout the paper, we reserve the letter $p$ for a prime number. We also write $L^{(n)}(s,\chi)$ the Euler product of $L(s,\chi)$ with the primes dividing $n$ removed. In particular, this applies to the Riemann zeta function $\zeta(s)$ as well. \newline

We shall focuse on the case of even characters (i.e. those characters $\chi$ with $\chi(-1)=1$) as the case of odd characters is similar. For this, we define
\begin{align} \label{Sdef}
\mathcal{S}(\alpha, \beta, \gamma, \delta; a,b) := 
 \mathop{{\sum}^+}_{\chi \shortmod q}  L\left(\tfrac{1}{2} + \alpha, \chi \right)L\left(\tfrac{1}{2} + \beta, \chi \right)L\left(\tfrac{1}{2} + \gamma, \overline{\chi} \right)L\left(\tfrac{1}{2} + \delta, \overline{\chi} \right)  \chi(a)\overline{\chi}(b), 
\end{align}
where henceforth the symbol $\sum^+$ indicates that the summation is over all primitive even characters.  \newline

  Our main result evaluates $\mathcal{S}(\alpha, \beta, \gamma, \delta; a,b)$ asymptotically for a prime power modulus $q$, as follows.
\begin{theorem}
\label{twisted_fourth_moment_theorem}
  With the notation as above, suppose $q=q^{n_0}_0$ with $n_0 \geq 50$ and $q_0$ is an odd prime. Let $a, b$ be positive integers such that $(a, b) = (ab, q)=1$. For a set of indices $S$, set
\begin{align}
\label{eqX}
X_S = \prod_{u \in S} X_u, \quad \mbox{where} \quad X_u = \left(\frac{q}{\pi}\right)^{-u}\frac{\Gamma\left(\frac{\frac{1}{2}-u}{2}\right)}{\Gamma\left(\frac{\frac{1}{2}+u}{2}\right)}. 
\end{align}
 Suppose that
\begin{align}
\label{scondition}
\begin{split}
 (1+|\alpha|)^4(1+|\beta|)^4(1+|\gamma|)^4 (1+|\delta|)^4 (ab)^{7} \ll 
  q^{\min(1/576, 1/n_0)-\varepsilon_0} \text{ for some } 0<\varepsilon_0< \min\Big(\frac 1{576}, \frac 1{n_0}\Big).  
\end{split}
\end{align}

  Then there exists $\eta > 0$ such that for $\alpha, \beta, \gamma, \delta \in \left\{z \in \bC: \Re(z) < \eta/\log q \right\}$ with $\Im
(\beta-\alpha)=O(1)$, $\Im(\delta-\gamma)=O(1)$ and that $c\pm d \neq 0$ for any $c, d \in \{\alpha, \beta, \gamma, \delta  \}$, we have, for any $\varepsilon>0$, 
\begin{align} 
\label{4thmomenteval}
\begin{split}
  \mathcal{S}(\alpha, \beta, \gamma, \delta; a,b) =& \sum^6_{j=1}S_j   +O(q^{\varepsilon} (1+|\alpha|)^4(1+|\beta|)^4(1+|\gamma|)^4 (1+|\delta|)^4 (ab)^{7}(q^{1-1/576}+q^{1-1/n_0})). 
\end{split}
\end{align}
  where
\begin{align} 
\label{Sdef}
\begin{split}
  S_1=& \frac {\phis(q)}{2}\frac{\tau_{ \alpha, \beta, \gamma, \delta}(a)\tau_{ \gamma, \delta, \alpha, \beta}(b)}{a^{1/2 + \gamma} b^{1/2+\alpha}}\frac{\zeta(1 +\alpha + \gamma)\zeta(1 + \alpha + \delta) \zeta(1 + \beta + \gamma) \zeta(1 + \beta + \delta)}{\zeta(2 +\alpha + \beta +\gamma+\delta)}, \\
S_2=&  \frac {\phis(q)}{2}X_{\alpha,\gamma}\frac{\tau_{-\gamma, \beta, -\alpha, \delta}(a) \tau_{-\alpha, \delta, -\gamma, \beta}(b)}{a^{1/2 -\alpha} b^{1/2 -\gamma}} 
 \frac{\zeta(1-\alpha -\gamma) \zeta(1-\gamma +\delta) \zeta(1 -\alpha + \beta) \zeta(1+\beta + \delta)}{\zeta(2-\alpha + \beta -\gamma +\delta)}, \\ 
S_3= & \frac {\phis(q)}{2}X_{\beta,\gamma}\frac{\tau_{\alpha, -\gamma, -\beta, \delta}(a) \tau_{-\beta, \delta, \alpha, -\gamma}(b)}{a^{1/2 -\beta} b^{1/2 + \alpha}} 
 \frac{\zeta(1+\alpha - \beta) \zeta(1+\alpha +\delta) \zeta(1-\beta-\gamma) \zeta(1-\gamma + \delta)}{\zeta(2+\alpha - \beta -\gamma +\delta)}, \\
 S_4= &  \frac {\phis(q)}{2}X_{\alpha,\delta}\frac{\tau_{-\delta, \beta, \gamma, -\alpha}(a) \tau_{\gamma, -\alpha, -\delta, \beta}(b)}{a^{1/2 +\gamma} b^{1/2 - \delta}}  
  \frac{\zeta(1+\gamma-\delta) \zeta(1-\alpha -\delta) \zeta(1 + \beta + \gamma) \zeta(1-\alpha+\beta)}{\zeta(2-\alpha + \beta +\gamma -\delta)},\\
S_5= &  \frac {\phis(q)}{2}X_{\beta,\delta}\frac{\tau_{\alpha, -\delta, \gamma, -\beta}(a) \tau_{\gamma, -\beta, \alpha, -\delta}(b)}{a^{1/2 +\gamma} b^{1/2 + \alpha}} 
   \frac{\zeta(1+\alpha +\gamma) \zeta(1+\alpha-\beta) \zeta(1 +\gamma - \delta) \zeta(1-\beta - \delta)}{\zeta(2+\alpha - \beta +\gamma -\delta)}, \quad \mbox{and} \\
 S_6= &  \frac {\phis(q)}{2}X_{\alpha, \beta, \gamma, \delta}\frac{\tau_{ -\gamma, -\delta, -\alpha, -\beta}(a)\tau_{ -\alpha, -\beta, -\gamma, -\delta }(b)}{a^{1/2 -\alpha} b^{1/2-\gamma}}
     \frac{\zeta(1 -\alpha - \gamma)\zeta(1 - \beta - \gamma) \zeta(1 - \alpha - \delta) \zeta(1 - \beta - \delta)}{\zeta(2 -\alpha - \beta -\gamma-\delta)}.
\end{split}
\end{align}
\end{theorem}
 
  Note that the main term in \eqref{4thmomenteval} is consistent with the conjectured formula given in \cite{CFKRS}.  Our proof of Theorem \ref{twisted_fourth_moment_theorem} heavily makes use of ideas from 
\cites{BFKMM, Liu24, BM15, Z2019} and most importantly those from \cite{BM15}. To be more precise, we apply the approximate functional equation for products of $L$-functions (c.f. Lemma \ref{lemafe}) to represent $\mathcal{S}(\alpha, \beta, \gamma, \delta; a,b)$  as convergent series over $m, n$. A main term arises from the diagonal contribution $ma = nb$ and is treated in Section \ref{sec: mainterm}. For the remaining sum, we apply partitions of unity to localize the sums over $m, n$ to $m \asymp M, n \asymp N$. Without loss of generality, we may assume that $N \geq M$. Then we treat the sums depending on the relative sizes of $M$ and $N$. When $M$ and $N$ are far apart, then we use crucially the ideas from \cite{BM15}, applying the Voronoi summation formula (see Lemma \ref{Voronoi} below) to the longer sum over $n$ to convert it to a shorter sum and then apply \cite[Theorem 5]{BM15} (see Lemma \ref{largesieveKS} below) to control the size of divisor-like sums in arithmetic progressions to show that the sums in this case only give rise to an error term.  Now, for $M$ and $N$ in close proximity, we follow the treatments in \cites{Liu24, Z2019} to use the $\delta$-method introduced by W. Duke, J. B. Friedlander and H. Iwaniec in \cite{DFI94} and then apply the Voronoi summation formula to both sums over $m$ and $n$ to identify a secondary main term together with other terms that will contribute to an error term.  The computation of the secondary main term follows from those given in \cite[Section 6.4.3]{Liu24} while the estimation on the remaining terms follows closely the proof of \cite[Theorem 4.2]{Z2019} by employing spectral theory of automorphic forms, in particular the Kuznetsov trace formula (see Lemma \ref{Kuznetsov} below).

\section{Preliminaries}
\label{sec 2}

\subsection{Various relations}
Let $\mu$ denote the M\"obius function and $\varphi$ the Euler totient function. The following orthogonality formula (see \cite{Sound2007}) evaluates sums over even Dirichlet characters.
\begin{lemma}
\label{lemof}
For $(mn,q)=1$, we have
\begin{align}
\label{lemof2}
\mathop{{\sum}^+}_{\chi \shortmod q}\chi(m)\overline{\chi}(n) =\frac12\left(\sum_{d\mid(q,m-n)}\varphi(d)\mu\left(\frac qd\right) +\sum_{d\mid(q,m+n)}\varphi(d)\mu\left(\frac qd \right)\right).
\end{align}
\end{lemma}

  Recall the definition of $\sigma_{\alpha,\beta}(n)$ in \eqref{sigmadef}. This next lemma gives some multiplicative identities for this arithmetic function.
\begin{lemma}
\label{sigmarelation}
 For any $\alpha, \beta \in \mc$ and any positive integers $m,n$, we have
\begin{align}
\label{sigmaprodrelation}
 \sigma_{\alpha,\beta}(n)\sigma_{\alpha,\beta}(m)= &\sum_{d|(n,m)}d^{-\alpha-\beta}\sigma_{\alpha,\beta}\Big( \frac {nm}{d^2} \Big) \quad \mbox{and} \\
\label{inversesigmaprodrelation}
  \sigma_{\alpha,\beta}(nm)=& \sum_{d|(n,m)}\mu(d)d^{-\alpha-\beta}\sigma_{\alpha,\beta}\Big(\frac {n}{d}\Big)\sigma_{\alpha,\beta}\Big(\frac {m}{d}\Big).
\end{align}
\end{lemma}
\begin{proof}
Both sides of \eqref{sigmaprodrelation} are multiplicative. Thus, it suffices to show that it holds for prime power, i.e., for any prime $p$ and any non-negative integers $u,v$,
\begin{align*}
 \sigma_{\alpha,\beta}(p^u)\sigma_{\alpha,\beta}(p^v)= &\sum_{d|p^{\min(u,v)}}d^{-\alpha-\beta}\sigma_{\alpha,\beta}\Big(\frac {p^{u+v}}{d^2}\Big),
\end{align*}
which can be verified in a straightforward way.  Now \eqref{inversesigmaprodrelation} follows from \eqref{sigmaprodrelation} and M\"obius inversion. This completes the proof of the lemma.
\end{proof}

   Let $\omega(n)$ denote the number of primes dividing $n$ and $d(n)$ the divisor function of $n$.  We have the following well-known estimates (see \cite[Theorems 2.10, 2.11]{MVa1}) that for $n \geq 3$,
\begin{align}
\label{omegandnbound}
\omega(n), \ \log d(n) \ll \frac {\log n}{\log \log n}. 
\end{align}

For $\alpha, \beta, \gamma, \delta \in \left\{z \in \bC: \Re(z) < \eta/\log q \right\}$ and that for $n, m \leq q^{O(1)}$,
\begin{align}
\label{sigmabound}
  \sigma_{\alpha,\beta}(n) \ll d(n) \ll n^{\varepsilon} \quad \mbox{and} \quad \sigma_{\gamma,\delta}(m) \ll d(m) \ll m^{\varepsilon}. 
\end{align}

\subsection{Approximate functional equation}
 We cite the following approximate functional equation concerning products of Dirichlet $L$-functions.
\begin{lemma}
\label{lemafe}
Let $\mathcal{G}(s)$ be an even, entire function of exponential decay in any strip $|\Re(s)|<C$ satisfying $G(0)=1$, and
\begin{align}
\label{defV}
V_{\alpha,\beta,\gamma,\delta}(x)=\frac1{2\pi i}\int\limits_{(2)}\frac{\mathcal{G}(s)}{s}g_{\alpha,\beta,\gamma,\delta}(s)x^{-s} \dif s, \; \mbox{where} \; g_{\alpha,\beta,\gamma,\delta}(s)=\pi^{-2s}\frac{\Gamma\left(\frac{ \frac12+\alpha+s}2\right) \Gamma\left(\frac{ \frac12+\beta+s}2\right) \Gamma\left(\frac{ \frac12+\gamma+s}2\right) \Gamma\left(\frac{ \frac12+\delta+s}2\right)}{\Gamma\left(\frac{\frac12+\alpha}2\right) \Gamma\left(\frac{\frac12+\beta}2\right) \Gamma\left(\frac{\frac12+\gamma}2\right) \Gamma\left(\frac{\frac12+\delta}2\right)}.
\end{align}
Define
\begin{align*}
\widetilde{V}_{\alpha,\beta,\gamma,\delta}(x)= X_{-\gamma,-\delta,-\alpha,-\beta}V_{\alpha,\beta,\gamma,\delta}(x),
\end{align*}
 with $X_{*,*,*,*}$ given in \eqref{eqX}. Then for any even primitive Dirichlet character $\chi$ modulo $q$, we have 
\begin{align}
\label{approxfcneqnL4}
\begin{split}
 L & \left(\tfrac12+\alpha,  \chi\right) L\left(\tfrac12+\beta,\chi\right) L\left(\tfrac12+\gamma,\overline{\chi}\right) L\left(\tfrac12+\delta,\overline{\chi}\right) \\
& \hspace*{1cm}  =  \sum_{m,n}\frac{\sigma_{\alpha,\beta}(m)\sigma_{\gamma,\delta}(n) \chi(m)\overline{\chi}(n)} {(mn)^{\frac12}}V_{\alpha,\beta,\gamma,\delta}\left(\frac{mn}{q^2}\right) + \sum_{m,n}\frac{\sigma_{-\gamma,-\delta}(m)\sigma_{-\alpha,-\beta}(n) \chi(m)\overline{\chi}(n)} {(mn)^{\frac12}}\widetilde{V}_{-\gamma,-\delta,-\alpha,-\beta}\left(\frac{mn}{q^2}\right).
\end{split}
\end{align}
  Moreover, for any $c>0$ and all integers $j \geq 0$,
\begin{align}
\label{W}
 V^{(j)}_{\alpha,\beta,\gamma,\delta}(x), \; \widetilde{V}^{(j)}_{-\gamma,-\delta,-\alpha,-\beta}(x)  \ll_c \min( 1 , (|\alpha|+1)^{c/2}(|\beta|+1)^{c/2}(|\gamma|+1)^{c/2}(|\delta|+1)^{c/2}x^{-c}).
\end{align}
\end{lemma}
\begin{proof}
  The proof of \eqref{approxfcneqnL4} can be found in \cite[Proposition 2.4]{Young2011}. By Stirling's formula \cite[(5.112)]{iwakow}, we see that
for $u, s \in \mc$ such that $\Re(u) \ll 1/\log q, \Im(u) \ll q^A$ for any real number $A>0$,
\begin{align*}
  \frac{\Gamma\left(u+s \right)}{ \Gamma\left(s \right) } \ll \frac {|s+u|^{\Re(s+u)-1/2}}{|s|^{\Re(s)-1/2}}\exp \Big(\frac {\pi}{2}(|s|-|s+u|)\Big) \ll (|s|+4)^{\Re(u)}\exp\Big(\frac {\pi}{2}|u|\Big) \quad \mbox{and} \quad \frac{\Gamma\left(\frac{1/2-u}{2}\right)}{\Gamma\left(\frac{1/2+u}{2}\right)} \ll_A 1. 
\end{align*}
  We apply the above to $g_{\alpha,\beta,\gamma,\delta}(s)$, arriving at
\begin{align} \label{gest}
  g_{\alpha,\beta,\gamma,\delta}(s) \ll (2\pi)^{-2\Re(s)}\Big ((|\alpha|+4)\cdot (|\beta|+4)\cdot (|\gamma|+4)\cdot (|\delta|+4) \Big )^{\frac {\Re(s)}2}\exp\Big(\pi|s|\Big) \quad \mbox{and} \quad
  X_{-\gamma,-\delta,-\alpha,-\beta} \ll_A 1.
\end{align}
Arguments simlilar to those in the proof of \cite[Propsition 5.4]{iwakow} lead to \eqref{W}, completing the proof of the lemma.
\end{proof}

  In what follows, we shall choose $\mathcal{G}$ without further mention to be of the form that $\mathcal{G}(s)=P_{\alpha,\beta,\gamma,\delta}(s)\exp(s^2)$, where $P_{\alpha,\beta,\gamma,\delta}(s)$ is an even polynomial in $s$ such that it is rational, symmetric, and even in the shifts $\alpha$, $\beta$, $\gamma$, $\delta$.  Moreover, we shall take  $P_{\alpha,\beta,\gamma,\delta}(s)$ to satisfy $P_{\alpha,\beta,\gamma,\delta}(0)=1$, $P_{\alpha,\beta,\gamma,\delta}\left(-\frac{\alpha+\gamma}2\right)= P_{\alpha,\beta,\gamma,\delta}\left(\frac12\pm\alpha\right)=0$ (which also implies by symmetry that as well as $P_{\alpha,\beta,\gamma,\delta}\left(\frac{\beta+\delta}2\right)=P_{\alpha,\beta,\gamma,\delta}\left(\frac12\pm\beta\right)=0$, etc.).

\subsection{Voronoi summation and Bessel functions}

  We quote the following Voronoi summation formula from \cite[Theorem 6.11]{Liu24}.
\begin{lemma}
\label{Voronoi}   
  Let $c$ be a positive integer and $a$ an integer co-prime to $c$, and $W$ a smooth function compactly supported on $(0,\infty)$. Then
 \begin{align*}
 \label{eq-voronoi-d}
\begin{split}
 & \sum_{n\geq 1} \sigma_{{\alpha_1, \alpha_2}}(n)W(n)e\Bigl(\frac{an}{c}\Bigr)= 
\frac{1}{c}  \int\limits_0^{\infty}\Big (x^{-\alpha_2}\frac {\zeta(1+\alpha_1-\alpha_2)}{c^{\alpha_1-\alpha_2}}+x^{-\alpha_1}\frac {\zeta(1+\alpha_2-\alpha_1)}{c^{\alpha_2-\alpha_1}}\Big )W(x) \dif x 
+  \frac{1}{c} \Big ( \mathcal{E}_J+\mathcal{E}_Y+\mathcal{E}_K \Big ),
  \end{split}
\end{align*} 
   where 
\begin{align*}
 \begin{split}
 \mathcal{E}_J=& -2\pi \sin\Big (\frac {\pi}{2}(\alpha_2-\alpha_1)\Big )\sum_{n\geq 1}\frac {\sigma_{\alpha_1-\alpha_2}(n)}{n^{(\alpha_1-\alpha_2)/2}}e\Bigl(-\frac{\overline{a}n}{c}\Bigr)\widetilde W_{J;\alpha_1, \alpha_2} \Bigl(\frac{n}{c^2}\Bigr), \\
\mathcal{E}_Y=& -2\pi \cos\Big (\frac {\pi}{2}(\alpha_2-\alpha_1)\Big )\sum_{n\geq 1}\frac {\sigma_{\alpha_1-\alpha_2}(n)}{n^{(\alpha_1-\alpha_2)/2}}e\Bigl(-\frac{\overline{a}n}{c}\Bigr)\widetilde W_{Y;\alpha_1, \alpha_2} \Bigl(\frac{n}{c^2}\Bigr), \quad \mbox{and} \\
\mathcal{E}_K=& 4 \sin\Big (\frac {\pi}{2}(\alpha_2-\alpha_1)\Big )\sum_{n\geq 1}\frac {\sigma_{\alpha_1-\alpha_2}(n)}{n^{(\alpha_1-\alpha_2)/2}}e\Bigl(-\frac{\overline{a}n}{c}\Bigr)\widetilde W_{K;\alpha_1, \alpha_2} \Bigl(\frac{n}{c^2}\Bigr). \\
\end{split}
\end{align*} 
Here, the transforms $
\widetilde W_{\mathcal{J};\alpha_1, \alpha_2}:(0,\infty)\to\mathbb{C}$ of $W$ for $\mathcal{J} \in \{J, Y, K \}$ are defined by
\begin{align*}
 \begin{split}
  \widetilde W_{\mathcal{J};\alpha_1, \alpha_2}(y)  =
   \displaystyle{ \int\limits_0^\infty u^{-(\alpha_1+\alpha_2)/2}W(u)\mathcal{J}_{\alpha_2-\alpha_1}(4\pi \sqrt{ uy}) \dif u},
 \end{split}
\end{align*} 
with $J_{\alpha}$, $Y_{\alpha}$ and $K_{\alpha}$ being classical Bessel functions.
\end{lemma}
 
From \cite[Lemma 2.4]{BFKMM}, we have the following result on the decay properties of the Bessel transforms. 
\begin{lemma}
\label{besseldecay}  Suppose that $W$ is a smooth function compactly supported in $[1/2,2]$ and for some $Q>0$, 
\begin{equation*}
 W^{(j)}(x)\ll_{j, \varepsilon} Q^{j}.
\end{equation*}
Then for any integers $i$, $j\geq 0$ and all real $y>0$, $\varepsilon>0$, 
\begin{displaymath}
\begin{split}
  y^j\Big (\displaystyle{ \int\limits_0^\infty W(u)\mathcal{J}_{\alpha}(4\pi \sqrt{ uy}) \dif
  u}\Big )^{(j)}(y) &\ll_{i,j,\varepsilon} (1+y)^{j/2}\big(1 +   y^{1/2}  Q^{-1}\big)^{-i}.
\end{split}
\end{displaymath}
In particular, the function $y \mapsto \int_0^\infty W(u)\mathcal{J}_{\alpha}(4\pi \sqrt{ uy}) \dif
  u$ decays rapidly when $y\gg Q^{2+\varepsilon}$.
\end{lemma}

\subsection{Automorphic Forms}
\label{SectionPreliminary}
 
  We include in this section some results (mostly taken from \cite[Section 2.2]{Z2019}) from the theory of automorphic forms necessary in the proof of Theorem \ref{twisted_fourth_moment_theorem}. Let $\chi$ be a Dirichlet character modulo $\ell$. We set $\kappa=\frac{1-\chi(-1)}{2}\in\{0,1\}$ and let $k\geq 2$ with $k\equiv \kappa$ (mod $2$).  Let $\mathcal{B}_k(\ell,\chi)$ (resp. $\mathcal{B}(\ell,\chi)$) denote a Hecke basis of the Hilbert space of holomorphic cusp forms of weight $k$ (resp. of Maa\ss \ cusp forms of weight $\kappa$) with respect to the Hecke congruence group $\Gamma_0(\ell)$ and with nebentypus $\chi$. We also use the Eisenstein series $E_{j}(\cdot,1/2+it)$ for a basis of the continuous spectrum, where $j$ runs over the set of parameters of the form
\begin{equation*}
\{ (\chi_1,\chi_2,f) \ | \ \chi_1\chi_2=\chi , \ f\in\mathcal{B}(\chi_1,\chi_2) \},
\end{equation*}
where $(\chi_1,\chi_2)$ ranges over the pairs of characters of modulus $\ell$ such that $\chi_1\chi_2=\chi$ and $\mathcal{B}(\chi_1,\chi_2)$ is some finite set depending on $(\chi_1,\chi_2)$. \newline
 
  For a Hecke eigenform $f$, denote its eigenvalues by $\lambda_f(n)$ for all $(n,\ell)=1$. The Fourier expansion of $f$ at a singular cusp $\mathfrak{a}$ can be written as ($z=x+iy$) :
$$f_{|^k\sigma_{\mathfrak{a}}}(z)=\sum_{n\geqslant 1}\rho_{f,\mathfrak{a}}(n)(4\pi n)^{k/2}e(nz) \ \ \mathrm{for} \ f\in\mathcal{B}_k(\ell,\chi),$$
and
$$f_{|_{\kappa}\sigma_{\mathfrak{a}}}(z)=\sum_{n\neq 0}\rho_{f,\mathfrak{a}}(n)W_{\frac{n}{|n|}\frac{\kappa}{2}it_f}(4\pi |n|y)e(nx) \ \ \mathrm{for } \ f\in\mathcal{B}(\ell,\chi),$$
where $\sigma_{\mathfrak{a}}$ is the scaling matrix of $\mathfrak{a}$, $W_{\frac{n}{|n|}\frac{\kappa}{2}it_f}$ is the Whittaker function and $t_f$ is the spectral parameter of $f$ such that $\lambda_f = 1/4+t_f^2$ with $\lambda_f$ being the eigenvalue for the hyperbolic Laplace operator. Moreover, the two slash operators $|^k\gamma$ and $|_\kappa \gamma$ of weights $k$ and $\kappa$ for any $\gamma=\left(\begin{smallmatrix} a & b \\ c & d \end{smallmatrix}\right)\in\mathrm{SL}_2(\mathbb{R})$ are defined by
$$f_{|^k\gamma}(z):= (cz+d)^{-k}f(\gamma z) \quad \mbox{and} \quad f_{|_\kappa \gamma}(z):= \left(\frac{cz+d}{|cz+d|}\right)^{-\kappa}f(\gamma z).$$

  We also write for an Eisenstein series $E_{\chi_1,\chi_2,f}(z,\tfrac12+it)$, 
\begin{alignat*}{1}
E_{\chi_1,\chi_2,f|_\kappa \sigma_{\mathfrak{a}}}(z, \tfrac12+it)=  \ c_{1,f,\mathfrak{a}}(t)y^{1/2+it}+c_{2,f,\mathfrak{a}}y^{1/2-it}   +  \sum_{n\neq 0}\rho_{f,\mathfrak{a}}(n,t)W_{\frac{n}{|n|}\frac{\kappa}{2}it_f}(4\pi |n|y)e(nx).
\end{alignat*}
  
  Denote $\rho_{f,\mathfrak{a}}$ by $\rho_{f}$ for $\mathfrak{a}=\infty$. Then for $(m,\ell)=1$ and $n\geqslant 1$, we have the multiplicative relations
\begin{align}
\label{RelationHecke1}
\lambda_{f}(m)\sqrt{n}\rho_{f}(n)= \sum_{d|(m,n)}\sqrt{\frac{mn}{d^2}}\rho_f\left(\frac{mn}{d^2}\right) \quad \mbox{and} \quad 
\sqrt{mn}\rho_f(mn)= \sum_{d|(m,n)}\mu(d)\rho_f\left(\frac{n}{d}\right)\sqrt{\frac{n}{d}}\lambda_f\left(\frac{m}{d}\right).
\end{align}
If $f$ is either a holomorphic cusp form or an Eisentein series, then
\begin{equation*}
|\lambda_f (n)|\leq d(n). 
\end{equation*}
For a Maa\ss \ cusp form $f$, from a result of H. H. Kim and P. Sarnak \cite{KimH03},
\begin{equation}\label{BoundHeckeEigenvalue2}
|\lambda_f(n)|\leqslant d(n)n^\theta , \ \ \theta=\frac{7}{64}.
\end{equation}
  Moreover,
\begin{equation}\label{BoundSpectralParamater}
|\Im m (t_f)|\leqslant \theta.
\end{equation}

\subsubsection{Kuznetsov trace formula}

 Let $\Phi : [0,\infty) \rightarrow \mathbb{C}$ be a smooth function such that $\Phi(0)=\Phi'(0)=0$ and $\Phi^{(j)}(x)\ll (1+x)^{-3}$ for $0\leqslant j\leqslant 3$. We define for $\kappa\in\{0,1\}$ the following three integral transforms. 
\begin{equation}\label{definitionBesselTransform}
\begin{split}
\dot{\Phi}(k) := & \ 4i^k\int\limits_0^\infty \Phi(x)J_{k-1}(x)\frac{\dif x}{x}, \quad \widehat{\Phi} (t) :=  \ \frac{2\pi i t^{\kappa}}{\sinh(\pi t)} \int\limits_0^\infty(J_{2it}(x)-(-1)^\kappa J_{-2it})\Phi(x)\frac{\dif x}{x} \quad \mbox{and} \\
\check{\Phi}(t):= & \ 8i^{-\kappa}\int\limits_0^\infty \Phi(x)\cosh (\pi t)K_{2it}(x)\frac{\dif x}{x}.
\end{split}
\end{equation}
 
  We have the following Kuznetsov trace formula from \cite[Proposition 2.3]{Z2019}. 
\begin{lemma}
\label{Kuznetsov} With the notation as above, let $\mathfrak{a},\mathfrak{b}$ two singular cusps for the congruence group $\Gamma_0(\ell)$ and let $a,b>0$ be integers. Then we have
\begin{align*}
\begin{split}
\sum_{\gamma}^{\Gamma_0(\ell)}\frac{1}{\gamma}S^{\chi}_{\mathfrak{a}\mathfrak{b}}(a,b;\gamma)\phi \left(\frac{4\pi \sqrt{ab}}{\gamma}\right) =  \sum_{\substack{k\geqslant 2 \\ k\equiv\kappa \ (2)}} & \sum_{f\in\mathcal{B}_k(\ell,\chi)}\dot{\phi}(k)\Gamma(k)\sqrt{ab}\overline{\rho_{f,\mathfrak{a}}}(a)\rho_{f,\mathfrak{b}}(b)  +  \sum_{f\in\mathcal{B}(\ell,\chi)}\hat{\phi}(t_f)\frac{\sqrt{ab}}{\cos(\pi t_f)}\overline{\rho_{f,\mathfrak{a}}}(a)\rho_{f,\mathfrak{b}}(b) \\ 
+ & \frac{1}{4\pi}\mathop{\sum\sum}_{\substack{\chi_1\chi_2=\chi \\ f\in\mathcal{B}(\chi_1,\chi_2)}}\int\limits_{\mathbb{R}}\hat{\phi}(t)\frac{\sqrt{ab}}{\cosh(\pi t)}\overline{\rho_{f,\mathfrak{a}}}(a,t)\rho_{f,\mathfrak{b}}(b,t) \dif t, 
\end{split}
\end{align*}
and 
\begin{align*}
\begin{split}
\sum_{\gamma}^{\Gamma_0(\ell)}\frac{1}{\gamma}S^{\chi}_{\mathfrak{a}\mathfrak{b}}(a,-b;\gamma)\phi \left(\frac{4\pi \sqrt{ab}}{\gamma}\right)  =   \sum_{f\in\mathcal{B}(\ell,\chi)} & \check{\phi}(t_f)\frac{\sqrt{ab}}{\cos(\pi t_f)}\overline{\rho_{f,\mathfrak{a}}}(a)\rho_{f,\mathfrak{b}}(-b) \\ + & \frac{1}{4\pi}\mathop{\sum\sum}_{\substack{\chi_1\chi_2=\chi \\ f\in\mathcal{B}(\chi_1,\chi_2)}}\int\limits_{\mathbb{R}}\check{\phi}(t)\frac{\sqrt{ab}}{\cosh(\pi t)}\overline{\rho_{f,\mathfrak{a}}}(a,t)\rho_{f,\mathfrak{b}}(-b,t) \dif t, 
\end{split}
\end{align*}
where $S^\chi_{\mathfrak{ab}}(n,m;\gamma)$ is the generalized twisted Kloosterman sum and defined by 
$$S^\chi_{\mathfrak{ab}}(n,m;\gamma):=\sum_{\left(\begin{smallmatrix} \alpha & \beta \\ \gamma & \delta \end{smallmatrix}\right) \in B\setminus \sigma_{\mathfrak{a}}^{-1}\Gamma_0(\ell)\sigma_{\mathfrak{b}}/ B}\overline{\chi}\left(\sigma_{\mathfrak{a}}\left(\begin{matrix}\alpha & \beta \\ \gamma & \delta   \end{matrix}\right)\sigma_{\mathfrak{b}}^{-1}\right)e\left(\frac{n\alpha+m\delta}{\gamma}\right).$$
The notation $\sum_\gamma^{\Gamma_0(\ell)}$ means that the sum runs over all positive $\gamma$ such that $S^\chi_{\mathfrak{ab}}(n,m;\gamma)$ is not empty.
\end{lemma}

\subsection{Large sieve inequalities}

   We include in this section some large sieve inequalities. We begin with one for Kloosterman sums. For integers $f, g$ and $r \geq 1$, the Kloosterman sum is defined to be
\begin{align}
\label{Kloosterman}
  S(f,g, r)=\sumstar_{\substack{ u \shortmod r }} e \Big(\frac {fu+g\overline u}{r}\Big), \quad \mbox{with} \quad u \overline{u} \equiv 1 \pmod{r}.
\end{align}
  
The following large sieve inequality for Kloosterman sums is a slight variation of given \cite[Theorem 5]{BM15} and its proof follows from a straightforward modification of that of the same.
\begin{lemma}
\label{largesieveKS} 
 Let $r,s, q \in \mn$ and $u \in \mz$ such that
\begin{equation*}
 s \mid r \mid q, \ (r/s, 2)=1, \ (u,r)=1.
\end{equation*}
    Let $M,K \geq 1$. Let $\lambda: \mn \rightarrow \mc$ be an arithmetic function with $|\lambda(n)| \ll 1$.Then
\begin{displaymath}
\begin{split}
 \sum_{\substack{M \leq m \leq 2M \\ (m,q)=1 }}\Big |\sum_{\substack{K \leq k \leq 2K }}\lambda(k)S(um,k,r) \Big |^2 \ll (qKM)^{\varepsilon}
\Big (MKrs+K^2M(rs)^{1/2}+\frac {K^2r^{3/2}}{s^{1/2}} \Big ).
\end{split}
\end{displaymath}
\end{lemma}

Next, we cite some spectral large sieve inequalities. Adapting the notation in Section \ref{SectionPreliminary}, let $\ell_0$ be the conductor of $\chi$ and recall that each cusp (not necessarily singular) $\mathfrak{a}$ for $\Gamma_0(\ell)$ is equivalent to a fraction of the form $u/v$, where $v\geqslant 1$, $v|\ell$ and $(v,u)=1$. We then define 
\begin{equation*}
\nu(\mathfrak{a}):= \ell^{-1}\left(v,\frac{\ell}{v}\right).
\end{equation*}
Also, we set for a sequence a complex numbers $(a_n)$,  
\begin{alignat*}{1}
\| a_n \|_N^2 := & \ \sum_{N<n\leqslant 2N}|a_n|^2, \quad \Sigma^{(\mathrm{H})}(k,f,N):=  \ \sqrt{(k-1)!}\sum_{N<n\leqslant 2N}a_n \rho_{f,\mathfrak{a}}(n)\sqrt{n}, \\
\Sigma_{\pm}^{(\mathrm{M})}(f,N):= & \ \frac{(1+|t_f|)^{\pm \kappa/2}}{\sqrt{\cosh(\pi t_f)}}\sum_{N<n\leqslant 2N}a_n\rho_{f,\mathfrak{a}}(\pm n)\sqrt{n} \quad \mbox{and} \quad \Sigma^{(\mathrm{E})}_{\pm}(f,t,N):=  \ \frac{(1+|t|)^{\pm \kappa/2}}{\sqrt{\cosh(\pi t)}}\sum_{N<n\leqslant 2N}a_n\rho_{f,\mathfrak{a}}(\pm n,t)\sqrt{n}.
\end{alignat*}
Then we have the following spectral large sieve results from \cite[Proposition 2.4]{Z2019}.
\begin{lemma}
\label{TheoremSpectralLargeSieve}
 With the notation as in Section \ref{SectionPreliminary}, let $\mathfrak{a}$ be a singular cusp for the group $\Gamma_0(\ell)$.  Then, for real numbers $T\geqslant 1$ and $N\geqslant 1/2$ and any $(a_n)$ a sequence of complex numbers, 
\begin{alignat*}{1}
\sum_{\substack{2\leqslant k\leqslant T \\ k\equiv \kappa \ (2)}}\sum_{f\in\mathcal{B}_k(\ell,\chi)}\left|\Sigma^{(\mathrm{H})}(k,f,N)\right|^2\ll & \ \left(T^2+\ell_0^{1/2}\nu(\mathfrak{a})N^{1+\varepsilon}\right) \| a_n \|_N^2, \\
\sum_{|t_f|\leqslant T}\left|\Sigma_\pm^{(\mathrm{M})}(f,N)\right|^2 \ll & \ \left(T^2+\ell_0^{1/2}\nu(\mathfrak{a})N^{1+\varepsilon}\right) \| a_n \|_N^2, \\
\sum_{\substack{\chi_1\chi_2=\chi \\ f\in\mathcal{B}(\chi_1,\chi_2)}}\int_{-T}^T\left|\Sigma_{\pm}^{(\mathrm{E})}(f,t,N)\right|^2 dt \ll & \ \left(T^2+\ell_0^{1/2}\nu(\mathfrak{a})N^{1+\varepsilon}\right) \| a_n \|_N^2,
\end{alignat*}
 where all the implied constants depend on $\varepsilon$ only.
\end{lemma}

\section{Initial Treatments}
\label{sec 3}

\subsection{Setup}
    We apply the approximate functional equation \eqref{approxfcneqnL4} and the orthogonality relation \eqref{lemof2} to \eqref{Sdef}, getting 
\begin{align}
\label{1}
\begin{split}
  \mathcal{S}(\alpha, \beta, \gamma, \delta; a,b) = & \sum_{m,n} \frac{\sigma_{\alpha,\beta}(m)\sigma_{\gamma,\delta}(n)} {\sqrt{mn}}V_{\alpha,\beta,\gamma,\delta}\left(\frac{mn}{q^2}\right)\mathop{{\sum}^+}_{\chi \shortmod q} \chi(ma)\overline{\chi}(nb) \\
& \hspace*{1cm} + \sum_{m,n}\frac{\sigma_{-\gamma,-\delta}(m)\sigma_{-\alpha,-\beta}(n)} {\sqrt{mn}}\widetilde{V}_{-\gamma,-\delta,-\alpha,-\beta}\left(\frac{mn}{q^2}\right)\mathop{{\sum}^+}_{\chi \shortmod q} \chi(ma)\overline{\chi}(nb) \\
=& \frac 12 \sum_{d \mid q} \varphi(d)\mu\left(\frac{q}{d}\right) \sum_{\substack{m, n \\ ma  \pm nb \equiv 0 \shortmod d\\ (mn, q) = 1}}\frac{\sigma_{\alpha,\beta}(m)\sigma_{\gamma,\delta}(n)} {\sqrt{mn}}V_{\alpha,\beta,\gamma,\delta}\left(\frac{mn}{q^2}\right)\\
& \hspace*{1cm} + \frac 12 \sum_{d \mid q} \varphi(d)\mu\left(\frac{q}{d}\right) \sum_{\substack{m, n \\ ma  \pm nb \equiv 0 \shortmod d\\ (mn, q) = 1}} \frac{\sigma_{-\gamma,-\delta}(m)\sigma_{-\alpha,-\beta}(n)} {\sqrt{mn}}\widetilde{V}_{-\gamma,-\delta,-\alpha,-\beta}\left(\frac{mn}{q^2}\right).
\end{split}
\end{align}

Applying M\"obius inversion to remove the condition $(mn, q)=1$ (or equivalently $(n, q)=(m,q)=1$) reveals that
\begin{align}
\label{removeqexp}
\begin{split}
 \sum_{\substack{m, n \\ ma  \pm nb \equiv 0 \shortmod d\\ (mn, q) = 1}} & \frac{\sigma_{\alpha,\beta}(m)\sigma_{\gamma,\delta}(n)} {\sqrt{mn}}V_{\alpha,\beta,\gamma,\delta}\left(\frac{mn}{q^2}\right) \\
   =& \sum_{f_1, f_2 \mid q} \frac{\mu(f_1)\mu(f_2)}{\sqrt{f_1f_2}}  \sum_{\substack{m, n \\ f_1ma  \pm f_2nb \equiv 0 \shortmod d}}\frac{\sigma_{\alpha,\beta}(f_1m)\sigma_{\gamma,\delta}(f_2n)} {\sqrt{mn}}V_{\alpha,\beta,\gamma,\delta}\left(\frac{f_1f_2mn}{q^2}\right) \\
=& \sum_{\substack{m, n \\ ma  \pm nb \equiv 0 \shortmod d}}\frac{\sigma_{\alpha,\beta}(m)\sigma_{\gamma,\delta}(n)} {\sqrt{mn}}V_{\alpha,\beta,\gamma,\delta}\left(\frac{mn}{q^2}\right) \\
& +O\Big ( \Big| \sum_{\substack{f_1, f_2 \mid q \\ q_0 \mid \max (f_1, f_2)}} \frac{\mu(f_1)\mu(f_2)}{\sqrt{f_1f_2}} \sum_{\substack{m, n \\ f_1ma  \pm f_2nb \equiv 0 \shortmod d}}\frac{\sigma_{\alpha,\beta}(f_1m)\sigma_{\gamma,\delta}(f_2n)} {\sqrt{mn}}V_{\alpha,\beta,\gamma,\delta}\left(\frac{f_1f_2mn}{q^2}\right) \Big| \Big ).
\end{split}
\end{align}

    In view of \eqref{W}, we may assume that $f_1f_2nm \leq \big ((1+|\alpha|)(1+|\beta|)(1+|\gamma|)(1+|\delta|)\big)^{1/2+\varepsilon}q^{2+\varepsilon}$  with a negligible error in the last sum above. It follows from this \eqref{sigmabound} and \eqref{W}  that
\begin{align}
\label{removeqexpest}
\begin{split}
 \sum_{\substack{f_1, f_2 \mid q \\ q_0 \mid \max (f_1, f_2)}} & \frac{\mu(f_1)\mu(f_2)}{\sqrt{f_1f_2}}  \sum_{\substack{m, n \\ f_1ma  \pm f_2nb \equiv 0 \shortmod d}}\frac{\sigma_{\alpha,\beta}(f_1m)\sigma_{\gamma,\delta}(f_2n)} {\sqrt{mn}}V_{\alpha,\beta,\gamma,\delta}\left(\frac{f_1f_2mn}{q^2}\right)\\
\ll & q^{\varepsilon}\sum_{\substack{f_1, f_2 \mid q \\ q_0 \mid \max (f_1, f_2)}} \frac{1}{\sqrt{f_1f_2}}  \sum_{\substack{m, n \\ f_1ma  \pm f_2nb \equiv 0 \shortmod d\\ f_1f_2nm \leq \big ((1+|\alpha|)(1+|\beta|)(1+|\gamma|)(1+|\delta|)\big)^{1/2+\varepsilon}q^{2+\varepsilon}}}\frac{1} {\sqrt{mn}}.
\end{split}
\end{align}

  Note that when $d=1$, the last expression above is
\begin{align}
\label{removeqexpestd1}
\begin{split}
 \ll q^{\varepsilon}\sum_{\substack{f_1, f_2 \mid q \\ q_0 \mid \max (f_1, f_2)}} \frac{1}{\sqrt{f_1f_2}}  \sum_{\substack{n  \leq \big ((1+|\alpha|)(1+|\beta|)(1+|\gamma|)(1+|\delta|)\big)^{1/2+\varepsilon}q^{2+\varepsilon}/f_1f_2}}\frac{d(n)} {\sqrt{n}} \ll q^{1+\varepsilon}q^{-3/2}_0.
\end{split}
\end{align}

 We may thus assume that $q_0|d$. We first consider the terms $f_1ma=f_2nb$ from the last expression in \eqref{removeqexpest}. As $(a,b)=1$, this implies that $a\frac {f_1}{(f_1,f_2)}|n$, $b\frac {f_2}{(f_1,f_2)}|m$. We may thus replace $m, n$ by $b\frac {f_2}{(f_1,f_2)}m$, $a\frac {f_1}{(f_1,f_2)}n$ respectively, mindful of $\alpha, \beta, \gamma, \delta \in \left\{z \in \bC: \Re(z) < \eta/\log q \right\}$ and \eqref{scondition}.  We get that the contribution from these terms is
\begin{align*}
\begin{split}
\ll & \frac{q^{\varepsilon}}{\sqrt{ab}} \sum_{\substack{f_1, f_2 \mid q \\ q_0 \mid \max (f_1, f_2)}} \frac{(f_1,f_2)}{f_1f_2} \sum_{\substack{ n \ll q^{O(1)}}}\frac{1} {n} \ll q^{\varepsilon}(ab)^{-1/2}q^{-1}_0.
\end{split}
\end{align*}

In the case that $f_1ma \neq f_2nb$, without loss of generality, we may assume that $f_1ma > f_2nb$. The condition $f_1ma  - f_2nb \equiv 0 \pmod{d}$ then implies that $f_1ma = f_2nb +dr$ for some integer $r>0$.  Now $f_1f_2nm \leq \big ((1+|\alpha|)(1+|\beta|)(1+|\gamma|)(1+|\delta|)\big)^{1/2+\varepsilon}q^{2+\varepsilon}$ further implies that $r \leq f_1ma/d \leq \big ((1+|\alpha|)(1+|\beta|)(1+|\gamma|)(1+|\delta|)\big)^{1/2+\varepsilon}q^{2+\varepsilon}a/f_2nd$. It follows that
\begin{align}
\label{removeqexpestoffdiag}
\begin{split}
 q^{\varepsilon} & \sum_{\substack{f_1, f_2 \mid q \\ q_0 \mid \max (f_1, f_2)}} \frac{1}{\sqrt{f_1f_2}}  \sum_{\substack{m, n \\ f_1ma  -f_2nb \equiv 0 \shortmod d\\ f_1ma  >f_2nb \\ f_1f_2nm \leq \big ((1+|\alpha|)(1+|\beta|)(1+|\gamma|)(1+|\delta|)\big)^{1/2+\varepsilon}q^{2+\varepsilon}}}\frac{1} {\sqrt{mn}} \\
\ll & q^{\varepsilon} \sqrt{ab} \sum_{\substack{f_1, f_2 \mid q \\ q_0 \mid \max (f_1, f_2)}} \ \sum_{\substack{n \\ (f_2nb)^2 \leq ab \big ((1+|\alpha|)(1+|\beta|)(1+|\gamma|)(1+|\delta|)\big)^{1/2+\varepsilon}q^{2+\varepsilon}}}\frac{1} {\sqrt{f_2nb}} \\
& \hspace*{1cm} \times \sum_{\substack{r \leq f_1ma/d \leq \big ((1+|\alpha|)(1+|\beta|)(1+|\gamma|)(1+|\delta|)\big)^{1/2+\varepsilon}q^{2+\varepsilon}a/f_2nd}}\frac{1} {\sqrt{f_2nb +dr}} \\
\ll & q^{\varepsilon}\sqrt{ab} \sum_{\substack{f_1, f_2 \mid q \\ q_0 \mid \max (f_1, f_2)}} \ \sum_{\substack{n \\ (f_2nb)^2 \leq ab \big ((1+|\alpha|)(1+|\beta|)(1+|\gamma|)(1+|\delta|)\big)^{1/2+\varepsilon}q^{2+\varepsilon}}}\frac{1} {\sqrt{f_2nb}} \\
& \hspace*{1cm} \times \sum_{\substack{r  \leq \big ((1+|\alpha|)(1+|\beta|)(1+|\gamma|)(1+|\delta|)\big)^{1/2+\varepsilon}q^{2+\varepsilon}a/f_2nd}}\frac{1} {\sqrt{dr}} \\
\ll & \big ((1+|\alpha|)(1+|\beta|)(1+|\gamma|)(1+|\delta|)\big)^{1/4+\varepsilon}q^{1+\varepsilon}\frac {a}{d}\sum_{\substack{f_1, f_2 \mid q \\ q_0 \mid \max (f_1, f_2)}} \ \sum_{\substack{n \\ (f_2nb)^2 \leq ab \big ((1+|\alpha|)(1+|\beta|)(1+|\gamma|)(1+|\delta|)\big)^{1/2+\varepsilon}q^{2+\varepsilon}}}\frac{1} {f_2n} \\
\ll & \big ((1+|\alpha|)(1+|\beta|)(1+|\gamma|)(1+|\delta|)\big)^{1/4+\varepsilon}q^{1+\varepsilon}\frac {a}{f_2d}.
\end{split}
\end{align}

  If $q_0|f_2$, then the above is
\begin{align}
\label{removeqexpestoffdiagq0f2}
\begin{split}
 \ll & \big ((1+|\alpha|)(1+|\beta|)(1+|\gamma|)(1+|\delta|)\big)^{1/4+\varepsilon}q^{1+\varepsilon}\frac {a}{q_0d}.
\end{split}
\end{align}
  
Otherwise, we have $f_2=1$, in which case $q_0|f_1$. Notice that the congruence condition $f_1ma  \pm f_2nb \equiv 0 \pmod{d}$ is solvable only when $(d, f_1a)|f_2nb=nb$.  As $(ab, q)=1$ and $d |q$, $(ab, d)=1$. It follows that $(d,f_1)|n$. Keeping in mind that $q_0 |d$ so $q_0|n$.  We now replace $n$ by $q_0n$ and apply the estimations in \eqref{removeqexpestoffdiag} to see that in this case the first expression given in \eqref{removeqexpestoffdiag} is bounded by the estimation given in \eqref{removeqexpestoffdiagq0f2} as well. \newline

 Similarly, the condition $f_1ma  + f_2nb \equiv 0 \pmod{d}$ implies that $f_1ma +f_2nb=dr$ for $r>0$. In this case, we have $dr<2f_1ma \leq 2\big ((1+|\alpha|)(1+|\beta|)(1+|\gamma|)(1+|\delta|)\big)^{1/2+\varepsilon}q^{2+\varepsilon}a/f_2n$. Moreover, we have $(f_1ma)^{-1/2} \leq 2^{1/2}(2f_1ma)^{-1/2} \leq
  2^{1/2}(f_1ma+f_2nb)^{-1/2}=2^{1/2}(rd)^{1/2}$. Thus our arguments above for the case $f_1ma  - f_2nb \equiv 0 \pmod{d}$ works for the case $f_1ma  + f_2nb \equiv 0 \pmod{d}$ as well.  Thus, from \eqref{removeqexpestd1} and \eqref{removeqexpestoffdiagq0f2} (with the observation that the case $ f_2nb >f_1ma$ leads to an estimation similar to that given in \eqref{removeqexpestoffdiagq0f2} with $a$ there being replaced by $b$) that for all $d|q$, 
\begin{align*}
\begin{split}
 \sum_{\substack{f_1, f_2 \mid q \\ q_0 \mid \max (f_1, f_2)}} & \frac{\mu(f_1)\mu(f_2)}{\sqrt{f_1f_2}}  \sum_{\substack{m, n \\ f_1ma  \pm f_2nb \equiv 0 \shortmod d}}\frac{\sigma_{\alpha,\beta}(f_1m)\sigma_{\gamma,\delta}(f_2n)} {\sqrt{mn}}V_{\alpha,\beta,\gamma,\delta}\left(\frac{f_1f_2mn}{q^2}\right) \\
\ll & \big ((1+|\alpha|)(1+|\beta|)(1+|\gamma|)(1+|\delta|)\big)^{1/4+\varepsilon}q^{1+\varepsilon}\frac {a+b}{q_0d}.
\end{split}
\end{align*}

Inserting the above into \eqref{removeqexp},
\begin{align}
\label{mnqrel}
\begin{split}
& \sum_{\substack{m, n \\ ma  \pm nb \equiv 0 \shortmod d\\ (mn, q) = 1}} \frac{\sigma_{\alpha,\beta}(m)\sigma_{\gamma,\delta}(n)} {\sqrt{mn}}V_{\alpha,\beta,\gamma,\delta}\left(\frac{mn}{q^2}\right) \\
& \hspace*{0.5cm} = \sum_{\substack{m, n \\ ma  \pm nb \equiv 0 \shortmod d}}\frac{\sigma_{\alpha,\beta}(m)\sigma_{\gamma,\delta}(n)} {\sqrt{mn}}V_{\alpha,\beta,\gamma,\delta}\left(\frac{mn}{q^2}\right) +O\Big (\big ((1+|\alpha|)(1+|\beta|)(1+|\gamma|)(1+|\delta|)\big)^{1/4+\varepsilon}q^{1+\varepsilon}\frac {a+b}{q_0d} \Big ).
\end{split}
\end{align}

  Similarly, using \eqref{gest},
\begin{align*}
\begin{split}
 & \sum_{\substack{m, n \\ ma  \pm nb \equiv 0 \shortmod d \\(mn,q)=1}} \frac{\sigma_{-\gamma,-\delta}(m)\sigma_{-\alpha,-\beta}(n)} {\sqrt{mn}}\widetilde{V}_{-\gamma,-\delta,-\alpha,-\beta}\left(\frac{mn}{q^2}\right) \\
=& \sum_{\substack{m, n \\ ma  \pm nb \equiv 0 \shortmod d}} \frac{\sigma_{-\gamma,-\delta}(m)\sigma_{-\alpha,-\beta}(n)} {\sqrt{mn}}\widetilde{V}_{-\gamma,-\delta,-\alpha,-\beta}\left(\frac{mn}{q^2}\right) +O\Big (\big ((1+|\alpha|)(1+|\beta|)(1+|\gamma|)(1+|\delta|)\big)^{1/4+\varepsilon}q^{1+\varepsilon}\frac {a+b}{q_0d} \Big ).
\end{split}
\end{align*}

  Applying the above in \eqref{1}, we see that
\begin{align}
\label{2}
\begin{split}
 \mathcal{S}(\alpha, \beta, \gamma, \delta; a,b) =& D_{+}+D_{-}+S_{+}+S_{-}+ O\Big (\big ((1+|\alpha|)(1+|\beta|)(1+|\gamma|)(1+|\delta|)\big)^{1/4+\varepsilon}q^{1+\varepsilon}\frac {a+b}{q_0} \Big ), 
\end{split}
\end{align}
  where
\begin{align}
\label{momentdecomp}
\begin{split}
 D_{+}=& \frac 12 \sum_{d \mid q} \varphi(d)\mu\left(\frac{q}{d}\right) \sum_{\substack{m, n \\ ma  = nb }}\frac{\sigma_{\alpha,\beta}(m)\sigma_{\gamma,\delta}(n)} {\sqrt{mn}}V_{\alpha,\beta,\gamma,\delta}\left(\frac{mn}{q^2}\right), \\
D_{-}=& \frac 12 \sum_{d \mid q} \varphi(d)\mu\left(\frac{q}{d}\right) \sum_{\substack{m, n \\ ma  = nb }} \frac{\sigma_{-\gamma,-\delta}(m)\sigma_{-\alpha,-\beta}(n)} {\sqrt{mn}}\widetilde{V}_{-\gamma,-\delta,-\alpha,-\beta}\left(\frac{mn}{q^2}\right), \\
S_+=& \frac 12 \sum_{d \mid q} \varphi(d)\mu\left(\frac{q}{d}\right) \sum_{\substack{m, n \\ ma  \pm nb \equiv 0 \shortmod d \\ ma \neq nb}}\frac{\sigma_{\alpha,\beta}(m)\sigma_{\gamma,\delta}(n)} {\sqrt{mn}}V_{\alpha,\beta,\gamma,\delta}\left(\frac{mn}{q^2}\right) \quad \mbox{and} \\
S_{-}=& \frac 12 \sum_{d \mid q} \varphi(d)\mu\left(\frac{q}{d}\right) \sum_{\substack{m, n \\ ma  \pm nb \equiv 0 \shortmod d\\ ma \neq nb}} \frac{\sigma_{-\gamma,-\delta}(m)\sigma_{-\alpha,-\beta}(n)} {\sqrt{mn}}\widetilde{V}_{-\gamma,-\delta,-\alpha,-\beta}\left(\frac{mn}{q^2}\right). 
\end{split}
\end{align}

\subsection{Decomposition of  $S_{+}+S_{-}$}
\label{off}

In this section, we further decompose $S_{+}+S_{-}$, the off-diagonal contribution to the fourth moment.  By \cite[Lemma 1.6]{BFKMM}, there exist two non-negative functions $v_1(x)$, $v_2(x)$ supported on $[1/2,2]$ such that, for any non-negative integer $k$,
\begin{align}
\label{Vbounds}
\begin{split}
 v^{(k)}_j(x) \ll_{k, \varepsilon} q^{k\varepsilon}, \quad j=1,2.
\end{split}
\end{align}
  Further, we have the following smooth partition of unity:
\begin{align*}
\begin{split}
  \sum_{k \geq 0}v_j \Big(\frac x{2^k}\Big)=1, \quad j=1,2.
\end{split}
\end{align*}

  We apply this to $S_{+}$ and $S_-$ in \eqref{momentdecomp}, obtaining
\begin{align}
\label{expoffdiagsmoothpartition}
\begin{split}
S_{+} = \frac 12 \sum_{M,N}S_{+, M,N} \quad \mbox{and} \quad  S_{-} = \frac 12 \sum_{M,N}S_{-, M,N},
\end{split}
\end{align}
  where
\begin{align}
\label{Splusdecomp}
\begin{split}
S_{+, M,N}  =& \sum_{d \mid q} \varphi(d)\mu\left(\frac{q}{d}\right) \sum_{\substack{m, n \\ ma  \pm nb \equiv 0 \shortmod d\\ ma  - nb \neq 0}}\frac{\sigma_{\alpha,\beta}(m)\sigma_{\gamma,\delta}(n)} {\sqrt{mn}}V_{\alpha,\beta,\gamma,\delta}\left(\frac{mn}{q^2}\right)v_1\left(\frac{m}{M}\right)v_2\left(\frac{n}{N}\right) \quad \mbox{and} \\
S_{-, M,N} =& \frac 12 \sum_{d \mid q} \varphi(d)\mu\left(\frac{q}{d}\right) \sum_{\substack{m, n \\ ma  \pm nb \equiv 0 \shortmod d\\ ma \neq nb}} \frac{\sigma_{-\gamma,-\delta}(m)\sigma_{-\alpha,-\beta}(n)} {\sqrt{mn}}\widetilde{V}_{-\gamma,-\delta,-\alpha,-\beta}\left(\frac{mn}{q^2}\right)v_1\left(\frac{m}{M}\right)v_2\left(\frac{n}{N}\right). 
\end{split}
\end{align}

   In view of \eqref{W}, we may assume that $MN \leq \big ((1+|\alpha|)(1+|\beta|)(1+|\gamma|)(1+|\delta|)\big)^{1/2+\varepsilon}q^{2+\varepsilon}$ with a negligible error. It also follows from this and \eqref{scondition} that the number of such $N$ or $M$ is $\ll \log q$. Applying similar arguments that lead to the relation \eqref{mnqrel}, we also see that for these $M$ and $N$ under our consideration, 
\begin{align}
\label{Splusrel}
\begin{split}
 S_{+, M,N}  =S'_{+, M,N}+O\Big (\big ((1+|\alpha|)(1+|\beta|)(1+|\gamma|)(1+|\delta|)\big)^{1/4+\varepsilon}q^{1+\varepsilon}\frac {a+b}{q_0} \Big ),
\end{split}
\end{align} 
  where
\begin{align*}
\begin{split}
S'_{+, M,N} =& \sum_{d \mid q} \varphi(d)\mu\left(\frac{q}{d}\right) \sum_{\substack{m, n \\ ma  \pm nb \equiv 0 \shortmod d\\ ma  - nb \neq 0 \\ (m,q)=1 }}\frac{\sigma_{\alpha,\beta}(m)\sigma_{\gamma,\delta}(n)} {(mn)^{\frac12}}V_{\alpha,\beta,\gamma,\delta}\left(\frac{mn}{q^2}\right)v_1\left(\frac{m}{M}\right)v_2\left(\frac{n}{N}\right). 
\end{split}
\end{align*}

  We define $S'_{-, M,N}$ similarly and a the relation simliar to \eqref{Splusrel} holds for $S'_{-, M,N}$ as well.  Moreover, by \eqref{scondition} and the condition that $MN \leq \big ((1+|\alpha|)(1+|\beta|)(1+|\gamma|)(1+|\delta|)\big)^{1/2+\varepsilon}q^{2+\varepsilon}$, 
we see that $n, m \ll q^{O(1)}$. This together with the condition $\alpha, \beta, \gamma, \delta \in \left\{z \in \bC: \Re(z) < \eta/\log q \right\}$ and \eqref{omegandnbound} implies that
\begin{align}
\label{sigmaest}
  \sigma_{\alpha, \beta}(m) \ll d(m) \ll M^{\varepsilon}  \quad \mbox{and} \quad \sigma_{\gamma, \delta}(n) \ll d(n) \ll N^{\varepsilon}.
\end{align}

   The above leads to the following trivial bounds
\begin{equation}
\label{verytrivial}
  S_{\pm, M,N}, S'_{\pm, M,N} \ll  \frac{1}{(MN)^{1/2-\varepsilon}}\sum_{d \mid q} \varphi(d) \sum_{M \leq m \leq 2M} \sum_{\substack{N \leq n \leq 2N\\ ma  \pm nb \equiv 0 \bmod{d}\\ ma  - nb \neq 0 }} 1 \ll  (MN)^{1/2+\varepsilon}. 
\end{equation}

From \eqref{expoffdiagsmoothpartition}, \eqref{Splusrel} and \eqref{verytrivial}, 
\begin{align}
\label{Splusdecomp12}
\begin{split}
 S_{+}+S_{-} =& \frac 12 ( S_{+,1}+S_{-,1})+\frac 12 (S_{+,2}+S_{-,2})+O\Big (\big ((1+|\alpha|)(1+|\beta|)(1+|\gamma|)(1+|\delta|)\big)^{1/4+\varepsilon}q^{1+\varepsilon}\frac {a+b}{q_0} + q^{1-\eta_0+\varepsilon} \Big ),
\end{split}
\end{align}
  where
\begin{align}
\label{S12def}
\begin{split}
S_{\pm,1}=& \sum_{\substack{M,N \ll \log q \\ q^{2-2\eta_0} \leq MN  \leq \big ((1+|\alpha|)(1+|\beta|)(1+|\gamma|)(1+|\delta|)\big)^{1/2+\varepsilon}q^{2+\varepsilon} \\\max(M, N) \leq \min (M,N)q^{1-2\eta_1} } }S_{\pm, M,N} \quad \mbox{and} \\
S_{\pm,2}=& \sum_{\substack{M,N \ll \log q \\ q^{2-2\eta_0} \leq MN  \leq \big ((1+|\alpha|)(1+|\beta|)(1+|\gamma|)(1+|\delta|)\big)^{1/2+\varepsilon}q^{2+\varepsilon} \\ \max(M, N) \geq \min (M,N)q^{1-2\eta_1} } }S'_{\pm, M,N}, 
\end{split}
\end{align}
 where $0<2\eta_0, 2\eta_1<1$ are constants to be specified later. \newline

Summarizing our discussions in this section from \eqref{2} and \eqref{Splusdecomp12}, we have the following result.
\begin{proposition}
\label{Seval}
  With the notation as above, we have
\begin{align*}
\begin{split}
 \mathcal{S}(\alpha, \beta, \gamma, \delta; a,b)  = D_{+}+D_{-} + & \frac 12 ( S_{+,1}+S_{-,1})+\frac 12 (S_{+,2}+S_{-,2}) \\
&+O\Big (\big ((1+|\alpha|)(1+|\beta|)(1+|\gamma|)(1+|\delta|)\big)^{1/4+\varepsilon}q^{1+\varepsilon}\frac {a+b}{q_0} + q^{1-\eta_0+\varepsilon} \Big ). 
\end{split}
\end{align*}
\end{proposition}

\section{The diagonal terms}
\label{sec: mainterm}
In this section, we evaluate the sum $D_{+}+D_{-}$, the diagonal contribution to the fourth moment.  From \cite[p.46]{iwakow}, we quote the identity
\begin{align}
\label{chistar}
 & \phis(q)=\sum_{d | q}\varphi(d)\mu\left(\frac{q}{d}\right).
\end{align}

As $(a,b)=1$, the condition $ma=nb$ implies that $a|n, b|m$. We may thus replace $m, n$ by $bm, an$ respectively.  Together with \eqref{chistar}, we get 
\begin{align}
\label{dt}
\begin{split}
 D_{+}=& \frac 12 \phis(q)\sum_{\substack{n }}\frac{\sigma_{\alpha,\beta}(bn)\sigma_{\gamma,\delta}(an)} {\sqrt{abn^2}}V_{\alpha,\beta,\gamma,\delta}\left(\frac{abn^2}{q^2}\right) \quad \mbox{and} \\
D_{-} =&   +\frac 12 \phis(q)\sum_{\substack{n }} \frac{\sigma_{-\gamma,-\delta}(bn)\sigma_{-\alpha,-\beta}(an)} {\sqrt{abn^2}}\widetilde{V}_{-\gamma,-\delta,-\alpha,-\beta}\left(\frac{abn^2}{q^2}\right).
\end{split}
\end{align}
  
  We apply \eqref{sigmarel} and \eqref{defV} to see from above that
\begin{align}
\label{diagint}
 D_{+}= \frac{\phis(q)}{4\pi i}\int\limits_{(2)}\frac{\mathcal{G}(s)}{s}g_{\alpha,\beta,\gamma,\delta}(s)q^{2s}Z_{\alpha, \beta, \gamma, \delta}(a, b,s) \dif s,
\end{align}
where
\begin{align*}
  Z_{\alpha, \beta, \gamma, \delta}(a, b, s)=& \sum_{n} \frac{\sigma_{\alpha,\beta}(b n)\sigma_{\gamma,\delta}(an)}{a^{1/2+s}b^{1/2+s} n^{1+2s}}= \sum_{n} \frac{\sigma_{\alpha-\beta}(b n)\sigma_{\gamma-\delta}(an)}{a^{1/2+\gamma+s}b^{1/2+\alpha+s} n^{1+\alpha+\gamma+2s}}.
\end{align*}

The discussions in Section A.3 in the preprint {\tt arXiv:1304.1241v3} gives that the Dirichlet series $ Z_{\alpha, \beta, \gamma, \delta}(a, b, s)$ can be recast as
\begin{align}
\label{Zseries}
\begin{split}
Z_{\alpha, \beta, \gamma, \delta}& (a, b, s) \\
=& \tau_{ \alpha, \beta, \gamma+2s, \delta+2s}(a)\tau_{ \gamma, \delta, \alpha+2s, \beta+2s}(b) \frac{\zeta(1 +\alpha + \gamma+2s)\zeta(1 + \alpha + \delta+2s) \zeta(1 + \beta + \gamma+2s) \zeta(1 + \beta + \delta+2s)}{a^{\frac{1}{2}+\gamma+s}b^{\frac{1}{2}+\alpha+s}\zeta(2 +\alpha + \beta +\gamma+\delta+4s)}.
\end{split}
\end{align}

  We shift the line of integration in \eqref{diagint} to $\Re(s)=-1/4+\max \{|\Re(\alpha)|,|\Re(\beta)|, |\Re(\gamma)|,|\Re(\delta)|\}+\varepsilon$, passing a simple pole at $s=0$, whose residue equals 
\begin{align}
\label{resmain}
\frac {\phis(q)}{2}\frac{\tau_{ \alpha, \beta, \gamma, \delta}(a)\tau_{ \gamma, \delta, \alpha, \beta}(b)}{a^{1/2 + \gamma} b^{1/2+\alpha}}\frac{\zeta(1 +\alpha + \gamma)\zeta(1 + \alpha + \delta) \zeta(1 + \beta + \gamma) \zeta(1 + \beta + \delta)}{\zeta(2 +\alpha + \beta +\gamma+\delta)},
\end{align}
with $\tau_{ \alpha, \beta, \gamma, \delta}$ defined in \eqref{taudef}.  Recall that $a, b \leq q$ and that $\alpha, \beta, \gamma, \delta \in \left\{z \in \bC: \Re(z) < \eta/\log q \right\}$. It follows from this and by taking $\eta$ small enough, we have on the line $\Re(s)=-1/4+\max \{|\Re(\alpha)|,|\Re(\beta)|, |\Re(\gamma)|,|\Re(\delta)|\}+\varepsilon$, 
\begin{align}
\label{zetapest}
\begin{split}
 \Big | \zeta_p(2 + \alpha + \beta + \gamma + \delta+4s) \Big | \leq & 4, \\ 
 \Big | (\zeta_p(1 +\alpha + \gamma+2s)\zeta_p(1 + \beta + \gamma+2s))^{-1} \Big |  \leq & 1 \quad \mbox{and} \\
 \Big | (\zeta_p(1 +\alpha + \delta+2s)\zeta_p(1 + \beta + \delta+2s))^{-1} \Big |  \leq & 1 .
\end{split}
\end{align}

Moreover, on that same line, we have
\begin{align}
\label{tauest}
\begin{split}
 \tau_{ \alpha, \beta, \gamma+2s, \delta+2s}(a) =& \prod_{\substack{ p| a\\ p^{\nu}\|a}}\left(1 + \frac{p^{\gamma-\delta}(p^{(\gamma-\delta)\nu}-1) \zeta_p(2 + \alpha + \beta + \gamma + \delta+4s)}{(p^{\gamma-\delta}-1)\zeta_p(1 +\alpha + \gamma+2s)\zeta_p(1 + \beta + \gamma+2s)}\right) \\
\ll & \prod_{\substack{ p| a\\ p^{\nu}\|a}}\left(1 + 4\Big |\frac{p^{\gamma-\delta}(p^{(\gamma-\delta)\nu}-1) }{p^{\gamma-\delta}-1}\Big |\right) = \prod_{\substack{ p| a\\ p^{\nu}\|a}}\left(1 + 4\Big |\sum^{\nu}_{j=1}p^{(\gamma-\delta)j} \Big |\right) \\
\ll & \prod_{\substack{ p| a\\ p^{\nu}\|a}}\left(1 + 4\nu p^{|\gamma-\delta|\nu} \right) \ll \prod_{\substack{ p| a\\ p^{\nu}\|a}}\left(5\nu p^{|\gamma-\delta|\nu} \right).
\end{split}
\end{align}
  
Note that it follows from \eqref{omegandnbound} that $\nu < \nu+1=d(p^{\nu}) \ll p^{\nu \varepsilon}$.  Hence, from this, \eqref{omegandnbound}  and \eqref{tauest} that
\begin{align}
\label{tauaest}
\begin{split}
 \tau_{ \alpha, \beta, \gamma+2s, \delta+2s}(a) \ll & 5^{\omega(a)}\prod_{\substack{ p| a\\ p^{\nu}\|a}}p^{(|\gamma-\delta|+\varepsilon)\nu} \ll a^{\varepsilon}\cdot a^{(|\gamma-\delta|+\varepsilon)} \ll a^{\varepsilon},
\end{split}
\end{align}
where the last estimate above follows by keeping in mind that $a, b \leq q$, $\alpha, \beta, \gamma, \delta \in \left\{z \in \bC: \Re(z) < \eta/\log q \right\}$ and $\Im (\gamma-\delta) =O(1)$.  From a similar argument, emerges also the bound
\begin{align}
\label{taubest}
\begin{split}
 \tau_{ \gamma, \delta, \alpha+2s, \beta+2s}(b) \ll b^{\varepsilon}.
\end{split}
\end{align}
Using \eqref{zetapest}, \eqref{tauaest} and \eqref{taubest} in \eqref{Zseries}, we see that, on the line $\Re(s)=-1/4+\max \{|\Re(\alpha)|,|\Re(\beta)|, |\Re(\gamma)|,|\Re(\delta)|\}+\varepsilon$, we have
\begin{align}
\label{zestfirst}
\begin{split}
 Z_{\alpha, \beta, \gamma, \delta}(a, b, s)  
\ll & (ab)^{-1/4+\varepsilon}\Big |\frac{\zeta(1 +\alpha + \gamma+2s)\zeta(1 + \alpha + \delta+2s) \zeta(1 + \beta + \gamma+2s) \zeta(1 + \beta + \delta+2s)}{\zeta(2 +\alpha + \beta +\gamma+\delta+4s)}\Big |. 
\end{split}
\end{align}

From \cite[(8.22)]{iwakow}, we have that the following subconvexity bound for $\zeta(s)$. For $1/2 \leq \Re(s) < 1$, 
\begin{align}
\label{weylsub}
 \zeta( s ) \ll (1+|s|)^{1/6+\varepsilon}.
\end{align}

  Moreover, it follows from \cite[Theorem 6.7]{MVa1} that on the line $\Re(s)=-1/4+ \max \{|\Re(\alpha)|,|\Re(\beta)|, |\Re(\gamma)|,|\Re(\delta)|\}+\varepsilon$,  
\begin{align}
\label{S0bound1}
 \frac{1}{\zeta(2 +\alpha + \beta +\gamma+\delta+4s)} \ll & \max (|\Re(2 +\alpha + \beta +\gamma+\delta+4s)-1|, \log (|\Im(2 +\alpha + \beta +\gamma+\delta+4s)|+4)) \\
\ll & \log |2 +\alpha + \beta +\gamma+\delta| \cdot \log (|s|+1).
\end{align} 

Utilizing \eqref{weylsub} and \eqref{S0bound1} in \eqref{zestfirst}, we get that on the afore-mentioned vertical line, 
\begin{align*}
Z_{\alpha, \beta, \gamma, \delta}(a, b, s) \ll  (ab)^{-\frac{1}{4}+\varepsilon}\big ((1+|\alpha|)(1+|\beta|)(1+|\gamma|)(1+|\delta|)\big)^{1/3+\varepsilon}(1+|s|)^{2/3+\varepsilon}.
\end{align*}

   We apply this, \eqref{gest} and use the rapid decay of $\mathcal{G}(s)$ on the vertical line to see that 
\begin{align*}
\frac1{2\pi i}  & \int\limits_{(-1/4+\max \{\Re(\alpha),\Re(\beta), \Re(\gamma),\Re(\delta)\}+\varepsilon)} \frac{\mathcal{G}(s)}{s}g_{\alpha,\beta,\gamma,\delta}(s)q^{2s}Z_{\alpha, \beta, \gamma, \delta}(a, b,s) \dif s \\
& \hspace*{1cm}  \ll (ab)^{-1/4+\varepsilon} q^{-1/2+\varepsilon}\big ((1+|\alpha|)(1+|\beta|)(1+|\gamma|)(1+|\delta|)\big)^{5/24+\max \{|\Re(\alpha)|,|\Re(\beta)|, |\Re(\gamma)|,|\Re(\delta)|\}/2+\varepsilon}.
\end{align*}

Hence, from \eqref{diagint}--\eqref{resmain} and the above,
\begin{align}
\label{diagasmp}
\begin{split}
 D_{+}
=& S_1 +O((ab)^{-1/4+\varepsilon} q^{\frac{1}{2}+\varepsilon}\big ((1+|\alpha|)(1+|\beta|)(1+|\gamma|)(1+|\delta|)\big)^{5/24+\max \{|\Re(\alpha)|,|\Re(\beta)|, |\Re(\gamma)|,|\Re(\delta)|\}/2+\varepsilon}),
\end{split}
\end{align}
  with $S_1$ given in \eqref{Sdef}. \newline

  Similar computation also renders
\begin{align}
\label{diagasmp1}
\begin{split}
D_{-}
=& S_6+O((ab)^{-1/4+\varepsilon} q^{1/2+\varepsilon}\big ((1+|\alpha|)(1+|\beta|)(1+|\gamma|)(1+|\delta|)\big)^{5/24+\max \{|\Re(\alpha)|,|\Re(\beta)|, |\Re(\gamma)|,|\Re(\delta)|\}/2+\varepsilon}).
\end{split}
\end{align}

Now the following proposition puts together our discussions in this section, after putting \eqref{diagasmp} and \eqref{diagasmp1} into \eqref{dt}.
\begin{proposition}
\label{Dsum}
  With the notation as above. We have
\begin{align*}
\begin{split}
 D_{+}+D_{-} =&  S_1+S_6 +O((ab)^{-1/4+\varepsilon} q^{1/2+\varepsilon}\big ((1+|\alpha|)(1+|\beta|)(1+|\gamma|)(1+|\delta|)\big)^{5/24+\max \{|\Re(\alpha)|,|\Re(\beta)|, |\Re(\gamma)|,|\Re(\delta)|\}/2+\varepsilon}).
\end{split}
\end{align*}
\end{proposition}

\section{Estimation of $S_{+, 2}+S_{-, 2}$}
\label{HeckeResidueClasses}

  In this section we estimate $S_{+, 2}+S_{-, 2}$. By symmetry, we may assume that $N \geq M$.  The condition $\max(M, N) \gg \min (M,N)q^{1-2\eta_1}$ then implies that $bN \geq 20aM$, provided that
\begin{align}
\label{abcond}
\begin{split}
  \max (a, b) \leq \min (a,b)\frac {q^{1-2\eta_1+\varepsilon}}{20},
\end{split}
\end{align}
a condition we shall henceforth assume. Then it follows that $bn \not= am$ is satisfied automatically.  Now using \eqref{defV} gives
\begin{align}
\label{Vint}
\begin{split}
 S'_{+, M,N}=& \sum_{d \mid q} \varphi(d)\mu\left(\frac{q}{d}\right) \int\limits_{(\varepsilon)} \frac{\mathcal{G}(s)}{s}g_{\alpha,\beta,\gamma,\delta}(s) \sum_{\substack{m, n \\ ma  \pm nb \equiv 0 \shortmod d\\  (m,q)=1}} \frac{\sigma_{\alpha, \beta}(m)\sigma_{\gamma, \delta}(n)  }{\sqrt{nm}} v_1\left(\frac{m}{M}\right)v_2\left(\frac{n}{N}\right) \left(\frac{nm}{q^2}\right)^{-s} \frac{\dif s}{2\pi i}.
\end{split}
\end{align}
  By the rapid decay of $\mathcal{G}(s)$ on the vertical line, the estimation given in \eqref{gest} and \eqref{scondition},  we can truncate the integral at $|\Im s| \leq (\log 5q)^2$ at a negligible error. We now define
\begin{align*}
\begin{split}
V_j \left(x \right)= x^{-1/2-s}v_j(x), \quad j=1,2.
 \end{split} 
\end{align*}

   Notice by \eqref{Vbounds} and the condition $|\Im s| \leq (\log 5q)^2$, the functions $V_1$ and $V_2$ are compactly support in $[1, 2]$ whose derivatives satisfy
\begin{equation}
\label{derivative}
  V_{j}^{(l)}(x) \ll (\log 5q)^{2l}\ll_j q^{l\varepsilon} , \; \mbox{for} \; j = 1, 2, \; l \geq 0. 
\end{equation}  

Using the above in \eqref{Vint}, \eqref{gest} and the rapid decay of $\mathcal{G}(s)$ on the critical line yield that \eqref{Vint} is
\begin{align}
\label{Vint1}
\begin{split}
\ll &  q^{\varepsilon}\Big |\sum_{d \mid q} \varphi(d)\mu\left(\frac{q}{d}\right) \frac{1}{\sqrt{NM}} \sum_{\substack{m, n \\ ma  \pm nb \equiv 0 \shortmod d\\(m,q)=1 }}  \sigma_{\alpha, \beta}(m)\sigma_{\gamma, \delta}(n)  V_1\left(\frac{m}{M}\right) V_2\left(\frac{n}{N}\right)\Big | \\
=& q^{\varepsilon}\Big |\sum_{d \mid q} \varphi(d)\mu\left(\frac{q}{d}\right) \frac{1}{\sqrt{NM}} \sum_{(m,q)=1} \sigma_{\alpha, \beta}(m) V_1\left(\frac{m}{M}\right) \sum_{ n \equiv \pm \overline{b}am \shortmod d} \sigma_{\gamma, \delta}(n) V_2\left(\frac{n}{N}\right)\Big |.
\end{split}
\end{align}
 
Using additive characters to detect the congruence condition above, the innermost sum in the last expression of \eqref{Vint1} equals
\begin{equation} \label{adddetect}
 \frac{1}{d} \sum_{ r \mid  d} \underset{u \shortmod r}{\left.\sum \right.^{\ast}} e\left(\pm \frac{\overline{b}amu}{r}\right) \sum_{n} \sigma_{\gamma, \delta}(n) e\left(-\frac{un}{r}\right) V_2\left(\frac{n}{N }\right),
\end{equation}
  Here and in what follows, the supscript $^*$ means that we restrict the summation to reduced residue classes. \newline

  We write for simplicity $V=V_2$ in this section and apply the Voronoi summation formula Lemma \ref{Voronoi}  to the $n$-sum in \eqref{adddetect} equals
\begin{align*}
   E_1+E_2+E_3+E_4,
\end{align*}
 where
\begin{align*}
 E_1=& \frac{1}{d} \sum_{ r \mid  d} \underset{u \shortmod r}{\left.\sum \right.^{\ast}} e\left(\pm \frac{\overline{b}amu}{r}\right) \frac{1}{r}  \int\limits_0^{+\infty}\Big (x^{-\gamma}\frac {\zeta(1+\gamma-\delta)}{r^{\gamma-\delta}}+x^{-\delta}\frac {\zeta(1+\delta-\gamma)}{r^{\delta-\gamma}}\Big )V\left(\frac{x}{N }\right) \dif x, \\
E_2=& -2\pi \sin\Big (\frac {\pi}{2}(\delta-\gamma)\Big )\frac{1}{d} \sum_{ r \mid  d}\frac{N^{1+(\gamma+\delta)/2}}{r}\sum_{n\geq 1}S\bigl(\pm \overline{b}am, n, r\bigr)\frac {\sigma_{\gamma-\delta}(n)}{n^{(\gamma-\delta)/2}}\widetilde V_{J;\gamma, \delta} \Bigl(\frac{nN}{r^2}\Bigr), \\
E_3=& -2\pi \cos\Big (\frac {\pi}{2}(\delta-\gamma)\Big )\frac{1}{d} \sum_{ r \mid  d}\frac{N^{1+(\gamma+\delta)/2}}{r}\sum_{n\geq 1}S\bigl(\pm \overline{b}am, n, r\bigr)\frac {\sigma_{\gamma-\delta}(n)}{n^{(\gamma-\delta)/2}}\widetilde V_{Y;\gamma, \delta} \Bigl(\frac{nN}{r^2}\Bigr) \quad \mbox{and} \\
E_4=& 4 \sin\Big (\frac {\pi}{2}(\delta-\gamma)\Big )\frac{1}{d} \sum_{ r \mid  d}\frac{N^{1+(\gamma+\delta)/2}}{r}\sum_{n\geq 1}S\bigl(\pm \overline{b}am, n, r\bigr)\frac {\sigma_{\gamma-\delta}(n)}{n^{(\gamma-\delta)/2}}\widetilde V_{K;\gamma, \delta} \Bigl(\frac{nN}{r^2}\Bigr), 
\end{align*}
   where $S$ is the Kloosterman sum defined in \eqref{Kloosterman}. \newline

Thus, from \eqref{Vint1} and the above that
\begin{align}
\label{SMNbound}
\begin{split}
S'_{+, M,N}\ll  &  q^{\varepsilon}\Big |\sum_{d \mid q} \varphi(d)\mu\left(\frac{q}{d}\right) \frac{1}{\sqrt{NM}} \sum_{(m,q)=1} \sigma_{\alpha, \beta}(m) V_1\left(\frac{m}{M}\right)(E_1+E_2+E_3+E_4) \Big |. 
\end{split}
\end{align}

  For integers $m,n$ with $m>0$, we recall that the Ramanujan sum $c_{m}(n)$ is defined to be (see \cite[(3.1)]{iwakow}) 
\begin{align}
\label{RSdef}
 c_{m}(n)=\underset{u \shortmod m}{\left.\sum \right.^{\ast}} e\left(\frac{nu}{m}\right). 
\end{align}
As $r|d|q$ and $(m,q)=1$, using \cite[(3.3)]{iwakow}, the Ramanujan sum can be evaluated as
\begin{align*}
 \underset{u \shortmod r}{\left.\sum \right.^{\ast}} e\left(\frac{\pm \overline{b}amu}{r}\right) =\underset{u \shortmod r}{\left.\sum \right.^{\ast}} e\left(\frac{u}{r}\right)=\mu(r). 
\end{align*}
Hence
\begin{align*}
 E_1=& \frac{1}{d} \sum_{ r \mid  d} \frac{N\mu(r)}{r}\Big (\frac {\zeta(1+\gamma-\delta)}{r^{\gamma-\delta}}\Big(\frac {1}{N} \Big)^{\gamma}\int\limits_0^{\infty}x^{-\gamma}V_2(x) \dif x+\frac {\zeta(1+\delta-\gamma)}{r^{\delta-\gamma}} \Big(\frac {1}{N} \Big)^{\delta}\int\limits_0^{\infty}x^{-\delta}V_2(x) \dif x \Big ) \\
=& \frac{1}{d}N^{1-\gamma}\zeta(1+\gamma-\delta)\prod_{p|d}(1-\frac {1}{r^{1+\gamma-\delta}})\int\limits_0^{\infty}x^{-\gamma}V_2(x) \dif x +\frac{1}{d}N^{1-\delta}\zeta(1+\delta-\gamma)\prod_{p|d}(1-\frac {1}{r^{1+\delta-\gamma}})\int\limits_0^{\infty}x^{-\delta}V_2(x) \dif x .
\end{align*}
  
  As $q=q^{n_0}_0$ with $n_0 \geq 4$, we see that
\begin{align*}
  \sum_{d \mid q} \varphi(d)\mu\left(\frac{q}{d}\right)E_1 = N^{1-\gamma} & \zeta(1+\gamma-\delta)\int\limits_0^{\infty}x^{-\gamma}V_2(x) \dif x\sum_{d \mid q}\frac {\varphi(d)}{d}\mu\left(\frac{q}{d}\right)\prod_{p|d}(1-\frac {1}{r^{1+\gamma-\delta}})\\
&+N^{1-\delta}\zeta(1+\delta-\gamma)\int\limits_0^{\infty}x^{-\delta}V_2(x) \dif x \sum_{d \mid q} \frac {\varphi(d)}{d}\mu\left(\frac{q}{d}\right)\prod_{p|d}(1-\frac {1}{r^{1+\delta-\gamma}}) = 0. 
\end{align*}
Therefore, \eqref{SMNbound} becomes
\begin{align*}
\begin{split}
S'_{+, M,N}\ll  &  q^{\varepsilon}\Big |\sum_{d \mid q} \varphi(d)\mu\left(\frac{q}{d}\right) \frac{1}{\sqrt{NM}} \sum_{(m,q)=1} \sigma_{\alpha, \beta}(m) V_1\left(\frac{m}{M}\right)(E_2+E_3+E_4) \Big |. 
\end{split}
\end{align*}

   We now estimate the right-hand side above. As the treatments are similar, we focus only the case $E_2$ here.  Applying Cauchy's inequality and using \eqref{sigmaest}, we see that
\begin{align}
\label{SboundE2}
 \begin{split}
 & q^{\varepsilon}\Big |\sum_{d \mid q} \varphi(d)\mu\left(\frac{q}{d}\right) \frac{1}{\sqrt{NM}} \sum_{(m,q)=1} \sigma_{\alpha, \beta}(m) V_1\left(\frac{m}{M}\right)E_2 \Big | \\
\ll & q^{\varepsilon}\Big (\Big|\sin\Big (\frac {\pi}{2}(\delta-\gamma)\Big )\Big |+ \Big|\cos\Big (\frac {\pi}{2}(\delta-\gamma)\Big )\Big |\Big )\sum_{d \mid q} \sum_{r \mid  d}  \frac{N^{1/2}}{r}\Bigl( \sum_{\substack{m \asymp M\\ (m,q)=1}}   \Bigl| \sum_{n}S\bigl(\pm\overline{b}am, n, r\bigr) \frac {\sigma_{\gamma-\delta}(n)}{n^{(\gamma-\delta)/2}}   V_{J;\gamma, \delta} \Bigl(\frac{nN}{r^2}\Bigr)\Bigr|^2\Bigr)^{1/2}. 
 \end{split} 
\end{align}

   It follows from \eqref{derivative} that we have
\begin{equation*}
  (x^{-(\gamma+\delta)/2}V_{2}(x))^{(j)} \ll  (|\gamma+\delta|q^{\varepsilon})^j .
\end{equation*}  
   We apply the above and Lemma \ref{besseldecay} to see that $V_{J;\alpha_1, \alpha_2} \Bigl(\frac{nN}{r^2}\Bigr)$ decays rapidly when $nN/r^2 > |\gamma+\delta|^{2+\varepsilon}q^{\varepsilon}$. Thus, we may restrict the sum over $n$ to $n \ll r^2|\gamma+\delta|^{2+\varepsilon}q^{\varepsilon}/N$ at the cost of a negligible error. Also, by  \eqref{sigmaest}, we see that $\frac {\sigma_{\gamma-\delta}(n)}{n^{(\gamma-\delta)/2}} \ll 1$ for the values of $n$ under our consideration. Moreover, the condition $(ab, q)=1$ implies that $(ab, r)=1$ as $r|d|q$.  We are thus able to apply the large sieve inequality for Kloosterman sums given in Lemma \ref{largesieveKS} by taking a divisor $q_1$ of $q$ and setting $s=(r,q_1)$ to the expression in \eqref{SboundE2} to see that with a negligible error, \eqref{SboundE2} is 
\begin{equation} \label{SboundE2simplified0}
 \ll  q^{\varepsilon}e^{\pi |\Im (\delta-\gamma)|/2}\sum_{d \mid q}\sum_{r \mid  d}  \frac{N^{1/2}}{r} \Bigl( \frac {r^2|\gamma+\delta|^{2}}{N}Mrq_1+\big (\frac {r^2|\gamma+\delta|^{2}}{N}\big )^2M(rq_1)^{1/2}+\big (\frac {r^2|\gamma+\delta|^{2}}{N}\big )^2\frac {r^{3/2}}{(r,q_1)^{1/2}}\Bigr)^{1/2}.
\end{equation}
  Now note that upon writing $d=rr'$ and using $(rr', q_0) \leq (r, q_0)(r',q_0)$, we see that for $r|d$,
\begin{align*}
 \begin{split}
  \frac {r^{11/2}}{(r,q_1)^{1/2}} \leq \frac {d^{11/2}}{(d,q_1)^{1/2}}.  
 \end{split} 
\end{align*}
Consequently, using the above in \eqref{SboundE2simplified0} and similar treatment for the terms with $E_3$ and $E_4$ in \eqref{SMNbound}, we get
\begin{align}
\label{SboundE2simplified}
 \begin{split}
  S'_{+, M,N} 
\ll & q^{\varepsilon}e^{\pi |\Im
(\delta-\gamma)|/2}\Bigl( q^{1/2}|\gamma+\delta|M^{1/2}q^{1/2}_1+ q^{5/4}|\gamma+\delta|^2\frac {M^{1/2}}{N^{1/2}}q^{1/4}_1+\frac {q^{7/4}|\gamma+\delta|^2}{N^{1/2}q^{1/4}_1}\Bigr),
  \end{split} 
\end{align}

  A similar treatment applies to $S'_{-, M,N}$ as well, so that
\begin{align}
\label{SboundE2simplified1}
 \begin{split}
  S'_{-, M,N} \ll & q^{\varepsilon}e^{\pi | \Im (\beta-\alpha)|/2}\Bigl( q^{1/2}|\alpha+\beta|M^{1/2}q^{1/2}_1+ q^{5/4}|\alpha+\beta|^2\frac {M^{1/2}}{N^{1/2}}q^{1/4}_1+\frac {q^{7/4}|\alpha+\beta|^2}{N^{1/2}q^{1/4}_1}\Bigr).
  \end{split} 
\end{align}

   We deduce from the conditions $q^{2-2\eta_0} \leq MN  \leq \big ((1+|\alpha|)(1+|\beta|)(1+|\gamma|)(1+|\delta|)\big)^{1/2+\varepsilon}q^{2+\varepsilon}, N  \geq Mq^{1-2\eta_1}$ that
\begin{align}
\label{Sprimebound}
\begin{split}
 q^{1/2}|\alpha+\beta|M^{1/2}q^{1/2}_1& + q^{5/4}|\alpha+\beta|^2\frac {M^{1/2}}{N^{1/2}}q^{1/4}_1+\frac {q^{7/4}|\alpha+\beta|^2}{N^{1/2}q^{1/4}_1} \\
\ll& \big ((1+|\alpha|)(1+|\beta|)(1+|\gamma|)(1+|\delta|)\big)^{2}q^{\varepsilon}\Big (q^{3/4+\eta_1/2}q^{1/2}_1+ q^{3/4+\eta_1}q^{1/4}_1+\frac {q^{1+\eta_0/2+\eta_1/2}}{q^{1/4}_1} \Big ).
\end{split}
\end{align}

  Similarly, we have
\begin{align}
\label{Sprimebound1}
\begin{split}
 q^{1/2}|\gamma+ & \delta|M^{1/2}q^{1/2}_1+ q^{5/4}|\gamma+\delta|^2\frac {M^{1/2}}{N^{1/2}}q^{1/4}_1+\frac {q^{7/4}|\gamma+\delta|^2}{N^{1/2}q^{1/4}_1} \\
\ll& \big ((1+|\alpha|)(1+|\beta|)(1+|\gamma|)(1+|\delta|)\big)^{2}q^{\varepsilon}\Big (q^{3/4+\eta_1/2}q^{1/2}_1+ q^{3/4+\eta_1}q^{1/4}_1+\frac {q^{1+\eta_0/2+\eta_1/2}}{q^{1/4}_1} \Big ).
\end{split}
\end{align}

  We summarize our discussions in this section from \eqref{S12def}, \eqref{SboundE2simplified}--\eqref{Sprimebound1} in the following result.
\begin{proposition}
\label{S2eval}
  With the notation as above, suppose that \eqref{abcond} holds. Then
\begin{align*}
 \begin{split}
  S_{+, 2} & +S_{-, 2} \\
&  \ll \big ((1+|\alpha|)(1+|\beta|)(1+|\gamma|)(1+|\delta|)\big)^{2}q^{\varepsilon}e^{\pi (|\Im
(\beta-\alpha)|+
|\Im(\delta-\gamma)|)/2}\Big (q^{3/4+\eta_1/2}q^{1/2}_1+ q^{3/4+\eta_1}q^{1/4}_1+\frac {q^{1+\eta_0/2+\eta_1/2}}{q^{1/4}_1} \Big ).
  \end{split} 
\end{align*}
\end{proposition}

\section{Transforming $S_{+, M,N}$ and $S_{-, M,N}$}
\label{Deltamethod}

\subsection{The $\delta$-method}  For an integers $n$, we define the $\delta$ function by
\begin{align*}
  \delta(n)=
\begin{cases}
 1,  \quad n=1, \\
 0,  \quad \text{otherwise}.
\end{cases}
\end{align*}

In the $\delta$-method introduced by W. Duke, J. B. Friedlander and H. Iwaniec in \cite{DFI94} (see also \cite[Proposition 20.16]{iwakow}), $\delta(n)$ is recast as
\begin{align}
\label{deltamethod}
 \delta (n)=\sum_{\ell = 1}^\infty \sideset{}{^*}\sum_{\substack{k \shortmod \ell}}e\left(\frac{kn}{\ell}\right)\Delta_\ell (n).
\end{align}
Here $\Delta_\ell (u)$ is a smooth function such that $\Delta_\ell (u)=0$ if $|u|\leqslant U \leq Q^2$ and $\ell > 2Q$ for some parameters $U, Q$  to be specified later.  Moreover, $\Delta_\ell$ satisfies the following bound (c.f. \cite[Lemma 2]{DFI94})
\begin{equation}
\label{BoundDelta}
\Delta_\ell (u)\ll \min \left( \frac{1}{Q^2}, \frac{1}{\ell Q}\right) + \min \left( \frac{1}{|u|},\frac{1}{\ell Q}\right).
\end{equation}

Let $\Psi$ be a smooth function $\Psi$ such that $\Psi(0)=1$, $\Psi(u)=0$ for $|u|\geqslant U$ and $\Psi^{(i)}\ll U^{-i}$ for some $U$ satisfying $U\leqslant Q^2$. We now write the congruence condition $ma \pm nb \equiv 0 \pmod d$ in \eqref{Splusdecomp} as $am \pm bn-hd$ for $h\neq 0$, as $am \neq bn$. We now keep track of the condition $am \pm bn-hd$ is not too large by attaching the function $\Psi$. It follows that $S_{+, M,N}$ given in \eqref{Splusdecomp} may be rewritten as, using \eqref{deltamethod},
\begin{align}
\label{Splusdelta}
\begin{split}
S_{+, M,N}  =& \frac{1}{\sqrt{MN}}\sum_{d \mid q} \varphi(d)\mu\left(\frac{q}{d}\right) \sum_{\ell\leqslant 2Q}\sum_{h\neq 0}\ \sumstar_{\substack{k \shortmod \ell}}e\left(\frac{-khd}{\ell}\right)  \sum_{n=1}^\infty\sum_{m=1}^\infty \sigma_{\alpha,\beta}(m)\sigma_{\gamma,\delta}(n)e\left(\frac{k(am \pm bn)}{\ell}\right)E^{\pm}(m,n,\ell),
\end{split}
\end{align}
where
\begin{equation}
\label{DefinitionEpm}
E^{\pm}(x,y,\ell):=W_1\left(\frac{x}{M}\right)W_2\left(\frac{y}{N}\right)\Psi(ax\pm by-hd)V_{\alpha,\beta,\gamma,\delta}\left(\frac{xy}{q^2}\right)\Delta_\ell (ax\pm by-hd) \quad \mbox{with} \quad
W_j \left(x \right)= x^{-1/2}v_j(x), \; j=1,2.
\end{equation}
It follows from \eqref{Vbounds} that the functions $W_1$ and $W_2$ are compactly support in $[1, 2]$ and their $j$-th derivatives are
\begin{equation*}
 \ll_j q^{j\varepsilon} .
\end{equation*}  

  Let $a' =a/(a, \ell ), \ell_a=\ell/(a, \ell )$ and $b' =b/(b, \ell ), \ell_b=\ell/(b, \ell )$. We apply the Voronoi
 summation formula, Lemma \ref{Voronoi}, to the sums over $m$ and $n$ in \eqref{Splusdelta}. Thus, $S_{+, M,N}$ can be decomposed as a main term plus an error term, such that
\begin{align}
\label{SplusVS}
\begin{split}
S_{+, M,N}  =& S^0_{+, M,N}  +E_{+, M,N},
\end{split}
\end{align}
  where
\begin{align}
\label{Splusmain}
S^0_{+, M,N} = S^{\pm;-+-+}_{+, M,N}+S^{\pm;-++-}_{+, M,N}+S^{\pm;+--+}_{+, M,N}+S^{\pm;+-+-}_{+, M,N} \quad \mbox{and} \quad E_{+, M,N}=\sum^{21}_{j=1} E_{j,+, M, N},
\end{align}
 Here
\begin{align*}
\begin{split}
& S^{\pm;-+-+}_{+, M,N}\\
& \hspace*{0.8cm} = \frac{\zeta(1-\alpha+\beta)\zeta(1-\gamma+\delta)}{\sqrt{MN}}\sum_{d \mid q} \varphi(d)\mu\left(\frac{q}{d}\right) \sum_{\ell\leqslant 2Q}\sum_{h\neq 0}\frac {c_{\ell}(hd)(a,\ell)^{1-\alpha+\beta}(b,\ell)^{1-\gamma+\delta}}{\ell^{2-\alpha+\beta-\gamma+\delta}}
\int\limits_0^{\infty}\int\limits_0^{\infty} x^{-\alpha}y^{-\gamma}E^{\pm}(x,y,\ell) \dif x \dif y,
\end{split}
\end{align*}
 where $c_{\ell}(hd)$ is the Ramanujan sum defined in \eqref{RSdef} and
\begin{align}
\label{E1exp}
\begin{split}
 E_{1,+, M, N} =& \frac {4\pi^2 \cos\Big (\frac {\pi}{2}(\alpha-\beta)\Big )\cos\Big (\frac {\pi}{2}(\delta-\gamma)\Big )}{M^{1/2+\beta}N^{1/2+\delta}}\sum_{d \mid q} \varphi(d)\mu\left(\frac{q}{d}\right) \sum_{\ell\leqslant 2Q}\frac 1{\ell_a\ell_b} \sum_{h\neq 0}\ \sumstar_{\substack{k \shortmod \ell}}e\left(\frac{-khd}{\ell}\right) \\
& \hspace*{1.5cm} \times \sum_{n=1}^\infty\sum_{m=1}^\infty \frac {\sigma_{\alpha,\beta}(m)}{m^{(\alpha-\beta)/2}}\frac {\sigma_{\gamma,\delta}(n)}{n^{(\gamma-\delta)/2}}e\left(\frac{-\overline{ka'}m}{\ell_a} \mp \frac{\overline{kb'}n}{\ell_b}  \right)I^{\pm}(m,n,\ell),
\end{split}
\end{align}
  where
\begin{align*}
\begin{split}
 I^{\pm}(m,n,\ell)=\int\limits^{\infty}_0 \int\limits^{\infty}_0x^{(\beta-\alpha)/2}y^{(\delta-\gamma)/2}Y_{\beta-\alpha} \Big(\frac {4\pi \sqrt{mx}}{\ell_a} \Big) Y_{\delta-\gamma} \Big(\frac {4\pi \sqrt{ny}}{\ell_b} \Big)E^{\pm}_1(x,y,\ell) \dif x \dif y, 
\end{split}
\end{align*}
 and
\begin{align*}
\begin{split}
 E^{\pm}_1(x,y,\ell)=\Delta_\ell (ax\pm by-hd)W_{1,\beta}\left(\frac{x}{M}\right)W_{2,\delta}\left(\frac{y}{N}\right)\Phi(ax\pm by-hd)V_{\alpha,\beta,\gamma,\delta}\left(\frac{xy}{q^2}\right),
\end{split}
\end{align*}
  with 
\begin{align}
\label{W1betadef}
\begin{split}
W_{1,\beta} \left(x \right)= x^{-1/2-\beta}v_1(x), \quad W_{2,\delta} \left(x \right)= x^{-1/2-\delta}v_2(x).
 \end{split} 
\end{align}

The other main and error terms admit similar expressions.  Moreover, in a complete analogous way, we obtain
\begin{align}
\label{SminusVS}
\begin{split}
S_{-, M,N}  =& S^0_{-, M,N}  +E_{-, M,N}.
\end{split}
\end{align}

\section{Estimating $E_{+, M,N}+E_{-, M,N}$}
\label{Esumest}

As the treatments are similar, it suffices to focus on one of the error terms, say $E_{1,+, M, N}$. 
We now set
\begin{align}
\label{UQvalue}
\begin{split}
 Q=abN^{1/2+\varepsilon}, \quad & U=Q^2. 
\end{split}
\end{align}

Tthe following result bounds the size of $E_{1,+, M, N}$, under certain conditions, analogous to \cite[Lemma 4.3]{Z2019}.
\begin{lemma}
\label{LemmeRestrictionQ} 
 With the notation as above and $Q^-=N^{1/2-\varepsilon}$, we have
\begin{enumerate}
\item[$(a)$] The $\ell$-sum contribution to $E_{1,+, M, N}$ is very small ($\ll_C q^{-C}$ for any $C>0$) unless
$$Q^-\leqslant \ell \leqslant 2Q.$$
\item[$(b)$] If $Q^-\leqslant\ell\leqslant 2Q$, then the $\ell$-sum contribution to the expression $E_{1,+, M, N}$ is very small ($\ll_C q^{-C}$ for any $C>0$) unless
$$n\leqslant\mathcal{N}_0:= \frac{(1+|\beta|)Q^{2+\varepsilon}}{N} , \  m\leqslant \mathcal{M}_0:= \frac{(1+|\delta|)Q^{2+\varepsilon}}{M}.$$
\end{enumerate}
\end{lemma}
\begin{proof}
  Our proof follows from those of \cite[Lemmas 4.1, 4.2]{Aryan15}. We evaluate the integral in \eqref{E1exp} by a change of variables to obtain 
\begin{align*}
\begin{split}
 I^{\pm} & (m,n,\ell) \\
 & =
4\Big( \frac {l^2_a}{(4\pi)^2m} \Big)^{1-(\alpha+\beta)/2} \Big(\frac {l^2_a}{(4\pi)^2n}\Big)^{1-(\alpha+\delta)/2} \int\limits^{\infty}_0 \int\limits^{\infty}_0u^{1+\beta-\alpha}v^{1+\delta-\gamma}Y_{\beta-\alpha}(u)Y_{\delta-\gamma}(v)E^{\pm}_1(\frac {u^2l^2_a}{(4\pi)^2m},\frac {v^2l^2_b}{(4\pi)^2n},\ell) \dif u \dif v. 
\end{split}
\end{align*}

  Now applying the recursive formula $\displaystyle{(z^v Y_{v}(z))^{\prime}=z^v Y_{v-1}(z)}$  (see \cite[Lemma C.1]{KMV02}) and integration by parts, we see that for integers $i,j \geq 0$,
\begin{align}
\label{Ievaluation1}
\begin{split}
 I^{\pm}(m,n,\ell)\asymp
\frac {l^{2i+2}_al^{2j+2}_b}{m^{i+1}n^{j+1}} \times \int\limits^{\infty}_0 \int\limits^{\infty}_0u^{i+1+\beta-\alpha}v^{j+1+\delta-\gamma+}Y_{i+\beta-\alpha}(u)Y_{j+\delta-\gamma}(v)\Big(E^{\pm}_1(x,y,\ell)\Big )^{(i,j,0)}\Big |_{\big(\frac {u^2l^2_a}{(4\pi)^2m},\frac {v^2l^2_b}{(4\pi)^2n},\ell \big)} \dif u \dif v. 
\end{split}
\end{align}

It follows from \cite[(11), (13)]{DFI94} that
\begin{align}
\label{deltader}
\begin{split}
 & (\Delta_{\ell}(u))^{(j)} \ll \Big( \frac {1}{\ell Q} \Big)^{j+1}.
\end{split}
\end{align}

Moreover, we infer from \eqref{Vbounds} that the functions $W_{1,\beta}, W_{2,\delta}$ are compactly support in $[1, 2]$ whose derivatives satisfy
\begin{equation}
\label{W1derivative}
  W^{(j)}_{1,\beta}(x) \ll ((1+|\beta|)q^{\varepsilon})^{j} \quad \mbox{and} \quad W^{(j)}_{2,\delta} \ll_j ((1+|\delta|)q^{\varepsilon})^{j} .
\end{equation}  

Thus, from \eqref{DefinitionEpm}, \eqref{deltader} and \eqref{W1derivative},
\begin{align}
\label{Eder}
\begin{split}
 & \Big(E^{\pm}_1(x,y,\ell)\Big )^{(i,j,0)}\Big |_{(\frac {u^2l^2_a}{(4\pi)^2m},\frac {v^2l^2_b}{(4\pi)^2n},\ell)}\ll    \frac {1}{\ell Q}\Big (\frac {(1+|\beta|)q^{\varepsilon}}{M}+ \frac {N}{q^2}+\frac {a}{U}+\frac {a}{\ell Q}\Big )^{i} \Big (\frac {(1+|\delta|)q^{\varepsilon}}{N}+ \frac {M}{q^2}+\frac {b}{U}+\frac {b}{\ell Q}\Big )^{j}.
\end{split}
\end{align}

  As $\ell \leq 2Q$ and $MN \ll q^{2+\varepsilon}$, we get
\begin{align}
\label{Derbound}
\begin{split}
  \frac {(1+|\delta|)q^{\varepsilon}}{N}+ \frac {M}{q^2}+\frac {b}{U}+\frac {b}{\ell Q} \ll & \frac {(1+|\delta|)q^{\varepsilon}}{N}+\frac {b}{\ell Q} \quad \mbox{and}  \\
  \frac {(1+|\beta|)q^{\varepsilon}}{M}+ \frac {N}{q^2}+\frac {a}{U}+\frac {a}{\ell Q} \ll &  \frac {(1+|\beta|)q^{\varepsilon}}{M}+\frac {a}{\ell Q}.
\end{split}
\end{align}

From \cite[(6.7)]{BM15}, for integers $j \geq 0$, 
\begin{align}
\label{Ymbound}
\begin{split}
 Y^{(j)}_{v}(u) \ll \frac{1}{\sqrt{u}}.
\end{split}
\end{align}
Apply \eqref{Ievaluation1}, \eqref{Eder} -- \eqref{Ymbound} in \eqref{Ievaluation1} reveals that
\begin{align}
\label{Ipmbound}
\begin{split}
 I^{\pm}(m, n, \ell) \ll & \frac {l^{2i+2}_al^{2j+2}_b}{m^{i+1}n^{j+1}(\ell Q)^{2}} \Big(\frac {(1+|\beta|)q^{\varepsilon}}{M}+\frac {a}{\ell Q}\Big)^i \Big(\frac {(1+|\delta|)q^{\varepsilon}}{N}+\frac {b}{\ell Q}\Big)^j \int\limits_{\  \frac{4\pi\sqrt{mM}}{\ell_a}}^{_{\frac{4\pi\sqrt{2mM}}{\ell_a}}} \int\limits_{\ \frac{ 4\pi\sqrt{nN}}{\ell_b}}^{_{\frac{4\pi\sqrt{2nN}}{\ell_b}}} u^{i+1/2}v^{j+1/2} \dif u \dif v  \\ 
 \ll_{i,j} & \frac{l^{i+1/2}_al^{j+1/2}_bM^{i/2+3/4}N^{j/2+3/4}}{m^{i/2+1/4}n^{j/2+1/4}(\ell Q)^{2}} \Big(\frac {(1+|\beta|)q^{\varepsilon}}{M}+\frac {a}{\ell Q}\Big)^i \Big(\frac {(1+|\delta|)q^{\varepsilon}}{N}+\frac {b}{\ell Q}\Big)^j.
 \end{split}
\end{align}

  Suppose that either
\begin{align}
\label{conditionllarge}
\begin{split}
 \frac {(1+|\delta|)q^{\varepsilon}}{N}+\frac {b}{\ell Q} \ll \frac {b}{\ell Q} \quad \text{or} \quad \frac {(1+|\beta|)q^{\varepsilon}}{M}+\frac {a}{\ell Q} \ll \frac {a}{\ell Q}.
 \end{split}
\end{align}
  Without loss of generality, we consider the case that the first one of the above inequalities holds. This implies that
\begin{align*}
\begin{split}
 \frac {(1+|\delta|)q^{\varepsilon}}{N}+\frac {b}{\ell Q} \ll \frac {ab}{\ell Q}. 
\end{split}
\end{align*}

Hence, from \eqref{Ipmbound} and the above,
\begin{align*}
\begin{split}
 I^{\pm}(m, n, \ell) 
 \ll_{i,j} & \frac{l^{i+1/2}_al^{j+1/2}_b(ab)^{j}M^{i/2+3/4}N^{j/2+3/4}}{m^{i/2+1/4}n^{j/2+1/4}(\ell Q)^{i+j+2}} \Big( \frac {(1+|\beta|)q^{\varepsilon}}{M}+\frac {a}{\ell Q} \Big)^i.
 \end{split}
\end{align*}

Setting $i=1$ above and applying the bound $\ell_b \leq \ell$ lead to
\begin{align*}
\begin{split}
 I^{\pm}(m, n, \ell) & \ll \frac{l^{3/2}_a(ab)^{j}M^{5/4}N^{j/2+3/4}}{m^{3/4}n^{j/2+1/4}\ell^{3/2}Q^{j+2}}\Big(\frac {(1+|\beta|)q^{\varepsilon}}{M}+\frac {a}{\ell Q}\Big).
 \end{split}
\end{align*}

We apply the above in \eqref{E1exp} and sum trivially to see that the corresponding $\ell$-sum contribution to $E_{1,+, M, N}$ is very small. Similar consideration leads to the same conclusion the second inequality in \eqref{conditionllarge} holds. Note in particular this is the case $\ell < N^{1/2-\varepsilon}/(1+|\delta|)$. This establishes part (a) of the lemma. \newline

  Suppose now that the inequalities in \eqref{conditionllarge} fail, i.e.
\begin{align*}
\begin{split}
 \frac {(1+|\delta|)q^{\varepsilon}}{N}+\frac {b}{\ell Q} \ll \frac {(1+|\delta|)q^{\varepsilon}}{N} \quad \text{and} \quad \frac {(1+|\beta|)q^{\varepsilon}}{M}+\frac {a}{\ell Q} \ll \frac {(1+|\beta|)q^{\varepsilon}}{M}.
 \end{split}
\end{align*}

   Then we deduce from \eqref{Ipmbound} and $\ell_a, \ell_b \leq \ell \leq 2Q$ that
\begin{align*}
\begin{split}
 I^{\pm}(m, n, \ell) 
\ll_{i,j} & \frac{(2Q)^{i+1/2}(2Q)^{j+1/2}M^{i/2+3/4}N^{j/2+3/4}}{m^{i/2+1/4}n^{j/2+1/4}(\ell Q)^{2}}\Big(\frac {(1+|\beta|)q^{\varepsilon}}{M}\Big)^i\Big(\frac {(1+|\delta|)q^{\varepsilon}}{N}\Big)^j \\
 \ll_{i,j} & \frac{QM^{3/4}N^{3/4}}{m^{1/4}n^{1/4}(\ell Q)^{2}}\Big(\frac {4Q^2(1+|\beta|)q^{\varepsilon}}{mM}\Big)^{i/2}\Big(\frac {4Q^2(1+|\delta|)q^{\varepsilon}}{nN}\Big)^{j/2}.  
 \end{split}
\end{align*}

   It follows from the above that the corresponding $\ell$-sum contribution to $E_{1,+, M, N}$ is negligible if $m > Q^{2+\varepsilon}(1+|\beta|)/M$ or $n> Q^{2+\varepsilon}(1+|\delta|)/N$. This implies part (b) of the lemma and hence completes the proof. 
\end{proof}

As in \cite[Section 4.2.3]{Z2019}, to estimate $E_{1,+, M, N}$, it suffices, by Lemma \ref{LemmeRestrictionQ}, to restrict the $\ell$-sum in \eqref{E1exp} to the range $\ell \in [Q^-,2Q]$.  Applying another partition of unity to the interval $[Q^-,2Q]$, it is enough to estimate sums of the shape (which is analogous to \cite[(4.16)]{Z2019})
\begin{align*}
\begin{split}
 & \frac {\cos\Big (\frac {\pi}{2}(\alpha-\beta)\Big )\cos\Big (\frac {\pi}{2}(\delta-\gamma)\Big )}{\mathcal{Q} M^{1/2+\beta}N^{1/2+\delta}}\sum_{d_1 \mid a', d_2 \mid b'}\sum_{d \mid q} \varphi(d)\mu\left(\frac{q}{d}\right)\sum_{(c, a'b')=1}c^{-1}\vartheta\left(\frac{cd_1d_2}{\mathcal{Q}}\right)  \\
& \hspace*{2cm} \times \sum_{n=1}^\infty\sum_{m=1}^\infty \frac {\sigma_{\alpha,\beta}(m)}{m^{(\alpha-\beta)/2}}\frac {\sigma_{\gamma,\delta}(n)}{n^{(\gamma-\delta)/2}}S(hd,d_2\overline{b'}n-d_1\overline{a'}m,cd_1d_2)I^{\pm} (m,n,cd_1d_2),
\end{split}
\end{align*} 
 where $Q^-\leqslant \mathcal{Q}\leqslant Q$ and $\vartheta$ is a smooth and compactly supported function on $\mathbb{R}_{>0}$ such that $\vartheta^{(j)}\ll_j 1$ for all $j\geqslant 0$. Here we note that, $\overline{a'}$ (resp $\overline{b'}$) need to be understood modulo $cd_2$ (resp $cd_1$). \newline

  We next factorize in an unique way $d_i=d_i^\ast d_i'$ with $(d_1^\ast , a')=(d_2^\ast , b')=1$ and $d_1' | (a')^{\infty}, d_2' | (b')^{\infty}$. Now as $(cd_1^\ast d_2^\ast,d_1'd_2')=1$, we may apply the twisted multiplicativity of the Kloosterman sums as in \cite[Section 4.2.3]{Z2019} by defining $v= d_1'd_2'$ and 
\begin{equation}
\label{Shat}
\hat{S}_v(\chi,m,n,a,b,hd):= \sumstar_{\substack{y \shortmod v }}\overline{\chi}(y)S(hd\overline{y},(d_2\overline{b'}n-d_1\overline{a'}m)\overline{y},v),
\end{equation}
  where the inverse of $a'$ (resp $b'$) are taken to modulo $d_2'$ (resp $d_1'$). \newline

  Using further partitions of unity, we are led to to estimate the sum
\begin{align}
\label{E1sumallres}
\begin{split}
 & \frac {\cos\Big (\frac {\pi}{2}(\alpha-\beta)\Big )\cos\Big (\frac {\pi}{2}(\delta-\gamma)\Big )}{\mathcal{Q} M^{1/2+\beta}N^{1/2+\delta}}\sum_{d_1 \mid a', d_2 \mid b'}\frac 1{\varphi(v)}\sum_{d \mid q} \varphi(d)\mu\left(\frac{q}{d}\right)\mathscr{D}(\mathcal{N},\mathcal{M},G,H;d,\chi),
\end{split}
\end{align} 
 where
\begin{align}
\label{SumShape3}
\begin{split}
\mathscr{D}(\mathcal{N},\mathcal{M},G,H;d,\chi):=& \sum_{h\asymp H}\sum_{|g|\asymp G}\mathop{\sum\sum}_{\substack{d_2b'n-d_1a'm=g \\ n\asymp \mathcal{N}, m\asymp \mathcal{M}}}\frac {\sigma_{\alpha,\beta}(m)}{m^{(\alpha-\beta)/2}}\frac {\sigma_{\gamma,\delta}(n)}{n^{(\gamma-\delta)/2}}\hat{S}_v(\overline{\chi},m,n,a,b,hd) \\ 
& \hspace*{1cm} \times  \sum_{(c, a'b')=1}\overline{\chi}(c)\frac{S(hd,\overline{v^2a'b'} g,cd_1^\ast d_2^\ast)}{c}\vartheta\left(\frac{cd_1d_2}{\mathcal{Q}}\right) F(m,n,|g|,h)I^{\pm} (m,n,cd_1d_2) \\ 
:= & \ \mathscr{D}^++\mathscr{D}^-+\mathscr{D}^0, 
\end{split}
\end{align}
and $1\leqslant \mathcal{N}\leqslant \mathcal{N}_0$, $1\leqslant\mathcal{M}\leqslant \mathcal{M}_0$, $1\leqslant H\leqslant abN/d$, and where $\mathscr{D}^0$ (respectively $\mathscr{D}^+$, $\mathscr{D}^-$) denotes the contribution of $g=0$ (respectively $g>0$, $g<0$). Here $F$ is a smooth and compactly supported function on $[\mathcal{N},2\mathcal{N}]\times [\mathcal{M},2\mathcal{M}]\times [B,2B]\times [H,2H]$ satisfying 
$$F^{(i,j,k,p)}\ll \mathcal{N}^{-i}\mathcal{M}^{-j}G^{-k}H^{-p}.$$ 
The size of $G$ depends on the sign of $d_2b'n-d_1a'm=g$. If $g>0$, then $G\leqslant d_2b'\mathcal{N}\leqslant (ab)^2\mathcal{N}$ while for $g<0$, $G\leqslant (ab)^2\mathcal{M}$. \newline

   We evaluate $\mathscr{D}^0$  as did in \cite[(4.21)]{Z2019} to see that 
\begin{align}
\label{BoundD0} 
\begin{split}
\mathscr{D}^0\ll & q^\varepsilon \frac{(NM\mathcal{N})^{3/4}v\sqrt{d_1^\ast d_2^\ast}}{\mathcal{M}^{1/4}Q}\sum_{h\asymp H}\sum_{c\leqslant\frac{\mathcal{Q}}{d_1d_2}}\frac{(h,v)(hd,cd_1^\ast d_2^\ast)^{1/2}}{\sqrt{c}}  \\ 
= & q^\varepsilon \frac{(NM\mathcal{N})^{3/4}v\sqrt{d_1^\ast d_2^\ast}}{\mathcal{M}^{1/4}Q}\sum_{f_1\mid v, f_2\mid cd_1^\ast d_2^\ast}\sum_{c\leqslant\frac{\mathcal{Q}}{d_1d_2}}\frac{f_1f^{1/2}_2}{\sqrt{c}}\sum_{\substack{h\asymp H \\ f_1\mid h, f_2\mid hd}}1.
\end{split}
\end{align}
  Note that $(v, cd_1^\ast d_2^\ast)=1$, so that $(f_1, f_2)=1$ and the condition $f_1\mid h, f_2\mid hd$ then implies that $f_1f_2/(f_2,d) \mid h$. Applying this in \eqref{BoundD0},
\begin{align*}
\begin{split}
\mathscr{D}^0\ll & q^\varepsilon \frac{(NM\mathcal{N})^{3/4}v\sqrt{d_1^\ast d_2^\ast}}{\mathcal{M}^{1/4}Q}\sum_{f_1\mid v, f_2\mid cd_1^\ast d_2^\ast}\sum_{c\leqslant\frac{\mathcal{Q}}{d_1d_2}}\frac{(f_2, d)}{\sqrt{f_2c}} \ll q^\varepsilon \frac{(NM\mathcal{N})^{3/4}v\sqrt{d_1^\ast d_2^\ast}H}{\mathcal{M}^{1/4}Q}\sum_{f_2\mid cd_1^\ast d_2^\ast}\sum_{c\leqslant\frac{\mathcal{Q}}{d_1d_2}}\frac{(f_2, d)}{\sqrt{f_2c}}. 
\end{split}
\end{align*}

  Note also that we have $(d, d_1^\ast d_2^\ast)=1$. Thus
\begin{align}
\label{BoundD02} 
\begin{split}
\sum_{f_2\mid cd_1^\ast d_2^\ast}\sum_{c\leqslant\frac{\mathcal{Q}}{d_1d_2}}\frac{(f_2, d)}{\sqrt{f_2c}}=& \sum_{\substack{f_3\mid d \\ (f_4,d)=1 \\ f_4 \leq \mathcal{Q} }}\frac {f_3}{\sqrt{f_3f_4}}\sum_{\substack{c\leqslant\frac{\mathcal{Q}}{d_1d_2} \\ f_3f_4 \mid cd_1^\ast d_2^\ast}}\frac{1}{\sqrt{c}}=\sum_{\substack{f_3\mid d \\ (f_4,d)=1\\ f_4 \leq \mathcal{Q}}}\frac {f_3}{\sqrt{f_3f_4}}\sum_{\substack{c\leqslant\frac{\mathcal{Q}}{d_1d_2} \\ f_3f_4/(f_4, d_1^\ast d_2^\ast)\mid c}}\frac{1}{\sqrt{c}}\\
 \ll & q^{\varepsilon} \sqrt{\frac{\mathcal{Q}}{d_1d_2}}  \sum_{\substack{(f_4,d)=1\\ f_4 \leq \mathcal{Q}}}\frac {(f_4, d_1^\ast d_2^\ast)}{f_4}  \ll q^{\varepsilon} \sqrt{\frac{\mathcal{Q}}{d_1d_2}} \sum_{f_5\mid d_1^\ast d_2^\ast}\sum_{\substack{f_6 \leq \mathcal{Q}}}\frac {1}{f_6} \ll q^{\varepsilon} \sqrt{\frac{\mathcal{Q}}{d_1d_2}} .
\end{split}
\end{align}

Applying \eqref{BoundD02} in \eqref{BoundD0}, together with the bounds $H\leq abN/d$, $\mathcal{N}\leq \mathcal{N}_0$, 
\begin{align}
\label{BoundD03} 
\begin{split}
\mathscr{D}^0 \ll & q^{\varepsilon} \frac{(NM\mathcal{N})^{3/4}v(d_1^\ast d_2^\ast)^{1/2}H\mathcal{Q}^{1/2}}{(d_1d_2)^{1/2}\mathcal{M}^{1/4}Q} \ll  q^{\varepsilon}abd^{-1}(1+|\beta|)^{3/4}M^{3/4}Nv^{1/2}Q^{1/2}\mathcal{Q}^{1/2}.
\end{split}
\end{align}

Using \eqref{BoundD03} in \eqref{E1sumallres} and then summing trivially, we see that the contribution of $\mathscr{D}^0$ to $E_{1,+, M, N}$ is 
\begin{align}
\label{Bound2Theorem}
\begin{split}
\ll & q^\varepsilon \Big|\cos\Big (\frac {\pi}{2}(\alpha-\beta)\Big )\cos\Big (\frac {\pi}{2}(\delta-\gamma)\Big )\Big |(ab)^{3/2}(1+|\beta|)^{3/4}M^{1/4}N^{1/2} \\
\ll & q^{1/2+\varepsilon}\Big|\cos\Big (\frac {\pi}{2}(\alpha-\beta)\Big )\cos\Big (\frac {\pi}{2}(\delta-\gamma)\Big )\Big | (ab)^{3/2}(1+|\beta|)^{3/4}N^{1/4}.
\end{split}
\end{align}

Returning to \eqref{SumShape3}, it still remains to consider 
\begin{align*}
\begin{split}
\mathscr{D}^{\pm}= d_1d_2 \sqrt{a'b'} & \sum_{h\asymp H}\sum_{|g|\asymp G}\mathop{\sum\sum}_{\substack{d_2b'n-d_1a'm=g \\ n\asymp \mathcal{N}, m\asymp \mathcal{M}}}\frac {\sigma_{\alpha,\beta}(m)}{m^{(\alpha-\beta)/2}}\frac {\sigma_{\gamma,\delta}(n)}{n^{(\gamma-\delta)/2}}\hat{S}_v(\overline{\chi},m,n,a,b,hd)\\ & \times  \sum_{(c, a'b'v)=1}\overline{\chi}(c)\frac{S(hd,\overline{v^2a'b'} g,cd_1^\ast d_2^\ast)}{cd_1d_2\sqrt{a'b'}}\Phi \left( \frac{4\pi \sqrt{|g|hd}}{cd_1d_2\sqrt{a'b'}}, m,n,g,h\right),
\end{split}
\end{align*}
  where
 \begin{align*}
\begin{split}
\Phi(z,m,n,g,h):=F(m,n,|g|,h)&\vartheta\left(\frac{4\pi\sqrt{|g|hd}}{z\mathcal{Q}\sqrt{a'b'}}\right)\int\limits_0^\infty\int\limits_0^\infty E\left(x,y,\frac{4\pi\sqrt{|g|hd}}{z\sqrt{a'b'}}\right) \\\times & Y_{\beta-\alpha}\left(zd_1\sqrt{\frac{a'b'}{|g|hd}mx}\right)Y_{\delta-\gamma}\left(zd_2\sqrt{\frac{a'b'}{|g|hd}ny}\right) \dif x \dif y,
\end{split}
\end{align*}
  and
\begin{align*}
\begin{split}
 E(x,y,\ell)=W_{1,(\alpha+\beta)/2}\left(\frac{x}{M}\right)W_{1,(\gamma+\delta)/2}\left(\frac{y}{N}\right)\Psi(ax+ by-hd)\Delta_\ell (ax+ by-hd). 
\end{split}
\end{align*}
  Here we recall that the functions $W_{1,(\alpha+\beta)/2},  W_{2,(\gamma+\delta)/2}$ are defined in \eqref{W1betadef}. \newline

We note the following estimation on $\Phi(z,m,n,b,h)$ from \cite[Proposition 4.6]{Z2019}. 
\begin{lemma}
\label{Lemmefunctionf} 
 The function $\Phi$ is $C_c^\infty(\mathbb{R}^5)$ with each variable supported in 
\begin{align}
\label{zrange}
\begin{split}
 z\asymp Z:= \frac{\sqrt{GHd}}{\mathcal{Q}\sqrt{a'b'}}, \ n\asymp\mathcal{N}, \ m\asymp\mathcal{M}, \ g\asymp G, \ h\asymp H.
\end{split}
\end{align}
Further, it satisfies the following bound on the partial derivatives
\begin{align}
\label{BoundDerivatives}
\begin{split}
 \Phi^{(\boldsymbol{\alpha})} \ll_{\boldsymbol{\alpha}} & \frac{M^{3/4}N^{1/4}}{ab(d_1d_2)^{1/2}(\mathcal{MN})^{1/4}} \\
& \times \big ((1+|\alpha+\beta|)^{-1}(1+|\gamma+\delta|)^{-1}q^{-\varepsilon}Z\big )^{-\alpha_1}\big ((1+|\alpha+\beta|)^{-1}q^{-\varepsilon}\mathcal{M}\big )^{-\alpha_2}\big ((1+|\gamma+\delta|)^{-1}q^{-\varepsilon}\mathcal{N}\big )^{-\alpha_3}G^{-\alpha_4} H^{-\alpha_5},
\end{split}
\end{align}
for any multi-index $\boldsymbol{\alpha}=(\alpha_1,...,\alpha_5)$ with $\alpha_1, \cdots , \alpha_5 \geq 0$.
\end{lemma}
\begin{proof}
  The ranges given in \eqref{zrange} and the estimation given in \eqref{BoundDerivatives} for the case $\boldsymbol{\alpha}=(\alpha_1,...,\alpha_5)=\bf{0}$ are established in the proof of \cite[Proposition 4.6]{Z2019}. To estimate $\Phi^{(\boldsymbol{\alpha})}$, we set  $\xi := 4\pi\sqrt{|g|hd/\ell_1'\ell_2'}$ and make a change of variables to 
see that
\begin{align*}
\begin{split}
\Phi(z,m,n,g,h)=\frac {4\xi^4}{(16\pi^2 d_1d_2z^2)^2mn}&F(m,n,|g|,h)\vartheta\left(\frac{4\pi\sqrt{|g|hd}}{z\mathcal{Q}\sqrt{\ell_1'\ell_2'}}\right) \\
& \times \int\limits_0^\infty\int\limits_0^\infty E\left(\frac {u^2\xi^2}{(4\pi d_1z)^2m},\frac {v^2\xi^2}{(4\pi d_2z)^2n},\frac{\xi}{z}\right) Y_{\beta-\alpha}\left(u\right)Y_{\delta-\gamma}\left(v\right)uv \dif u \dif v,
\end{split}
\end{align*}
  It now follows from the above, \eqref{deltader}, \eqref{W1derivative}, \eqref{Ymbound} that \eqref{BoundDerivatives} holds.  As pointed out in the proof of  \cite[Proposition 4.6]{Z2019} that when taking a derivative of $\Psi(ax+by-hd)\Delta_\ell(ax+by-hd)$ with respect to $h$, we obtain a factor $\max(d/(\ell Q), d/U) \ll 1/H$ as $H\ll abN/d$, as needed. This completes the proof of the lemma.
\end{proof} 

  We now apply the Kuznetsov trace formula, Lemma \ref{Kuznetsov}, as did in \cite{Z2019}.  This leads to
\begin{align*}
\begin{split}
\mathscr{D}^{-}= d_1d_2 \sqrt{a'b'}\sum_{h\asymp H} & \sum_{|g|\asymp G}e\left(-\frac{b\overline{d_1^\ast d_2^\ast}}{v^2\ell_1'\ell_2'}\right)\mathop{\sum\sum}_{\substack{d_2b'n-d_1a'm=b \\ n\asymp \mathcal{N}, m\asymp \mathcal{M}}}\frac {\sigma_{\alpha,\beta}(m)}{m^{(\alpha-\beta)/2}}\frac {\sigma_{\gamma,\delta}(n)}{n^{(\gamma-\delta)/2}}\hat{S}_v(\overline{\chi},m,n,a,b,hd)\\  
& \times \left(\mathscr{M}^-(m,n,g,h)+\mathscr{E}^-(m,n,g,h)\right),
\end{split}
\end{align*}
  where $\mathscr{M}$ and $\mathscr{E}$ denote the contribution of the Maa\ss \ cusp forms and the Eisenstein spectrum, i.e.
\begin{align*}
\begin{split}
\mathscr{M}^-(m,n,g,h) = & \ \sum_{f\in\mathcal{B}(vab,\chi)}\check{\Phi}_{m,n,g,h}(t_f)\frac{\sqrt{|g|hd}}{\cosh(\pi t_f)}\overline{\rho_{f,\infty}}(hd)\rho_{f,\mathfrak{a}}(b) \quad \mbox{and} \\ 
\mathscr{E}^-(m,n,g,h) = & \ \sum_{\substack{\chi_1\chi_2=\chi \\ f\in \mathcal{B}(\chi_1,\chi_2)}}\frac{1}{4\pi}\int\limits_\mathbb{R}\check{\Phi}_{m,n,g,h}(t)\frac{\sqrt{|g|hd}}{\cosh(\pi t)}\overline{\rho_{f,\infty}}(hd,t)\rho_{f,\mathfrak{a}}(b,t) \dif t,
\end{split}
\end{align*}
where $\check{\Phi}_{m,n,g,h}(t)$ is defined as in \eqref{definitionBesselTransform}. \newline

A similar expression holds for $\mathscr{D}^{+}$. As the treatments are similar, we focus on the estimation of $\mathscr{D}^{-}$ in what follows.
For the same reason, it suffices to consider the contribution of the term $\mathscr{M}^-(m,n,g,h)$ to $\mathscr{D}^{-}$, which we denote by $\mathscr{D}^{-,\mathscr{M}}$.  Further let $\mathscr{D}_K^{-,\mathscr{M}}$ denote the expression resulting from restricting the spectral parameter to the dyadic interval $K\leqslant t_f< 2K$ in $\mathscr{D}^{-,\mathscr{M}}$. It follows from \cite[Lemma 2.1]{Z2019} and Lemma \ref{Lemmefunctionf} that we can restrict our attention to $K\leq  (1+|\alpha+\beta|)(1+|\gamma+\delta|)q^\varepsilon Z$ with a negligible error.  The Mellin inversion reveals that
$$\check{\Phi}_{m,n,b,h}(t)=\frac{1}{(2\pi i)^4}\int\limits_{(0)}\int\limits_{(0)}\int\limits_{(0)}\int\limits_{(0)}\frac{\widetilde{\check{\Phi}(t)}(s_1,...,s_4)}{m^{s_1}n^{s_2}|g|^{s_3}h^{s_4}}\dif s_4 \dif s_3 \dif s_2 \dif s_1,$$
where $\widetilde{\check{\Phi}(t)}(s_1,...,s_4)$ denotes the Mellin transform of $\check{\Phi}_{m,n,g,h}(t)$, specificially
\begin{equation}\label{DefinitionMellin-Bessel}
\widetilde{\check{\Phi}(t)}(s_1,...,s_4) = \int\limits_{(\mathbb{R}_{>0})^4}\check{\Phi}_{\mathfrak{m},\mathfrak{n},\mathfrak{g},\mathfrak{h}}(t)\mathfrak{m}^{s_1}\mathfrak{n}^{s_2}\mathfrak{g}^{s_3}
\mathfrak{h}^{s_4}\frac{\dif \mathfrak{m} \dif \mathfrak{n} \dif \mathfrak{g} \dif \mathfrak{h}}{\mathfrak{m}\mathfrak{n}\mathfrak{g}\mathfrak{h}}.
\end{equation}
 Again by Lemma \ref{Lemmefunctionf}, we can restrict the supports of the integrals to $|\Im(s_i)|\leq ((1+|\alpha+\beta|)(1+|\gamma+\delta|)Kq)^{\varepsilon}$, incurring a negligible error, so that
\begin{equation}
\label{DefinitionDMaass}
\mathscr{D}_K^{-,\mathscr{M}} =\frac{d_1d_2\sqrt{a'b'}}{(4\pi i)^4}\iiiint\limits_{\substack{\Re(s_i)=0 \\ |\Im (s_i)|\leqslant ((1+|\alpha+\beta|)(1+|\gamma+\delta|)Kq)^{\varepsilon}}}\mathscr{B}_K^{-,\mathscr{M}}(s_1,...,s_4)\dif s_4 \dif s_3 \dif s_2 \dif s_1 + \text{negligible error},
\end{equation}
  where
\begin{equation*}
\begin{split}
\mathscr{B}_K^{-,\mathscr{M}}(s_1,...,s_4):= & \ \sum_{\substack{f\in\mathcal{B}(v\ell_1\ell_2,\chi) \\  K\leqslant |t_f|<2K}}\frac{\widetilde{\check{\Phi}(t_f)}(s_1,...,s_4)}{\cosh(\pi t_f)} \sum_{h\asymp H}h^{-s_4}  \sum_{g\asymp G}|g|^{-s_3}\alpha(g,h,s_1,s_2)\sqrt{hd|g|}\overline{\rho_{f,\infty}}(hd)\rho_{f,\mathfrak{a}}(b),
\end{split}
\end{equation*}
and
\begin{equation*}
\alpha(g,h,s_1,s_2):= e\left(-\frac{g\overline{d_1^\ast d_2^\ast}}{v^2a'b'}\right) \mathop{\sum\sum}_{\substack{d_2b'n-d_1a'm=b \\ n\asymp \mathcal{N}, m\asymp \mathcal{M}}}\frac {\sigma_{\alpha,\beta}(m)}{m^{(\alpha-\beta)/2+s_1}}\frac {\sigma_{\gamma,\delta}(n)}{n^{(\gamma-\delta)/2+s_2}}\hat{S}_v(\overline{\chi},m,n,a,b,hd).
\end{equation*}

  We now use \eqref{Shat} and open the Kloosterman sum to see that 
\begin{equation}
\label{phi^2}
\begin{split}
\mathscr{B}_K^{-,\mathscr{M}}(s_1,...,s_4) = \sum_{\substack{x,y (v) \\ (xy,v)=1}}\chi(y)\mathscr{A}_K(x,y,s_1,...,s_4),
\end{split}
\end{equation}
where
\begin{equation*}
\begin{split}
\mathscr{A}_K(x,y,s_1,...,s_4) := & \ \sum_{\substack{f\in\mathcal{B}(vab,\chi) \\ K\leqslant |t_f|<2K}}\frac{\widetilde{\check{\Phi}(t_f)}(s_1,...,s_4)}{\cosh(\pi t_f)}\sum_{h\asymp H}\tau(h,s_4) \sum_{g\asymp G}|g|^{-s_3}\omega(g,s_1,s_2)\sqrt{hd|g|}\overline{\rho_{f,\infty}}(hd)\rho_{f,\mathfrak{a}}(g).
\end{split}
\end{equation*}
Here
\begin{align*}
\begin{split}
 \tau(h,s_4) :=& h^{-s_4}e\left(\frac{hd\overline{y}x}{v}\right) \quad \mbox{and} \\
\omega(g,s_1,s_2,s_3):= & e\left(-\frac{g\overline{d_1^\ast d_2^\ast}}{v^2\ell_1'\ell_2'}\right) \mathop{\sum\sum}_{\substack{d_2b'n-d_1a'm=g \\ n\asymp \mathcal{N}, m\asymp \mathcal{M}}}\frac {\sigma_{\alpha,\beta}(m)}{m^{(\alpha-\beta)/2+s_1}}\frac {\sigma_{\gamma,\delta}(n)}{n^{(\gamma-\delta)/2+s_2}}e\left(\frac{\left(d_2\overline{b'}n-d_1\overline{a'}m\right)\overline{xy}}{v}\right).
\end{split}
\end{align*}

  As the supports of the integrals in \eqref{DefinitionDMaass} are restricted to $|\Im m (s_i)|\leqslant (Kq)^{\varepsilon}$, we deduce from this and 
\eqref{phi^2} that 
\begin{align}
\label{DKbound}
\begin{split}\mathscr{D}_K^{-,\mathscr{M}} \ll q^{\varepsilon} d_1d_2\sqrt{a'b'}\phi(v)^2 \sup_{\substack{x,y,s_i \\ \Re(s_i)=0, |\Im (s_i)|\leqslant (Kq)^{\varepsilon}}}|\mathscr{A}_K(x,y,s_1,...,s_4)|.
\end{split}
\end{align}

  Note that by \eqref{BoundSpectralParamater}, $|\Im (t_f)|\leqslant \theta=7/64$ so that $\cosh(\pi t_f)$ is positive. We may thus apply the Cauchy-Schwarz inequality to see that
\begin{align}
\label{CauchySBeforeLargeSieve} 
\begin{split}
\left|\mathscr{A}_K(x,y,s_1,...,s_4)\right|\leqslant & \sup_{\substack{K\leqslant t<2K \\ \Re e (s_i)=0}}\left|\widetilde{\check{\Phi}(t)}(s_1,...,s_4)\right| \ \left( \sum_{\substack{f\in\mathcal{B}(vab,\chi) \\ K\leqslant |t_f|<2K}}\frac{(1+|t_f|)^{\kappa/2}}{\cosh(\pi t_f)}\left|\sum_{h\asymp H}\tau(h,s_4)\sqrt{hd}\overline{\rho_{f,\infty}}(hd)\right|^2\right)^{1/2} \\ 
& \hspace*{1cm} \times \ \left( \sum_{\substack{f\in\mathcal{B}(vab,\chi) \\ K\leqslant |t_f|<2K}}\frac{(1+|t_f|)^{-\kappa/2}}{\cosh(\pi t_f)}\left|\sum_{g\asymp G}|g|^{-s_3}\omega(g,s_1,s_2) \sqrt{|g|}\rho_{f,\mathfrak{a}}(g)\right|^2\right)^{1/2},
\end{split}
\end{align}
where $\kappa\in\{0,1\}$ satisfies $\chi(-1)=(-1)^\kappa$. \newline

 We now apply \eqref{definitionBesselTransform} and  \eqref{DefinitionMellin-Bessel} to see that
\begin{align*}
\begin{split}
\widetilde{\check{\Phi}(t)}(s_1,...,s_4)= & \ 8i^{-\kappa}\int\limits_0^\infty \Omega(x,s_1,...,s_4) \cosh (\pi t)K_{2it}(x)\frac{\dif x}{x},
\end{split}
\end{align*}
  where
$$\Omega(z,s_1,...,s_4):=\mathop{\iiiint}_{(\mathbb{R}_{>0})^4}\Phi(z,\mathfrak{n},\mathfrak{m},\mathfrak{g},\mathfrak{h})\mathfrak{n}^{s_1}
\mathfrak{m}^{s_2}\mathfrak{g}^{s_3}\mathfrak{h}^{s_4}\frac{\dif \mathfrak{n} \dif \mathfrak{m} \dif \mathfrak{g} \dif \mathfrak{h}}{\mathfrak{n}\mathfrak{m}\mathfrak{g}\mathfrak{h}}.$$
Using again Lemma~\ref{Lemmefunctionf}, we see that the support of $\Omega$ as the function of $z$ is $z\asymp Z$ and that it satisfies the uniform bound (recall that $\Re (s_i)=0$)
\begin{align*}
\begin{split}
\frac {\partial^{i}\Omega}{\partial z^{i}} \leq \mathop{\iiiint}_{(\mathbb{R}_{>0})^4}\Big |\frac {\partial^{i}\Omega}{\partial z^{i}}\Big | \Big |\mathfrak{n}^{s_1}\mathfrak{m}^{s_2}\mathfrak{g}^{s_3}\mathfrak{h}^{s_4} \Big |
\frac{\dif \mathfrak{n} \dif \mathfrak{m} \dif \mathfrak{g} \dif \mathfrak{h}}{\mathfrak{n}\mathfrak{m}\mathfrak{g}\mathfrak{h}} \ll q^{\varepsilon}\frac{M^{3/4}N^{1/4}}{ab(d_1d_2)^{1/2}(\mathcal{MN})^{1/4}} \big ((1+|\alpha+\beta|)^{-1}(1+|\gamma+\delta|)^{-1}q^{-\varepsilon}Z\big )^{-i}.
\end{split}
\end{align*}

Therefore, it follows from \cite[Lemma 2.1]{Z2019} (applying to $\phi= (q^{\varepsilon}\frac{M^{3/4}N^{1/4}}{L(d_1d_2)^{1/2}(\mathcal{MN})^{1/4}})^{-1}\Omega $ there) that 
\begin{equation}
\label{BoundMellin-Bessel}
\sup_{\substack{K\leqslant t<2K \\ \Re(s_i)=0}}\left|\widetilde{\check{\Phi}(t)}(s_1,...,s_4)\right| \ll q^{\varepsilon}\frac{(1+|\alpha+\beta|)(1+|\gamma+\delta|)M^{3/4}N^{1/4}}{Zab(d_1d_2)^{1/2}(\mathcal{MN})^{1/4}}.
\end{equation}

  Next, we writing $h=h_1h_2$ with $(h_1, d)=1, h_2\mid d^{\infty}$, we obtain
\begin{align}
\label{sumhbound} 
\begin{split}
 \sum_{\substack{f\in\mathcal{B}(vab,\chi) \\ K\leqslant |t_f|<2K}} & \frac{(1+|t_f|)^{\kappa/2}}{\cosh(\pi t_f)}\left|\sum_{h\asymp H}\tau(h,s_4)\sqrt{hd}\overline{\rho_{f,\infty}}(hd)\right|^2 \\
=& \sum_{\substack{f\in\mathcal{B}(vab,\chi) \\ K\leqslant |t_f|<2K}}\frac{(1+|t_f|)^{\kappa/2}}{\cosh(\pi t_f)}\left|\sum_{h_2 \mid d^{\infty}}\sum_{\substack{h_1 \asymp H/h_2 \\ (h_1, d)=1}}\tau(h_1h_2,s_4)\sqrt{h_1h_2d}\overline{\rho_{f,\infty}}(h_1h_2d)\right|^2.
\end{split}
\end{align}
 Note that the sum over $h_2$ is finite, due to the fact that $H$ is bounded and $h_2$ is of the form $h_2=q^k_0$ for some non-negative integer $k$. It follows from this and an application of Cauchy's inequality that 
\begin{align}
\label{sumhbound1} 
\begin{split}
\sum_{\substack{f\in\mathcal{B}(vab,\chi) \\ K\leqslant |t_f|<2K}} & \frac{(1+|t_f|)^{\kappa/2}}{\cosh(\pi t_f)}\left|\sum_{h\asymp H}\tau(h,s_4)\sqrt{hd}\overline{\rho_{f,\infty}}(hd)\right|^2 \\
 \ll & \sum_{h_2 \mid d^{\infty}} \sum_{\substack{f\in\mathcal{B}(vab,\chi) \\ K\leqslant |t_f|<2K}}\frac{(1+|t_f|)^{\kappa/2}}{\cosh(\pi t_f)}\left|\sum_{\substack{h_1 \asymp H/h_2 \\ (h_1, d)=1}}\tau(h_1h_2,s_4)\sqrt{h_1h_2d}\overline{\rho_{f,\infty}}(h_1h_2d)\right|^2.
\end{split}
\end{align}  

  Note that we also have $(d, vab)=1$ as $(ab,q)=1$ and $v|ab$.  We then infer from the second relation in \eqref{RelationHecke1} that
$$\sqrt{h_1h_2d}\rho_{f,\infty}(h_1h_2d)=\lambda_f(h_2d)\sqrt{h_1}\rho_{f,\infty}(h_1).$$
  We are thus able to apply the large sieve inequality in Lemma \ref{TheoremSpectralLargeSieve} that
\begin{align}
\label{LargeSieve1} 
\begin{split}
 & \left( \sum_{\substack{f\in\mathcal{B}(vab,\chi) \\ K\leqslant |t_f|<2K}}  \frac{(1+|t_f|)^{\kappa/2}}{\cosh(\pi t_f)}\left|\sum_{\substack{h_1 \asymp H/h_2 \\ (h_1, d)=1}}\tau(h_1h_2,s_4)\sqrt{h_1h_2d}\overline{\rho_{f,\infty}}(h_1h_2d)\right|^2\right)^{1/2} \\ 
& \hspace*{1.5cm} = \left( \sum_{\substack{f\in\mathcal{B}(vab,\chi) \\ K\leqslant |t_f|<2K}}\frac{(1+|t_f|)^{\kappa/2}}{\cosh(\pi t_f)}\left|\sum_{\substack{h_1 \asymp H/h_2 \\ (h_1, d)=1}}\tau(h_1h_2,s_4)\lambda_f(h_2d)\overline{\rho_{f,\infty}}(h_1)\sqrt{h_1} \right|^2\right)^{1/2} \\
 & \hspace*{2cm} \ll q^{\varepsilon}|\lambda_f(h_2d)|\left(K+\ell^{1/4}_0(\nu(\infty))^{1/2}\frac {^{1/2}}{h^{1/2}_2}\right)\|\tau(h_1h_2,s_4)\|_{H/h_2}
\end{split}
\end{align}
  Note that we have from  \eqref{BoundHeckeEigenvalue2} that $|\lambda_f(h_2d)|\leqslant 2n_0(h_2d)^\theta$. Note also that $\ell_0$ is the conductor of $\chi$ and does not exceed $v$. Moreover, we have $\nu(\mathfrak{a}) \leq 1$ for any cusp $\mathfrak{a}$ and we have $\| \tau(h_1h_2,s_4) \|_H \leq H^{1/2}/h^{1/2}_2$. It follows from these and \eqref{sumhbound}--\eqref{LargeSieve1} that 
\begin{align}
\label{LargeSieve2} 
\begin{split}
 & \left( \sum_{\substack{f\in\mathcal{B}(vab,\chi) \\ K\leqslant |t_f|<2K}}\frac{(1+|t_f|)^{\kappa/2}}{\cosh(\pi t_f)}\left|\sum_{h\asymp H}\tau(h,s_4)\sqrt{hd}\overline{\rho_{f,\infty}}(hd)\right|^2\right)^{1/2} \ll q^{\varepsilon}d^{\theta}H^{1/2}\left(K+v^{1/4}H^{1/2}\right).
\end{split}
\end{align}

Similarly, using $||\omega||_G \ll q^{\varepsilon} \mathcal{N}G^{1/2}$,
\begin{align}
\label{LargeSieve3} 
\begin{split}
 & \left( \sum_{\substack{f\in\mathcal{B}(vab,\chi) \\ K\leqslant |t_f|<2K}}\frac{(1+|t_f|)^{-\kappa/2}}{\cosh(\pi t_f)}\left|\sum_{g\asymp 
G}|b|^{-s_3}\omega(g,s_1,s_2) \sqrt{|g|}\rho_{f,\mathfrak{a}}(g)\right|^2\right)^{1/2} \ll  q^{\varepsilon}\mathcal{N}G^{1/2}\left(K+v^{1/4}
G^{1/2} \right).
\end{split}
\end{align}

   We conclude from \eqref{CauchySBeforeLargeSieve}, \eqref{BoundMellin-Bessel}, \eqref{LargeSieve2} and \eqref{LargeSieve3} that
\begin{align}
\label{Aest} 
\begin{split}
 \mathscr{A}_K \ll q^{\varepsilon}(1+|\alpha+\beta|)(1+|\gamma+\delta|)d^\theta\frac{M^{3/4}N^{1/4}(GH)^{1/2}\mathcal{N}^{3/4}}{abZ(d_1d_2)^{1/2}\mathcal{M}^{1/4}}
\left(K+v^{1/4}G^{1/2}\right)\left(K+v^{1/4}H^{1/2}\right).
\end{split}
\end{align}
  
 Note that as $Z=\frac{\sqrt{GHd}}{\mathcal{Q}\sqrt{a'b'}}$, $\mathcal{Q}\geqslant N^{1/2-\varepsilon}$ and $H \leqslant abN/d$, we have
\begin{align*}
\begin{split}
K \leqslant & (1+|\alpha+\beta|)(1+|\gamma+\delta|)q^\varepsilon Z = (1+|\alpha+\beta|)(1+|\gamma+\delta|) q^\varepsilon \frac{\sqrt{
GHd}}{\mathcal{Q}\sqrt{a'b'}} \\
\ll & (1+|\alpha+\beta|)(1+|\gamma+\delta|)q^\varepsilon \frac{(abG)^{1/2}}{(a'b')^{1/2}}.
\end{split}
\end{align*}
  It follows that
\begin{align}
\label{Ksumbound} 
\begin{split}
 K+v^{1/4}G^{1/2} \ll & (1+|\alpha+\beta|)(1+|\gamma+\delta|)(abG)^{1/2} \quad \mbox{and} \\
 K+v^{1/4}H^{1/2} \ll & (1+|\alpha+\beta|)(1+|\gamma+\delta|)(abG)^{1/2}+H^{1/2}.
\end{split}
\end{align}

Hence, from \eqref{Aest}, \eqref{Ksumbound} and the expression of $Z$,
\begin{align*}
\begin{split}
 \mathscr{A}_K \ll & q^{\varepsilon}(1+|\alpha+\beta|)^3(1+|\gamma+\delta|)^3d^\theta\frac{M^{3/4}N^{1/4}GH^{1/2}\mathcal{N}^{3/4}}{Z(d_1d_2)^{1/2}\mathcal{M}^{1/4}(ab)^{1/2}}
\left((abG)^{1/2}+H^{1/2}\right)\\
\ll & q^{\varepsilon}(1+|\alpha+\beta|)^3(1+|\gamma+\delta|)^3d^{-1/2+\theta}\frac{M^{3/4}N^{1/4}G^{1/2}\mathcal{N}^{3/4}\mathcal{Q}}{(d_1d_2)^{1/2}
\mathcal{M}^{1/4}}
\left((abG)^{1/2}+H^{1/2}\right) \\
 :=& \mathscr{A}_K(G)+\mathscr{A}_K(H).
\end{split}
\end{align*}

   We now use $G \leqslant (ab)^2\mathcal{M}$ and the maximum values of $\mathcal{M}$ and $\mathcal{N}$ given by Lemma \ref{LemmeRestrictionQ} (b) to obtain
\begin{align}
\label{ABest} 
\begin{split}
\mathscr{A}_K(G) \ll & \  q^{\varepsilon}(1+|\alpha+\beta|)^3(1+|\gamma+\delta|)^3d^{-1/2+\theta}(ab)^{5/2}(d_1d_2)^{-1/2}(M\mathcal{MN})^{3/4}N^{1/4}\mathcal{Q} \\  
\ll & q^{\varepsilon}(1+|\alpha+\beta|)^3(1+|\gamma+\delta|)^3d^{-1/2+\theta}(ab)^{11/2}(d_1d_2)^{-1/2}N\mathcal{Q}.
\end{split}
\end{align}

  Similarly, using the bound $H\leqslant abN/d$, we have
\begin{align}
\label{AHest} 
\begin{split}
\mathscr{A}_K(H) 
\ll & \ q^{\varepsilon}(1+|\alpha+\beta|)^3(1+|\gamma+\delta|)^3d^{-1+\theta}(d_1d_2)^{-1/2}(ab)^{3/2}(NM\mathcal{N})^{3/4}\mathcal{M}^{1/4}\mathcal{Q} \\
\ll & q^{\varepsilon}(1+|\alpha+\beta|)^3(1+|\gamma+\delta|)^3d^{-1+\theta}(d_1d_2)^{-1/2}(ab)^{7/2}NM^{1/2}\mathcal{Q}.
\end{split}
\end{align}

  Applying \eqref{ABest}  and \eqref{AHest} in \eqref{DKbound} gives
\begin{align*}
\begin{split}
 \mathscr{D}_K^{-,\mathscr{M}} \ll q^{\varepsilon}(1+|\alpha+\beta|)^3(1+|\gamma+\delta|)^3 \sqrt{d_1d_2a'b'}\phi(v)^2 (d^{-1/2+\theta}(ab)^{11/2}N\mathcal{Q}+d^{-1+\theta}(ab)^{7/2}NM^{1/2}\mathcal{Q}).
\end{split}
\end{align*}
  We then deduce that the above estimation holds for $\mathscr{D}^{-}$ as well. Substituting the above into \eqref{E1sumallres} and then summing trivially to see that the contribution from $\mathscr{D}^{-}$ to $E_{1,+, M, N}$ is 
\begin{align}
\label{Bound2Theoremsimplified}
\begin{split}
\ll & q^{1/2+\theta+\varepsilon+} (1+|\alpha+\beta|)^3(1+|\gamma+\delta|)^3 (ab)^{3/2} \Big|\cos\Big (\frac {\pi}{2}(\alpha-\beta)\Big )\cos\Big (\frac {\pi}{2}(\delta-\gamma)\Big )\Big| \left(\left(\frac{(ab)^{7}N}{q}\right)^{1/2}+\left(\frac{(ab)^{11}N}{M}\right)^{1/2}\right) \\
\ll & q^{1/2+\theta+\varepsilon} (1+|\alpha+\beta|)^3(1+|\gamma+\delta|)^3 (ab)^{7} \Big|\cos\Big (\frac {\pi}{2}(\alpha-\beta)\Big )\cos\Big (\frac {\pi}{2}(\delta-\gamma)\Big )\Big| \left(\frac{N}{M}\right)^{1/2},
\end{split}
\end{align}
  where the last bound above follows by noting that $M/q\leqslant 1$. \newline

  We now conclude this section by gathering \eqref{Bound2Theorem} and \eqref{Bound2Theoremsimplified} to arrive at an estimation for $E_{+, M,N}$.  Note that a similar estimation holds for $E_{-, M,N}$ as well.  We summarize these results to arrive at the following result.
\begin{proposition}
\label{Eval} With the notation as above. We have
\begin{align*}
 E_{+, M,N}+E_{-, M,N} \ll   q^{\varepsilon} (1+|\alpha|)^3(1+|\beta|)^3(1+|\gamma|)^3 (1+|\delta|)^3 (ab)^{7}e^{\frac {\pi}{2}(\Im
|\beta-\alpha|+\Im
|\delta-\gamma|)} \Big (q^{1/2+\theta} \left(\frac{N}{M}\right)^{1/2}  + q^{1/2}N^{1/4}\Big ), 
\end{align*}
where the implied constant depends only on $\varepsilon$.
\end{proposition}

\section{Evaluation of $S^0_{+, M,N}+S^0_{-, M,N}$}

In evaluating $S^0_{+, M,N}+S^0_{-, M,N}$, we shall show that the sums over $m$ and $n$ contribute to a second main term of the fourth moment. As the treatments are similar, we focus on $S^{\pm;-+-+}_{+, M,N}$ in the sequal.  By a change of variables, we see that
\begin{align*}
\begin{split}
& \int\limits_0^{\infty} \int\limits_0^{\infty} x^{-\alpha}y^{-\gamma}E^{\pm}(x,y,\ell)  \dif x \dif y\\
=&\frac 1{a^{1-\alpha}b^{1-\gamma}}\iint x^{-\alpha}(\mp (u+hd-x))^{-\gamma}\Delta_\ell (u)
\Phi(u)V_{\alpha,\beta,\gamma,\delta}\left(\frac{\mp x(u+hd-x)}{q^2}\right)W_1\big( \frac {x}{aM} \big) W_2\Big( \frac {\mp (u+hd-x)}{bN} \Big) \dif x \dif u \\
=&\frac 1{a^{1-\alpha}b^{1-\gamma}}\iint ( u+hd \mp y)^{-\alpha}y^{-\gamma}\Delta_\ell (u)
\Phi(u)V_{\alpha,\beta,\gamma,\delta}\left(\frac{ y (u+hd \mp y)}{q^2}\right)W_1\Big( \frac {u+hd \mp y}{aM} \Big)W_2 \big(\frac {y}{bN} \big) \dif y \dif u.
\end{split}
\end{align*}

  As  $W_{1, 2}$ are compactly support in $[1, 2]$ and $\Phi$ supported in $[0, U]$, we see that the above expressions are bounded by
\begin{align*}
\begin{split}
 \frac {\min (aM, bN)}{ab}\int\limits_{|u| \leq U} \Big |\Delta_\ell (u)\Big | \dif u \ll 
\frac {\min (aM, bN)}{ab}\int\limits_{|u| \leq U} \frac 1{Q^2} \dif u \ll \frac {\min (aM, bN)}{ab},
\end{split}
\end{align*}
 where the last estimation above follows  from the condition $U = Q^2$ and \eqref{BoundDelta}. We then conclude that for all $\ell$, we have
\begin{align}
\label{Integralbound}
\begin{split}
& \int\limits_0^{\infty} \int\limits_0^{\infty} x^{-\alpha}y^{-\gamma}E^{\pm}(x,y,\ell) \dif x \dif y \ll \frac {\min (aM, bN)}{ab} \ll \frac {MN}{aM+bN}.
\end{split}
\end{align}

  Moreover, observe that the function $W_2(\frac {\mp(u+hd-x)}{bN})$ is non-zero only when $\mp(u+hd-x) \ll 2bN$. Similarly, $W_1(\frac {x}{aM})$ is non-zero only when $x \ll 2aM$. It follows from this and \eqref{Vbounds} that
\begin{align}
\label{higherderest}
\begin{split}
& \Big ((\mp (u+hd-x))^{-\gamma}
\Phi(u)V_{\alpha,\beta,\gamma,\delta}\left(\frac{\mp x(u+hd-x)}{q^2}\right)W_2 \Big(\frac {\mp (u+hd-x)}{bN}\Big)\Big )^{(j)} \ll \max \Big( \Big(\frac {aM}{q^2} \Big)^{j}, (bN)^{-j}, U^{-j} \Big).
\end{split}
\end{align}
  
 Recall that $MN \leq \big ((1+|\alpha|)(1+|\beta|)(1+|\gamma|)(1+|\delta|)\big)^{1/2+\varepsilon}q^{2+\varepsilon}$, so that $abMN \leq ab\big ((1+|\alpha|)(1+|\beta|)(1+|\gamma|)(1+|\delta|)\big)^{1/2+\varepsilon}q^{2+\varepsilon}$. Consequently,
\begin{align*}
\begin{split}
 \max \Big( \Big(\frac {aM}{q^2} \Big)^{j}, (bN)^{-j} \Big) \ll (ab\big ((1+|\alpha|)(1+|\beta|)(1+|\gamma|)(1+|\delta|)\big)^{1/2+\varepsilon}q^{\varepsilon})^j(bN)^{-j}.
\end{split}
\end{align*}

  It follows from \eqref{UQvalue} that \eqref{higherderest} is
\begin{align*}
\begin{split}
\ll  (ab\big ((1+|\alpha|)(1+|\beta|)(1+|\gamma|)(1+|\delta|)\big)^{1/2+\varepsilon}q^{\varepsilon})^j(bN)^{-j}.
\end{split}
\end{align*}

  We then deduce from the above and \cite[(18)]{DFI94}, for any integer $j \geq 0$, 
\begin{align}
\label{doubleIntegralsimplified}
\begin{split}
\frac 1{a^{1-\alpha}b^{1-\gamma}} & \iint x^{-\alpha}(\mp (u+hd-x))^{-\gamma}\Delta_\ell (u)
\Phi(u)V_{\alpha,\beta,\gamma,\delta}\left(\frac{\mp x(u+hd-x)}{q^2}\right)W_1\Big( \frac {x}{aM} \Big)W_2 \Big(\frac {\mp (u+hd-x)}{bN} \Big) \dif x \dif u \\
=&\frac 1{a^{1-\alpha}b^{1-\gamma}}\int x^{-\alpha}(\mp (hd-x))^{-\gamma}V_{\alpha,\beta,\gamma,\delta}\left(\frac{\mp x(hd-x)}{q^2}\right)W_1 \Big( \frac {x}{aM} \Big) W_2 \Big(\frac {\mp (hd-x)}{bN} \big) \dif x \\
&\hspace*{3cm} +O\Big( \Big(\frac {\ell Q ab\big ((1+|\alpha|)(1+|\beta|)(1+|\gamma|)(1+|\delta|)\big)^{1/2+\varepsilon}q^{\varepsilon})}{bN} \Big)^j \Big).
\end{split}
\end{align}

  We may thus replace the integral in \eqref{SplusVS} by the integral appearing on the right-hand side above and ignore the contribution of the error term above for $\ell < \Big ((bN)/(Q ab\big ((1+|\alpha|)(1+|\beta|)(1+|\gamma|)(1+|\delta|)\big)^{1/2})\Big )^{1-\varepsilon} :=L_0$. On the other hand, we note that $h \ll (aM+bN+U)/d \ll (aM+bN)/d$. Also, it follows from \cite[(3.5)]{iwakow}) that
\begin{align*}
\begin{split}
 c_{\ell}(hd) \ll (\ell, hd).
\end{split}
\end{align*}
From the above and \eqref{Integralbound},
\begin{align}
\label{Splusminusplusllarge}
\begin{split}
& \frac{\zeta(1-\alpha+\beta)\zeta(1-\gamma+\delta)}{\sqrt{MN}}\sum_{d \mid q} \varphi(d)\mu\left(\frac{q}{d}\right) \sum_{L_0 \leq \ell\leqslant 2Q}\sum_{h\neq 0}\frac {c_{\ell}(hd)(a,\ell)^{1-\alpha+\beta}(b,\ell)^{1-\gamma+\delta}}{\ell^{2-\alpha+\beta-\gamma+\delta}}
\int\limits_0^{\infty}\int\limits_0^{+\infty} x^{-\alpha}y^{-\gamma}E^{\pm}(x,y,\ell) \dif x \dif y\\
\ll & \frac{ab|\zeta(1-\alpha+\beta)\zeta(1-\gamma+\delta)|}{\sqrt{MN}}\sum_{d \mid q} \varphi(d)\Big |\mu\left(\frac{q}{d}\right)\Big |\sum_{\substack{f\mid hd \\ h\ll (aM+bN)/d }} \sum_{\substack{L_0 \leq \ell\leqslant 2Q \\ f\mid \ell}}\frac {1}{\ell^{2+\Re(-\alpha+\beta-\gamma+\delta)}}
\frac {MN}{aM+bN} \\
\ll & \frac{|\zeta(1-\alpha+\beta)\zeta(1-\gamma+\delta)|\min (aM, bN)}{\sqrt{MN}}\sum_{d \mid q} \varphi(d)\Big |\mu\left(\frac{q}{d}\right)\Big |\sum_{\substack{f\mid hd \\ h\ll (aM+bN)/d }} \frac {1}{f^{2+\Re(-\alpha+\beta-\gamma+\delta)}}\cdot \frac {1}{(L_0/f)^{1+\Re(-\alpha+\beta-\gamma+\delta)}} \\
\ll & \frac{q^{1+\varepsilon}|\zeta(1-\alpha+\beta)\zeta(1-\gamma+\delta)|\min (aM, bN)}{\sqrt{MN}}\frac { (ab\big ((1+|\alpha|)(1+|\beta|)(1+|\gamma|)(1+|\delta|)\big))^{1/4}}{(bN)^{1/4}}.
\end{split}
\end{align}   

  Similarly, using the estimate (which can be established analogue to that in \eqref{Integralbound})
\begin{align}
\label{Integralboundsimpleone}
\begin{split}
 \frac 1{a^{1-\alpha}b^{1-\gamma}}\int x^{-\alpha}(\mp (hd-x))^{-\gamma}V_{\alpha,\beta,\gamma,\delta}\left(\frac{\mp x(hd-x)}{q^2}\right)W_1 \Big( \frac {x}{aM} \Big) W_2 \Big( \frac {\mp (hd-x)}{bN} \Big) \dif x \ll  \frac {\min (aM, bN)}{ab},
\end{split}
\end{align} 
 we have
\begin{align}
\label{Splusminusplusllarge1}
\begin{split}
& \frac{\zeta(1-\alpha+\beta)\zeta(1-\gamma+\delta)}{\sqrt{MN}} \sum_{d \mid q} \varphi(d)\mu\left(\frac{q}{d}\right) \sum_{L_0 \leq \ell\leqslant 2Q}\sum_{h\neq 0}\frac {c_{\ell}(hd)(a,\ell)^{1-\alpha+\beta}(b,\ell)^{1-\gamma+\delta}}{\ell^{2-\alpha+\beta-\gamma+\delta}} \\
& \hspace*{3cm} \times \frac 1{a^{1-\alpha}b^{1-\gamma}}\int x^{-\alpha}(\mp (hd-x))^{-\gamma}V_{\alpha,\beta,\gamma,\delta}\left(\frac{\mp x(hd-x)}{q^2}\right)W_1\Big(\frac {x}{aM}\Big)W_2 \Big(\frac {\mp (hd-x)}{bN}\Big) \dif x \\
& \hspace*{1cm} \ll  \frac{q^{1+\varepsilon}|\zeta(1-\alpha+\beta)\zeta(1-\gamma+\delta)|\min (aM, bN)}{\sqrt{MN}}\frac { (ab\big ((1+|\alpha|)(1+|\beta|)(1+|\gamma|)(1+|\delta|)\big))^{1/4}}{(bN)^{1/4}}.
\end{split}
\end{align}   

  We deduce from \eqref{Splusminusplusllarge} and \eqref{Splusminusplusllarge1} that we may first replace the integral in \eqref{SplusVS} by the integral appearing on the right-hand side of \eqref{doubleIntegralsimplified} for $\ell < L_0$ and discard the remaining integral in \eqref{SplusVS} for $\ell \geq L_0$ by an error of size 
\begin{align*}
\begin{split}
\ll \frac{q^{1+\varepsilon}|\zeta(1-\alpha+\beta)\zeta(1-\gamma+\delta)|\min (aM, bN)}{\sqrt{MN}}\frac { (ab\big ((1+|\alpha|)(1+|\beta|)(1+|\gamma|)(1+|\delta|)\big))^{1/4}}{(bN)^{1/4}}.
\end{split}
\end{align*} 

  We then extend the resulting expression to all $\ell \geq 1$ with an error of the same size as above. This way, we get
\begin{align}
\label{Splusminusplussimplified}
\begin{split}
S^{\pm;-+-+}_{+, M,N} =& \frac{\zeta(1-\alpha+\beta)\zeta(1-\gamma+\delta)}{\sqrt{MN}}\sum_{d \mid q} \varphi(d)\mu\left(\frac{q}{d}\right) \sum^{\infty}_{\ell=1}\sum_{h\neq 0}\frac {c_{\ell}(hd)(a,\ell)^{1-\alpha+\beta}(b,\ell)^{1-\gamma+\delta}}{\ell^{2-\alpha+\beta-\gamma+\delta}} \\
& \hspace*{1.2cm} \times \frac 1{a^{1-\alpha}b^{1-\gamma}}\int x^{-\alpha}(\pm (hd-x))^{-\gamma}V_{\alpha,\beta,\gamma,\delta}\left(\frac{\pm x(hd-x)}{q^2}\right)W_1\Big(\frac {x}{aM}\Big)W_2\Big(\frac {\pm (hd-x)}{bN} \Big) \dif x \\
& \hspace*{0.7cm} + O\Big(\frac{q^{1+\varepsilon}|\zeta(1-\alpha+\beta)\zeta(1-\gamma+\delta)|\min (aM, bN)}{\sqrt{MN}}\frac { (ab\big ((1+|\alpha|)(1+|\beta|)(1+|\gamma|)(1+|\delta|)\big))^{1/4}}{(bN)^{1/4}}\Big).
\end{split}
\end{align}

  We now note from \cite[(3.2)]{iwakow} that
\begin{align*}
\begin{split}
 c_{\ell}(hd)=\sum_{\substack{ef=\ell \\ f \mid hd}}\mu(e)f.
\end{split}
\end{align*}

Applying this in \eqref{Splusminusplussimplified}, we arrive at
\begin{align}
\label{Splusminusplusintegralsimplified}
\begin{split}
S^{\pm;-+-+}_{+, M,N} =& \frac{\zeta(1-\alpha+\beta)\zeta(1-\gamma+\delta)}{\sqrt{MN}}\sum_{d \mid q} \varphi(d)\mu\left(\frac{q}{d}\right) \sum_{e \geq 1}\frac {\mu(e)}{e^{2-\alpha+\beta-\gamma+\delta}}\sum_{f \geq 1}\sum_{\substack{h\neq 0 \\ f \mid hd}}\frac {(a,ef)^{1-\alpha+\beta}(b,ef)^{1-\gamma+\delta}}{f^{1-\alpha+\beta-\gamma+\delta}} \\
& \hspace*{1.2cm}\times \frac 1{a^{1-\alpha}b^{1-\gamma}}\int x^{-\alpha}(\pm (hd-x))^{-\gamma}V_{\alpha,\beta,\gamma,\delta}\left(\frac{\pm x(hd-x)}{q^2}\right)W_1\Big(\frac {x}{aM}\Big)W_2\Big(\frac {\pm (hd-x)}{bN}\Big) \dif x \\
& \hspace*{0.7cm} + O\Big(\frac{q^{1+\varepsilon}|\zeta(1-\alpha+\beta)\zeta(1-\gamma+\delta)|\min (aM, bN)}{\sqrt{MN}}\frac { (ab\big ((1+|\alpha|)(1+|\beta|)(1+|\gamma|)(1+|\delta|)\big))^{1/4}}{(bN)^{1/4}} \Big).
\end{split}
\end{align}

  If $d \neq 1$, then $q_0 |d$. If moreover, in this case, $(f, d)>1$, then we must have $q_0 \mid f$. If follows from this and \eqref{Integralboundsimpleone} that
\begin{align}
\label{Splusminusplussimplifieddlarge}
\begin{split}
& \frac{\zeta(1-\alpha+\beta)\zeta(1-\gamma+\delta)}{\sqrt{MN}}\sum_{\substack{d \mid q\\ (d,q)>1}} \varphi(d)\mu\left(\frac{q}{d}\right) \sum_{e \geq 1}\frac {\mu(e)}{e^{2-\alpha+\beta-\gamma+\delta}}\sum_{f \geq 1}\sum_{\substack{h\neq 0 \\ f \mid hd \\ (f, d)>1}}\frac {(a,ef)^{1-\alpha+\beta}(b,ef)^{1-\gamma+\delta}}{f^{1-\alpha+\beta-\gamma+\delta}} \\
& \hspace*{3cm}  \times \frac 1{a^{1-\alpha}b^{1-\gamma}}\int x^{-\alpha}(\pm (hd-x))^{-\gamma}V_{\alpha,\beta,\gamma,\delta}\left(\frac{\pm x(hd-x)}{q^2}\right)W_1\Big(\frac {x}{aM} \Big)W_2 \Big(\frac {\pm (hd-x)}{bN} \Big) \dif x \\
& \hspace*{1cm} \ll \frac{\zeta(1-\alpha+\beta)\zeta(1-\gamma+\delta)\min (aM, bN)}{(MN)^{1/2}q_0}\sum_{\substack{d \mid q\\ (d,q)>1}} \varphi(d)\Big |\mu\left(\frac{q}{d}\right) \Big |\sum_{\substack{f, h \\ f\mid hd \\ h\ll (aM+bN)/d }}\frac {1}{f^{1+\Re(-\alpha+\beta-\gamma+\delta)}} \\
& \hspace*{1cm} \ll \frac{\zeta(1-\alpha+\beta)\zeta(1-\gamma+\delta)\min (aM, bN)(aM+bN)}{(MN)^{1/2}q_0} \ll  \frac{\zeta(1-\alpha+\beta)\zeta(1-\gamma+\delta)ab(MN)^{1/2}}{q_0}  . 
\end{split}
\end{align}

   We now deduce from \eqref{Splusminusplusintegralsimplified} and \eqref{Splusminusplussimplifieddlarge} that
\begin{align}
\label{Splusminusplusintegralremovingd}
\begin{split}
S^{\pm;-+-+}_{+, M,N} =& \frac{\zeta(1-\alpha+\beta)\zeta(1-\gamma+\delta)}{\sqrt{MN}}\sum_{d \mid q} \varphi(d)\mu\left(\frac{q}{d}\right) \sum_{e \geq 1}\frac {\mu(e)}{e^{2-\alpha+\beta-\gamma+\delta}}\sum_{f \geq 1}\sum_{\substack{h\neq 0 \\ f \mid hd \\ (f,d)=1}}\frac {(a,ef)^{1-\alpha+\beta}(b,ef)^{1-\gamma+\delta}}{f^{1-\alpha+\beta-\gamma+\delta}} \\
& \hspace*{1cm} \times \frac 1{a^{1-\alpha}b^{1-\gamma}}\int x^{-\alpha}(\pm (hd-x))^{-\gamma}V_{\alpha,\beta,\gamma,\delta}\left(\frac{\pm x(hd-x)}{q^2}\right)W_1 \Big(\frac {x}{aM} \Big)W_2 \Big(\frac {\pm (hd-x)}{bN} \Big) \dif x \\
&\hspace*{1cm} + O\Big(\frac{q^{1+\varepsilon}|\zeta(1-\alpha+\beta)\zeta(1-\gamma+\delta)|\min (aM, bN)}{\sqrt{MN}}\frac { (ab\big ((1+|\alpha|)(1+|\beta|)(1+|\gamma|)(1+|\delta|)\big))^{1/4}}{(bN)^{1/4}}\Big) \\
&\hspace*{1cm} +O\Big( \frac{\zeta(1-\alpha+\beta)\zeta(1-\gamma+\delta)ab\sqrt{MN}}{q_0} \Big) \\
=:& T^{\pm;-+-+}_{+, M,N} + O\Big(\frac{q^{1+\varepsilon}|\zeta(1-\alpha+\beta)\zeta(1-\gamma+\delta)|\min (aM, bN)}{\sqrt{MN}}\frac { (ab\big ((1+|\alpha|)(1+|\beta|)(1+|\gamma|)(1+|\delta|)\big))^{1/4}}{(bN)^{1/4}}\Big) \\
& \hspace*{1cm} +O\Big(\frac{\zeta(1-\alpha+\beta)\zeta(1-\gamma+\delta)ab(MN)^{1/2}}{q_0}\Big).
\end{split}
\end{align}

Let $\varphi^{+}(q)$ be the number of even primitive Dirichlet characters modulo $q$ and we write $\tilde{v}_j$, $j=1,2$ the Mellin transforms of $v_j$. We also set $\boldsymbol{\alpha}=(\alpha, \beta), \boldsymbol{\beta}=(\gamma, \delta)$. Then we note that $T^{\pm;-+-+}_{+, M,N}=2\varphi^{+}(q)T^{\pm;-+-+}_{M,N}(\boldsymbol{\alpha}, \boldsymbol{\beta})$, with $T^{-;-+-+}_{M,N}(\boldsymbol{\alpha}, \boldsymbol{\beta})$ defined on the bottom of \cite[p. 71]{Liu24} by setting $t=0$, $\widetilde{W}(u)=\tilde{v}_1(u)$, $\widetilde{W}(v)=\tilde{v}_2(v)$ there. The expression $T^{+;-+-+}_{M,N}(\boldsymbol{\alpha}, \boldsymbol{\beta})$ is similarly defined. \newline

   We also define the function $C_{\boldsymbol{\alpha}, \boldsymbol{\beta},a,b}(s)$ such that
\begin{align}
\label{Cproduct}
\begin{split}
C_{\boldsymbol{\alpha}, \boldsymbol{\beta},a,b}(s)=C_{\boldsymbol{\alpha}, \boldsymbol{\beta},a}(s)C_{\boldsymbol{\beta},\boldsymbol{\alpha}, b}(s),
\end{split}
\end{align}
  where 
\begin{align}
\label{Cadef}
\begin{split}
C_{\boldsymbol{\alpha}, \boldsymbol{\beta},a}(s)=\prod_{p|a}\Big( 1-\frac {1}{p^{2-\alpha+\beta-\gamma+\delta}}\Big)^{-1}\prod_{p^{l_a}\|a}\Big(1-\frac {1}{p^{2s+\alpha+\delta}}\Big)^{-1}\Big( C^0_{\boldsymbol{\alpha}, \boldsymbol{\beta},a}(s)-\frac {C^1_{\boldsymbol{\alpha}, \boldsymbol{\beta},a}(s)}{p}+\frac {C^2_{\boldsymbol{\alpha}, \boldsymbol{\beta},a}(s)}{p^2} \Big),
\end{split}
\end{align}
  and where
\begin{align}
\label{C0adef}
\begin{split}
 C^0_{\boldsymbol{\alpha}, \boldsymbol{\beta},a}(s)=& 1-\frac {1}{p^{(1+l_a)(2s+\alpha+\delta)}}, \quad
 C^1_{\boldsymbol{\alpha}, \boldsymbol{\beta},a}(s)= \Big(p^{\gamma-\delta}+\frac {1}{p^{2s+\beta+\delta}}\Big)\Big(1-\frac {1}{p^{l_a(2s+\alpha+\delta)}}\Big) \quad \mbox{and} \\
 C^2_{\boldsymbol{\alpha}, \boldsymbol{\beta},a}(s)=& p^{\alpha-\beta+\gamma-\delta}\Big(\frac {1}{p^{2s+\alpha+\delta}}-\frac {1}{p^{l_a(2s+\alpha+\delta)}}\Big).
\end{split}
\end{align}
  We are thus able to proceed as did on \cite[p. 71-75]{Liu24} and utilize \eqref{chistar} to arrive at, analogous to the last display on \cite[p. 75]{Liu24}, that
\begin{align}
\label{Tsum}
\begin{split}
 T^{+;-+-+}_{+, M,N} & +  T^{-;-+-+}_{+, M,N}\\
=& \frac{\zeta(1-\alpha+\beta)\zeta(1-\gamma+\delta)}{\zeta(2-\alpha+\beta-\gamma+\delta)}\sum_{d \mid q} \varphi(d)\mu\left(\frac{q}{d}\right)\frac 1{(2\pi i)^3}\int\limits_{(2\varepsilon)}\int\limits_{(\varepsilon)}\int\limits_{(\varepsilon)}C_{\boldsymbol{\alpha}, \boldsymbol{\beta},a,b} \Big(s+\frac {u+v}{2}\Big)\mathcal{G}(s)g_{\alpha, \beta, \gamma, \delta}(s)\tilde{v}_1(u)\tilde{v}_2(v) \\
& \hspace*{2cm} \times \pi^{1/2}\big(\frac {q}{\pi}\big)^{2s}d^{-(2s+u+v+\alpha+\gamma)}M^uN^v\frac {\zeta(2s+u+v+\alpha+\gamma)\zeta(1+2s+u+v+\beta+\delta)}{a^{1/2-s-u-\alpha}b^{1/2-s-v-\gamma}}\\
& \hspace*{2cm} \times \frac {\zeta(\frac {2s+u+v+\alpha+\gamma}{2})\zeta(\frac {1/2-s-u-\alpha}{2})\zeta(\frac {1/2-s-v-\gamma}{2})}{\zeta(\frac {1/2+s+u+\alpha}{2})\zeta(\frac {1/2+s+v+\gamma}{2})\zeta(\frac {1-2s-u-v-\alpha-\gamma}{2})} \dif v \dif u \frac {\dif s}{s} +2\varphi^+(q)(R_1-R'_1) \\
=& \frac{\zeta(1-\alpha+\beta)\zeta(1-\gamma+\delta)}{\zeta(2-\alpha+\beta-\gamma+\delta)}\sum_{d \mid q} \varphi(d)\mu\left(\frac{q}{d}\right)\frac 1{(2\pi i)^3}\int\limits_{(2\varepsilon)}\int\limits_{(\varepsilon)}\int\limits_{(\varepsilon)}C_{\boldsymbol{\alpha}, \boldsymbol{\beta},a,b}\Big(s+\frac {u+v}{2}\Big)\mathcal{G}(s)g_{\alpha, \beta, \gamma, \delta}(s)\tilde{v}_1(u)\tilde{v}_2(v) \\
& \hspace*{2cm} \times \pi^{1/2}(\frac {q}{\pi})^{2s}d^{-(2s+u+v+\alpha+\gamma)}M^uN^v\frac {\zeta(2s+u+v+\alpha+\gamma)\zeta(1+2s+u+v+\beta+\delta)}{a^{1/2-s-u-\alpha}b^{1/2-s-v-\gamma}}\\
& \hspace*{2cm} \times \frac {\Gamma(\frac {2s+u+v+\alpha+\gamma}{2})\Gamma(\frac {1/2-s-u-\alpha}{2})\Gamma(\frac {1/2-s-v-\gamma}{2})}{\Gamma(\frac {1/2+s+u+\alpha}{2})\Gamma(\frac {1/2+s+v+\gamma}{2})\Gamma(\frac {1-2s-u-v-\alpha-\gamma}{2})}\dif v \dif u \frac {\dif s}{s},
\end{split}
\end{align}
 where the difference $R_1-R'_1$ is given on \cite[p. 75]{Liu24} with similar modifications as mentioned in the paragraph above \eqref{Cproduct}. Here the last expression above follows by noting that $R_1-R'_1$ equals certain expression multiplied by $\sum_{d|q}\mu(q/d)\varphi(d)/d$, which equals $0$ in our case. \newline

  We recall that $\alpha, \beta, \gamma, \delta \in \left\{z \in \bC: \Re(z) < \eta/\log q \right\}$ and the condition given in \eqref{scondition}. It follows from these conditions and \cite[Corollary 1.17]{MVa1} that for $\Re(2s+u+v+\beta+\delta)>0$, we have
\begin{align}
\label{zetabound}
\begin{split}
 \zeta(1-\alpha+\beta)\zeta(1-\gamma+\delta), \ \zeta(1+2s+u+v+\beta+\delta) \ll q^{\varepsilon}.
\end{split}
\end{align}
  
  Moreover, by \cite[Theorem 6.7]{MVa1}, we have
\begin{align}
\label{zetainversebound}
\begin{split}
 (\zeta(2-\alpha + \beta -\gamma +\delta))^{-1} \ll q^{\varepsilon}.
\end{split}
\end{align}

 Also, the convexity bound (see \cite[Exercise 3, p. 100]{iwakow}) for $\zeta(s)$ implies that
\begin{align}
\label{zetaest}
\begin{split}
  \zeta(2s+u+v+\alpha+\gamma) \ll & q^{\varepsilon}\left( 1+|\alpha+\gamma| \right)^{(1-\Re(2s+u+v+\alpha+\gamma))/2+\varepsilon}, \quad \mbox{for} \; 0 \leq \Re(2s+u+v+\alpha+\gamma) \leq 1.
\end{split}
\end{align}

 Furthermore, by Stirling’s formula, \cite[(5.112)]{iwakow}, the condition that $\alpha, \beta, \gamma, \delta \in \left\{z \in \bC: \Re(z) < \eta/\log q \right\}$ and \eqref{scondition}, we see that 
\begin{align}
\label{gammaratiobound}
\begin{split}
\frac {\Gamma(\frac {2s+u+v+\alpha+\gamma}{2})\Gamma(\frac {1/2-s-u-\alpha}{2})\Gamma(\frac {1/2-s-v-\gamma}{2})}{\Gamma(\frac {1/2+s+u+\alpha}{2})\Gamma(\frac {1/2+s+v+\gamma}{2})\Gamma(\frac {1-2s-u-v-\alpha-\gamma}{2})}\ll q^{\varepsilon}(1+|\alpha|)^{-\Re(s+u)}(1+|\gamma|)^{-\Re(s+v)}(1+|\alpha+\gamma|)^{-\Re(2s+u+v)-1/2}.
\end{split}
\end{align}

Additionally, for $\Re(s) \geq \varepsilon>0$, we see from \eqref{C0adef}
\begin{align*}
\begin{split}
 C^0_{\boldsymbol{\alpha}, \boldsymbol{\beta},a}(s)-\frac {C^1_{\boldsymbol{\alpha}, \boldsymbol{\beta},a}(s)}{p}+\frac {C^2_{\boldsymbol{\alpha}, \boldsymbol{\beta},a}(s)}{p^2} \ll & 1,  \quad \Big(1-\frac {1}{p^{2-\alpha+\beta-\gamma+\delta}}\Big)^{-1} \ll  1 \quad \mbox{and} \\
 \Big(1-\frac {1}{p^{2s+\alpha+\delta}}\Big)^{-1} \ll & \Big(1-\frac {1}{1000^{\varepsilon}}\Big)^{-1} \ll_{\varepsilon} 1.
\end{split}
\end{align*}
  Using the above, \eqref{Cproduct}, \eqref{Cadef} and following the arguments that leading to \eqref{tauaest}, we obtain that when $\Re(s) \geq \varepsilon>0$,
\begin{align}
\label{Cproductest}
\begin{split}
C_{\boldsymbol{\alpha}, \boldsymbol{\beta},a,b}(s) \ll (ab)^{\varepsilon}.
\end{split}
\end{align}

  We now apply the estimations to see that we may first replace $d^{-(2s+u+v+\alpha+\gamma)}$ in \eqref{Tsum} by $q^{-(2s+u+v+\alpha+\gamma)}$ and replace $\sum_{d \mid q} \varphi(d)\mu\left(\frac{q}{d}\right)$ by $\phis(q)$, incurring an error term of size $O(q^{1+\varepsilon}(ab)^{-1/2}/q_0)$.  Thus,
\begin{align}
\label{Tsum1}
\begin{split}
 T^{+;-+-+}_{+, M,N} &+  T^{-;-+-+}_{+, M,N}\\
=& \phis(q)\frac{\zeta(1-\alpha+\beta)\zeta(1-\gamma+\delta)}{\zeta(2-\alpha+\beta-\gamma+\delta)}\frac 1{(2\pi i)^3}\int\limits_{(2\varepsilon)}\int\limits_{(\varepsilon)}\int\limits_{(\varepsilon)}C_{\boldsymbol{\alpha}, \boldsymbol{\beta},a,b}(s+\frac {u+v}{2})\mathcal{G}(s)g_{\alpha, \beta, \gamma, \delta}(s)\tilde{v}_1(u)\tilde{v}_2(v) \\
& \hspace*{2cm} \times \pi^{1/2-2s}q^{-(u+v+\alpha+\gamma)}M^uN^v\frac {\zeta(2s+u+v+\alpha+\gamma)\zeta(1+2s+u+v+\beta+\delta)}{a^{1/2-s-u-\alpha}b^{1/2-s-v-\gamma}}\\
& \hspace*{2cm} \times \frac {\Gamma(\frac {2s+u+v+\alpha+\gamma}{2})\Gamma(\frac {1/2-s-u-\alpha}{2})\Gamma(\frac {1/2-s-v-\gamma}{2})}{\Gamma(\frac {1/2+s+u+\alpha}{2})\Gamma(\frac {1/2+s+v+\gamma}{2})\Gamma(\frac {1-2s-u-v-\alpha-\gamma}{2})} \dif v \dif u \frac{\dif s}{s}+O\Big(\frac {q^{1+\varepsilon}(ab)^{-1/2}}{q_0}\Big).
\end{split}
\end{align}
 
  Moreover, using \eqref{zetabound}, we see that the sum of the $O$-terms in \eqref{Splusminusplusintegralremovingd} is 
\begin{align*}
\begin{split}
 \ll & \frac{q^{1+\varepsilon}\min (aM, bN)}{\sqrt{MN}}\frac { (ab\big ((1+|\alpha|)(1+|\beta|)(1+|\gamma|)(1+|\delta|)\big))^{1/4}}{(bN)^{1/4}}+\frac{q^{\varepsilon}ab(MN)^{1/2}}{q_0} \\
\ll & q^{1+\varepsilon}(\big ((1+|\alpha|)(1+|\beta|)(1+|\gamma|)(1+|\delta|)\big))^{1/4}\Big (\frac{abM}{\sqrt{MN}}\frac { a^{1/4}}{N^{1/4}}+\frac{ab}{q_0}\Big ) \\
\ll & q^{1+\varepsilon}(ab)^{5/4}\big ((1+|\alpha|)(1+|\beta|)(1+|\gamma|)(1+|\delta|)\big)^{1/4}\Big (\frac{\sqrt{MN}}{N^{5/4}}+\frac{1}{q_0} \Big).
\end{split}
\end{align*}

   Note that similar expressions to \eqref{Splusminusplusintegralremovingd} and the above estimation hold for other  $S^{\pm;****}_{\pm, M,N}$ as well. We then deduce from these the following result. 
\begin{proposition}
\label{S0val} With the notation as above, we have
\begin{align*}
\begin{split}
 S^0_{+, M,N}+S^0_{-, M,N} =& \sum_{****}T^{\pm;****}_{\pm, M,N}  +O\Big(q^{1+\varepsilon}(ab)^{5/4}\big ((1+|\alpha|)(1+|\beta|)(1+|\gamma|)(1+|\delta|)\big)^{1/4}\Big (\frac{\sqrt{MN}}{N^{5/4}}+\frac{1}{q_0}\Big ) \Big ).
\end{split}
\end{align*}
\end{proposition}
  
\section{Evaluating $S_{+,1}+S_{-,1}$}

    We recall the definition of $S_{\pm, 1}$ from \eqref{S12def}. We deduce from the conditions $q^{2-2\eta_0} \leq MN  \leq \big ((1+|\alpha|)(1+|\beta|)(1+|\gamma|)(1+|\delta|)\big)^{1/2+\varepsilon}q^{2+\varepsilon}$ and $M \leq  N  \leq Mq^{1-2\eta_1}$ that
\begin{align*}
\begin{split}
 q^{1-\eta_0} \leq N \ll  \big ((1+|\alpha|)(1+|\beta|)(1+|\gamma|)(1+|\delta|)\big)^{1/4+\varepsilon}q^{3/2-\eta_1+\varepsilon}.
\end{split}
\end{align*}

   We then deduce from this, \eqref{SplusVS}, \eqref{Splusmain}, \eqref{SminusVS}, Proposition \ref{Eval} and Proposition \ref{S0val} that
\begin{align}
\label{S1sum}
S_{+, 1}+S_{-, 1} =  \sum_{\substack{M,N \ll \log q \\ q^{2-2\eta_0} \leq MN  \leq \big ((1+|\alpha|)(1+|\beta|)(1+|\gamma|)(1+|\delta|)\big)^{1/2+\varepsilon}q^{2+\varepsilon} \\ N  \ll Mq^{1-2\eta_1} } }\sum_{****}T^{\pm;****}_{\pm, M,N} +E 
\end{align}
where
\[ E \ll q^{\varepsilon} (1+|\alpha|)^4(1+|\beta|)^4(1+|\gamma|)^4 (1+|\delta|)^4 (ab)^{7}e^{\pi(|\Im(\beta-\alpha)|+|\Im
(\delta-\gamma)|)/2}  (q^{1+\theta-\eta_1}  + q^{7/8-\eta_1/4}+q^{3/4+5\eta_0/4}+q^{1-1/n_0}). \] 

   We now want to extend the sums of $T^{\pm;****}_{\pm, M,N}$ to all $M, N$. To do so, we need the following result from the proof of \cite[Lemma 6.2]{Young2011}. 
\begin{lemma}
\label{LemmaPartition}
 Let $F(s_1,s_2)$ be a holomorphic function in the strip $a<\Re e (s_i)<b$ with $a<0<b$ that decays rapidly to zero in each variable (in the imaginary direction). Then we have 
$$\sum_{M,N}\frac{1}{(2\pi i)^2}\int\limits_{(c_1)}\int\limits_{(c_2)}\tilde{v}_1(u)\tilde{v}_2(v)F(u,v) \dif u \dif v = F(0,0).$$
\end{lemma}

   We also need the following result concerns the size of the sums of $T^{\pm;****}_{\pm, M,N}$. 
\begin{lemma}
\label{LemmeAdding} 
 The first expression on the right-hand side of \eqref{Tsum1} is
\begin{align}
\label{Tsumbound}
\begin{split}
\ll & q^{1+\varepsilon}\min\left\{ (1+|\gamma|)^{1/2}\left(\frac{M}{N}\right)^{1/2} , \left( \frac{(1+|\alpha|)\cdot (1+|\beta|)\cdot (1+|\gamma|)\cdot (1+|\delta|)^{1/2}q^2}{MN}\right)^{C} , \frac{(MN)^{1/2}}{q}\right\},
\end{split}
\end{align}
where in the second estimation, the implied constant depends on $C$.
\end{lemma}
\begin{proof} 
 Our proof follows from that of \cite[Lemma 5.5]{Z2019}.
  We want to estimate the last expression in \eqref{Tsum1} by shifting contours of the triple integration there. To do so, we first want to bound the various factors appearing in the last expression in \eqref{Tsum}. Due to the rapid decay of the functions $\mathcal{G}(s), \tilde{v}_1(u), \tilde{v}_2(v)$ on the vertical line (here we note this is trivial for $\mathcal{G}(s)$ and easily verified for $\tilde{v}_1(u), \tilde{v}_2(v)$), we shall therefore only bound the related factors in terms of $\alpha, \beta, \gamma, \delta$, and this has been done in \eqref{zetabound}--\eqref{gammaratiobound}. \newline

 We now shift the $u$-contour to $\Re (u)=1/2-2\varepsilon+2\eta/\log q$ and the $v$-contour to $\Re (v)=-1/2+2\varepsilon$ in the last triple integration \eqref{Tsum}. In this process, we note that the conditions lead to the estimations given in \eqref{zetabound}--\eqref{gammaratiobound} and \eqref{Cproductest} are all satisfied. This way together with the estimation for $g_{\alpha, \beta, \gamma, \delta}(s)$ given in \eqref{gest}, we see that the first bound given in \eqref{Tsumbound}. \newline
  
For the second bound, we fix a constant $C>1$ and we first shift the contour to $\Re (s)=1/2+2\eta/\log q+\varepsilon$ and then to $\Re (u), \Re(v)=-1/2+\varepsilon/4^{2C}$. We next shift the contour to $\Re (s)=1-2\varepsilon/4^{2C}+2\eta/\log q+\varepsilon$ and then $\Re(u)=\Re(v)=-1+4\varepsilon/4^{2C}$. We then shift the contour to $\Re (s)=3/2-8\varepsilon/4^{2C}+2\eta/\log q+\varepsilon$ and then to $\Re (u)=\Re(v)=-3/2+16\varepsilon /4^{2C}$. It follows that after the $j^{th}$ step, we arrive at ($j\geqslant 2$)
$$\Re  (s)=\frac{j}{2}-\frac{4^{j-1}\varepsilon}{2\cdot4^{2C}}+\frac {2\eta}{\log q}+\varepsilon \hspace{0.4cm} \mathrm{and} \hspace{0.4cm} \Re (u)=\Re(v)=-\frac{j}{2}+\frac{4^{j-1}\varepsilon}{4^{2C}}.$$
Note that in the above processes, the conditions lead to the estimates in \eqref{zetabound}-- \eqref{Cproductest} are all satisfied. This now leads to the second bound given in \eqref{Tsumbound}. \newline

Lastly, the third  bound given in \eqref{Tsumbound} follows by shifting $\Re (u), \Re(v)$ to $1/2-\eta/\log q-2\varepsilon$ while applying the bounds in \eqref{gest}, \eqref{zetabound}--\eqref{Cproductest}. This completes the proof of the lemma. 
\end{proof}

  We now apply \eqref{S12def}, Lemma \ref{LemmaPartition} (and its analogue for the sum $T^{+;+-+-}_{-, M,N}+  T^{-;+-+-}_{-, M,N}$), Lemma \ref{LemmeAdding}, \eqref{Tsum1} and proceed as \cite[p. 76-77]{Liu24}.  This leads to
\begin{align}
\label{Tsumresult}
\begin{split}
& \sum_{\substack{M,N \ll \log q \\ q^{2-2\eta_0} \leq MN  \leq \big ((1+|\alpha|)(1+|\beta|)(1+|\gamma|)(1+|\delta|)\big)^{1/2+\varepsilon}q^{2+\varepsilon} \\ N  \ll Mq^{1-2\eta_1} } } (T^{+;-+-+}_{+, M,N}+  T^{-;-+-+}_{+, M,N}+T^{+;+-+-}_{-, M,N}+  T^{-;+-+-}_{-, M,N})\\
& = \sum_{M,N}(T^{+;-+-+}_{+, M,N}+  T^{-;-+-+}_{+, M,N}+T^{+;+-+-}_{-, M,N}+  T^{-;+-+-}_{-, M,N})\\
&\hspace*{2cm} +O(( \sqrt{1+|\alpha|}+\sqrt{1+|\beta|}+\sqrt{1+|\gamma|}+\sqrt{1+|\delta|})q^{1/2+\eta_1+\varepsilon}+q^{1-\eta_0}) \\
&= \phis(q)X_{\alpha,\gamma}\frac{C_{\boldsymbol{\alpha}, \boldsymbol{\beta},a,b}(0)}{a^{1/2 -\alpha} b^{1/2 -\gamma}} 
 \frac{\zeta(1-\alpha -\gamma) \zeta(1-\gamma +\delta) \zeta(1 -\alpha + \beta) \zeta(1+\beta + \delta)}{\zeta(2-\alpha + \beta -\gamma +\delta)} \\
& \hspace*{2cm}+O\Big( ( \sqrt{1+|\alpha|}+\sqrt{1+|\beta|}+\sqrt{1+|\gamma|}+\sqrt{1+|\delta|})q^{1/2+\eta_1+\varepsilon}+q^{1-\eta_0+\varepsilon}+\frac {q^{1+\varepsilon}(ab)^{-1/2}}{q_0} \Big).
\end{split}
\end{align}
 
  We deduce from \eqref{C0adef} that
\begin{align}
\label{C0exp}
\begin{split}
  C_{\boldsymbol{\alpha}, \boldsymbol{\beta},a}(0)= \prod_{p^{l_a}\|a} & \Big(1-\frac {1}{p^{2-\alpha+\beta-\gamma+\delta}} \Big)^{-1} \Big(1-\frac {1}{p^{\alpha+\delta}} \Big)^{-1} \\
& \times \Big(1-\frac {1}{p^{(1+l_a)(\alpha+\delta)}}-\frac {1}{p}\Big(p^{\gamma-\delta}+\frac {1}{p^{\beta+\delta}}\Big)\Big(1-\frac {1}{p^{l_a(\alpha+\delta)}}\Big)+\frac {1}{p^{2-\alpha+\beta-\gamma+\beta}}\Big(\frac {1}{p^{\alpha+\delta}}-\frac {1}{p^{l_a(\alpha+\delta)}}\Big)\Big),
\end{split}
\end{align}

  Note that
\begin{align}
\label{C0exp1}
\begin{split}
 & \Big(1-\frac {1}{p^{(1+l_a)(\alpha+\delta)}}-\frac {1}{p}\Big(p^{\gamma-\delta}+\frac {1}{p^{\beta+\delta}}\Big)\Big(1-\frac {1}{p^{l_a(\alpha+\delta)}}\Big)+\frac {1}{p^{2-\alpha+\beta-\gamma+\beta}}\Big(\frac {1}{p^{\alpha+\delta}}-\frac {1}{p^{l_a(\alpha+\delta)}}\Big)\Big) \\
=& \Big(1-\frac {1}{p^{2-\alpha+\beta-\gamma+\delta}}\Big)\Big(1-\frac {1}{p^{\alpha+\delta}}\Big)+\Big(1-\frac {1}{p^{l_a(\alpha+\delta)}}\Big)\Big(\frac {1}{p^{\alpha+\delta}}-\frac {1}{p}\Big(p^{\gamma-\delta}+\frac {1}{p^{\beta+\delta}}\Big)+\frac {1}{p^{2-\alpha+\beta-\gamma+\delta}}\Big) \\
=& \Big(1-\frac {1}{p^{2-\alpha+\beta-\gamma+\delta}}\Big)\Big(1-\frac {1}{p^{\alpha+\delta}}\Big)+\frac {1}{p^{\alpha+\delta}}\Big(1-\frac {1}{p^{l_a(\alpha+\delta)}}\Big)\Big(1-\frac {1}{p^{1-\alpha-\gamma}}\Big)\Big(1-\frac {1}{p^{1-\alpha+\beta}}\Big). 
\end{split}
\end{align}  

Now from \eqref{C0exp} and \eqref{C0exp1}, we infer
\begin{align}
\label{C0exptau}
\begin{split}
  C_{\boldsymbol{\alpha}, \boldsymbol{\beta},a}(0)=& \prod_{p^{l_a}\|a}\Big (1+\frac {p^{-\alpha-\delta}(p^{-(\alpha+\delta)l_a}-1)\zeta_p(2-\alpha+\beta-\gamma+\delta)}{(p^{-\alpha-\delta}-1)\zeta_p(1-\alpha-\gamma)\zeta_p(1-\alpha+\beta)} \Big )
=\tau_{-\gamma, \beta, -\alpha, \delta}(a).
\end{split}
\end{align}

  Similarly,
\begin{align}
\label{C0exptau1}
\begin{split}
  C_{\boldsymbol{\beta}, \boldsymbol{\alpha},b}(0)=& 
=\tau_{-\alpha, \delta, -\gamma, \beta}(b).
\end{split}
\end{align}

   We conclude from \eqref{Tsumresult}, \eqref{C0exptau} and \eqref{C0exptau1} that
\begin{align}
\label{Tsumexp}
\begin{split}
& \sum_{\substack{M,N \ll \log q \\ q^{2-2\eta_0} \leq MN  \leq \big ((1+|\alpha|)(1+|\beta|)(1+|\gamma|)(1+|\delta|)\big)^{1/2+\varepsilon}q^{2+\varepsilon} \\ N  \ll Mq^{1-2\eta_1} } }(T^{+;-+-+}_{+, M,N}+  T^{-;-+-+}_{+, M,N}+T^{+;+-+-}_{-, M,N}+  T^{-;+-+-}_{-, M,N}) \\
 = & 2S_2+O\Big( ( \sqrt{1+|\alpha|}+\sqrt{1+|\beta|}+\sqrt{1+|\gamma|}+\sqrt{1+|\delta|})q^{1/2+\eta_1+\varepsilon}+q^{1-\eta_0+\varepsilon}+\frac {q^{1+\varepsilon}(ab)^{-1/2}}{q_0}\Big),
\end{split}
\end{align}
  where $S_2$ is given in \eqref{Sdef}. Similarly, the sums of the other terms $T^{\pm;****}_{\pm, M,N}$ yield the sums $S_3+S_4+S_5$ on the right-hand side expression of \eqref{4thmomenteval} with an error term of the same size. Thus, we derive from \eqref{S1sum} and \eqref{Tsumexp} the following evaluation of $ S_{+, 1}+S_{-, 1}$.
\begin{proposition}
\label{Splusoneval} With the notation as above, we have
\begin{align*}
\begin{split}
   S_{+, 1}+S_{-, 1} 
 = &  2\sum^5_{j=2}S_j+O(q^{\varepsilon} (1+|\alpha|)^4(1+|\beta|)^4(1+|\gamma|)^4 (1+|\delta|)^4 (ab)^{7}e^{\pi(|\Im(\beta-\alpha)|+|\Im
(\delta-\gamma)|)/2} \\
& \hspace*{3cm} \times (q^{1+\theta-\eta_1}+q^{7/8-\eta_1/4}+q^{3/4+5\eta_0/4}+q^{1/2+\eta_1}+q^{1-1/n_0} )).
\end{split}
\end{align*}
\end{proposition}

\section{Proof of Theorem \ref{twisted_fourth_moment_theorem}}
\label{sec 10}

   We deduce from Proposition \ref{Seval}, Propositions \ref{Dsum}, \ref{S2eval} and \ref{Splusoneval}, recalling that $\alpha, \beta, \gamma, \delta \in \left\{z \in \bC: \Re(z) < \eta/\log q \right\}$, under the condition  in \eqref{abcond}, 
\begin{align}
\label{Sasymp}
\begin{split}
 \mathcal{S}(\alpha, \beta, \gamma, \delta; a,b) =& \sum^6_{j=1}S_j+O(q^{\varepsilon} (1+|\alpha|)^4(1+|\beta|)^4(1+|\gamma|)^4 (1+|\delta|)^4 (ab)^{7} e^{\pi(|\Im(\beta-\alpha)|+|\Im
(\delta-\gamma)|)/2} \mathcal{R}).
\end{split}
\end{align}
  where $S_j$ are given as in \eqref{Sdef} and
\begin{align}
\label{Rdef}
\begin{split}
 \mathcal{R} =q^{1+\theta-\eta_1}+q^{7/8-\eta_1/4}+q^{3/4+5\eta_0/4}+q^{1/2+\eta_1}+ q^{1-\eta_0}+q^{3/4+\eta_1/2}q^{1/2}_1+ q^{3/4+\eta_1}q^{1/4}_1+\frac {q^{1+\eta_0/2+\eta_1/2}}{q^{1/4}_1} +q^{1-1/n_0} .
\end{split}
\end{align}

  We now set $\eta_0=1/576, \eta_1=1/9$ so that $1+\theta-\eta_1=1-1/576=1-\eta_0$ as we recall that $\theta=7/64$. We may also assume that $q_0$ is large enough so that the condition in \eqref{abcond} is valid by the above choice of $\eta_1$ and \eqref{scondition}. We now set $q_1=q^{i_0}_0=q^{i_0/n_0}$ such that $q_1 \leq q^{1/4}$ but $q_1q_0>q^{1/4}$. It follows that $q^{3/4+\eta_1/2}q^{1/2}_1+ q^{3/4+\eta_1}q^{1/4}_1 \ll q^{1-1/576}$ and that $\frac {q^{1+\eta_0/2+\eta_1/2}}{q^{1/4}_1} \leq q^{1+1/1152+1/18-1/16+1/(4n_0)} \leq q^{1-1/576}$ as $n_0 \geq 50$. We then conclude from \eqref{Sasymp}, \eqref{Rdef} and the condition that $\Im (\beta-\alpha)=O(1), \Im(\delta-\gamma)=O(1)$ that the asymptotic formula given in \eqref{4thmomenteval} is valid. This completes the proof of Theorem \ref{twisted_fourth_moment_theorem}.

\vspace*{.5cm}

\noindent{\bf Acknowledgments.} P. G. is supported in part by NSFC grant 12471003 and L. Z. by the FRG Grant PS71536 at the University of New South Wales.

\bibliography{biblio}
\bibliographystyle{amsxport}

\vspace*{.5cm}

\end{document}